\documentclass[reqno,12pt]{amsart}
\usepackage{eurosym}
\usepackage{amssymb}
\usepackage{mathrsfs}
\usepackage{txfonts}
\usepackage{amsfonts}
\usepackage{amstext}
\usepackage{amssymb}
\usepackage{amsmath}
\usepackage{graphicx,subfig}
\usepackage{hyperref}
\usepackage{url}
\usepackage{amssymb}
\usepackage{enumitem}
\usepackage{tikz}

\hypersetup{
	colorlinks   = true,
	citecolor    = blue,
	linkcolor    = blue,
	urlcolor     = blue
}
\allowdisplaybreaks
\newtheorem{theorem}{Theorem}[section]
\newtheorem{proposition}[theorem]{Proposition}
\newtheorem{lemma}[theorem]{Lemma}
\newtheorem{definition}[theorem]{Definition}
\newtheorem{corollary}[theorem]{Corollary}
\newtheorem{remark}[theorem]{Remark}
\newtheorem{example}[theorem]{Example}
\renewcommand{\theequation}{\thesection.\arabic{equation}}

\newenvironment{acknowledgement}{\smallskip{\sc Acknowledgement.}\rm}{\smallskip}

\numberwithin{equation}{section}
\newcounter{counterConstant}

\input{tcilatex}

\def\func#1{\mathop{\mathrm{#1}}\nolimits}

\def\tbigcup{\mathop{\textstyle \bigcup }}
\def\tbigcap{\mathop{\textstyle \bigcap }}

\def\Xint#1{\mathchoice
{\XXint\displaystyle\textstyle{#1}}%
{\XXint\textstyle\scriptstyle{#1}}%
{\XXint\scriptstyle\scriptscriptstyle{#1}}%
{\XXint\scriptscriptstyle\scriptscriptstyle{#1}}%
\!\int}
\def\XXint#1#2#3{{\setbox0=\hbox{$#1{#2#3}{\int}$ }
\vcenter{\hbox{$#2#3$ }}\kern-.6\wd0}}

\def\oint{\Xint-}

\def\FRAME#1#2#3#4#5#6#7#8
{
 \begin{figure}[ptbh]
 \begin{center}
 \includegraphics[height=#3]{#7}
 \caption{#5}
 \label{#6}
 \end{center}
 \end{figure}
}

\def\enddoc{\end{document}}

\begin{document}
\title[]{Heat kernel-based $p$-energy norms on metric measure spaces}
\author[Gao]{Jin Gao}
\address{ School of Mathematics, Hangzhou Normal University, Hangzhou
310036, China.}
\email{gaojin@hznu.edu.cn}
\author[Yu]{Zhenyu Yu}
\address{College of Science, National University
of Defense Technology, Changsha 410073, China.}
\email{yuzy23@nudt.edu.cn}
\author[Zhang]{Junda Zhang}
\address{School of Mathematics, South China University of Technology,
Guangzhou 510641, China.}
\email{summerfish@scut.edu.cn}
\date{}

\begin{abstract}
We investigate heat kernel-based and other $p$-energy norms ($1<p<\infty$)
on bounded and unbounded metric measure spaces, in particular, on nested
fractals and their blow-ups. With the weak-monotonicity properties for these
semi-norms, we generalize the celebrated Bourgain-Brezis-Mironescu (BBM) type
characterization for $p\neq2$. When the underlying space admits a heat kernel satisfying
the sub-Gaussian estimates, we establish the equivalence of various $p$-energy
semi-norms and weak-monotonicity properties, and show that these
weak-monotonicity properties hold when $p=2$ (that is the case of Dirichlet
form). Our paper's key results concern the equivalence and verification of
various weak-monotonicity properties on fractals. Consequently, many
classical results on $p$-energy norms hold on nested fractals and their
blow-ups, including the BBM type characterization and Gagliardo-Nirenberg
inequality.
\end{abstract}

\subjclass[2010]{28A80, 46E30, 46E35}
\keywords{Heat kernel, weak-monotonicity property, $p$-energy, Besov norm,
nested fractal.}
\maketitle
\tableofcontents

\section{Introduction}

\subsection{Historical background and motivation}

Recently, there has been considerable development on the $p$-energy defined
on fractals and general metric measure spaces (initiated by \cite%
{HermanPeironeStrichartz.2004.PA125}). For a smooth Euclidean domain $D$,
 its associated $p$-energy is defined as $\int_{D}|\nabla u(x)|^{p}dx$, but for
many fractals or metric measure spaces, it is not easy to define proper
gradient structures to characterize the smoothness of certain `core'
functions arising naturally from analysis. Previous constructions of $p$%
-energy ($1<p<\infty $) are based on the graph-approximation to the
underlying space, including p.c.f. self-similar sets by Cao, Gu and Qiu \cite%
{Caoqiugu2022adv}, the Sierpi\'{n}ski carpet by Shimizu \cite{Shimuzu.2022}
and more general fractal spaces by Kigami \cite{Kigami2022penergy}.

In this paper, we propose an alternative way to define $p$-energy norms via
equivalent Besov-type norms, which naturally generalize the energy of
Dirichlet forms when $p=2$, and do not make use of the graph-approximation for their
the definition. Heat kernel-based $p$-energy norms were introduced by
K.~Pietruska-Pa{\l }uba in \cite{Pietruska-Paluba.2010.}, and heat
semigroup-based norms were later studied by Alonso Ruiz, Baudoin et al. for $1\leq
p<\infty $ in \cite%
{AlonsoBaudoinchen2020JFA,AlonsoBaudoinchen2020CVPDE,AlonsoBaudoinchen2021CVPDE}%
. There, the authors focus on generalizing classical analysis results including
Sobolev embedding, isoperimetric inequalities and bounded
variation functions. It is possible to use heat kernel-based $p$-energy
norms to construct equivalent (or exactly the same) $p$-energy given by the
graph-approximation approach, to satisfy further restrictions like convexity
or self-similarity for self-similar sets, see for example \cite%
{Caoqiugu2022adv,GaoYuZhang2022PA}.

Recall the celebrated `Bourgain-Brezis-Mironescu (BBM) convergence' in \cite[%
Corollary 2]{BourgainBrezisMironescu.2001.439}:
\begin{equation}
\lim_{\sigma \uparrow 1}(1-\sigma )\int_{D}\int_{D}\frac{|u(x)-u(y)|^{p}}{%
|x-y|^{n+p\sigma }}dxdy=C_{n,p}\int_{D}|\nabla u(x)|^{p}dx \ \ \
(1<p<\infty),  \label{BBM}
\end{equation}%
where $D$ is a smooth domain in $\mathbb{R}^{n}$. It states that
multiplying by a scaling factor $1-\sigma $, the fractional Gagliardo
semi-norms of a function converge to the first-order Sobolev semi-norm as $%
\sigma \rightarrow 1$. The BBM type characterization thus corresponds to the convergence of
Besov semi-norms $(\sigma ^{\ast }-\sigma )B_{p,p}^{\sigma }$ to $%
B_{p,\infty }^{\sigma ^{\ast }}$ on metric measure spaces, where $\sigma
^{\ast }$ is a certain critical exponent. This characterization has been studied for instance on fractals with `property (E)' \cite[Theorem 1.6]{GaoYuZhang2022PA}%
, and on spaces supporting a $p$-Poincar\'e inequality \cite{DS19,Munnier15,Gorny22.JGA}.%

In this paper, we establish a BBM type characterization for both bounded and unbounded metric
measure spaces under appropriate weak-monotonicity properties that replace
property (E) on fractals from \cite{GaoYuZhang2022PA}. We utilize the concept of property $(KE)_{\sigma}$
introduced in \cite[Definition 6.7]{AlonsoBaudoinchen2020JFA} and $(NE)_{\sigma}$
defined in \cite[Definition 4.5]{BaudoinLecture2022} (all termed $P(p,\alpha
)$ therein), and use $(VE)_{\sigma}$ for bounded and unbounded fractals based on \cite[%
Definition 3.1]{GaoYuZhang2022PA}, where the letters `E' stands for energy
control and `K',`N',`V' represent (heat) kernel, (Besov) norm, vertex
respectively. Additionally, we explore their slightly weaker variants by
replacing `sup' with `limsup'.

These weak-monotonicity properties correspond to certain energy-control conditions,
which are essentially required and important in all the studies related to $p$%
-energy (norms) mentioned above, to guarantee `some level of $L^p$
infinitesimal regularity and global controlled $L^p$ geometry' , as stated in \cite%
{BaudoinLecture2022}. In analysis on fractals, weak-monotonicity appears
in the form of a $p$-resistance estimate in \cite%
{Caoqiugu2022adv,Shimuzu.2022,Kigami2022penergy}, which is not easy to
obtain. Verifying the weak-monotonicity properties $(KE)_{\sigma}$, $(NE)_{\sigma}$ and $(VE)_{\sigma}$ for certain fractal-type spaces, like nested
fractals and their blow-ups, is one of the main contributions of this paper.

To do so, we establish the equivalence of $(KE)_{\sigma}$ and $(NE)_{\sigma}$ on metric measure
spaces, the equivalence of $(NE)_{\sigma}$ and $(VE)_{\sigma}$ on certain fractals (including
nested fractals), and then verify a relatively easier condition $(VE)_{\sigma}$ using
the arguments in a main theorem of \cite{Caoqiugu2022adv}. Such properties
are not easy to examine when $p\neq 2$ in general (but property $(KE)_{\sigma}$ holds
automatically when $p=2$). Alonso Ruiz, Baudoin et al. examined property
$(KE)_{\sigma}$ in the case $p=1$ using a weak Bakry-Emery type estimate and in the case $p\neq 2$
for spaces with the same critical exponent as Euclidean smooth manifolds
(see for example \cite[Lemma 4.13]{AlonsoBaudoinchen2021CVPDE}). But for
fractal-type spaces, different ideas are required, since the critical
exponents $\sigma^*$ from (\ref{BBM}) are different from manifolds ($\sigma^*=1$ for
smooth manifolds; but even when $p=2$, $\sigma^*=\log15/\log9 $ on the
Vicsek set \cite{BC23} and $\sigma^*=\log5/\log4$ on the Sierpi\'nski gasket
\cite{BP88}) and (bi-)linearity is missing.

\subsection{Preliminaries}

\label{secPre}

\subsubsection{Basic definitions}

Let $(M,d,\mu )$ be a metric measure space, that is, $(M,d)$\ is a locally
compact, separable metric space and $\mu $ is a Radon measure with full
support on $M$. Fix a value (which could be infinite)
\begin{equation*}
R_{0}\in (0,\mathrm{diam}(M)]
\end{equation*}
that will be used for localization throughout the paper, where
\begin{equation*}
\mathrm{diam}(M):=\sup \{d(x,y):x,y\in M\}\in (0,\infty]
\end{equation*}
is infinite if $M$ is unbounded and is finite if $M$ is bounded.

The measure $\mu$ is assumed to be Borel regular with the following
positivity and \emph{volume doubling property}: there exists a constant $%
C_{d}>0$ such that for every $x\in M$, $0<r<\infty $,
\begin{equation}
0<\mu (B(x,2r))\leq C_{d}\mu (B(x,r))<\infty,  \label{VD1}
\end{equation}%
where $B(x,r):=\{y\in M:d(x,y)<r\}$ denotes the open metric ball centered at
$x\in M$ with radius $r>0$. Denote
\begin{equation*}
V(x,r):=\mu (B(x,r)).
\end{equation*}
It is known from \cite[Proposition 5.1]{GrigoryanHu.2014.MMJ505} that if %
\eqref{VD1} holds, then there exists $\alpha _{1}>0$ such that
\begin{equation}
\frac{V(x,R)}{V(x,r)}\leq C_{d}\left( \frac{R}{r}\right) ^{\alpha _{1}},
\label{VD2}
\end{equation}%
for all $x\in M$, $0<r\leq R<\infty$.

For $\alpha>0$, we say that $\mu $ is \emph{$\alpha $-regular} if there
exists a constant $C\geq 1$ such that for all $0<r<R_{0}$,
\begin{equation}
C^{-1}r^{\alpha }\leq V(x,r)\leq Cr^{\alpha }.  \label{alpha_r}
\end{equation}%
We remark that when $\mu $ is $\alpha $-regular, then (\ref{VD2}) holds with
$\alpha _{1}=\alpha $.

A family $\{p_{t}\}_{t>0}$ of non-negative measurable functions on $M \times
M$ is a \emph{heat kernel} if it satisfies: for all $x, y\in M$ and $s, t>0$:
\begin{enumerate}[label=\textup{(\arabic*)}]
\item  Symmetry: $p_t(x,y) =p_t(y,x)$.
\item Markov property: $\int_M p_t(x,y)d\mu(y) \leq 1$.
\item Semigroup property: $p_{s+t}(x,y)=\int_{M}p_{s}(x,z)p_{t}(z,y)d\mu (z)$.
\item Identity approximation: for any $f\in L^{2}(M,\mu )$, $\int_{M}p_{t}(x,y)f(y)d\mu (y)\rightarrow
f(x)$\quad as$~~t\downarrow 0$, strongly in$~~L^{2}(M,\mu )$.
\end{enumerate}

For some parameter $\beta ^{\ast }>1$, we consider the following two
conditions on the heat kernel:

We say that $\{p_{t}\}_{t>0}$ satisfies the lower heat kernel estimate
(LHE), if for all $t\in (0,R_{0}^{\beta ^{\ast }}) $ and $\mu $-almost all $%
x,y\in M$,
\begin{equation}
p_{t}(x,y)\geq \frac{c_{1}}{V(x,t^{1/\beta ^{\ast }})}\exp \left(
-c_{2}\left( \frac{d(x,y)}{t^{1/\beta ^{\ast }}}\right) ^{\frac{\beta ^{\ast
}}{\beta ^{\ast }-1}}\right).  \label{LHE}
\end{equation}

We say that $\{p_{t}\}_{t>0}$ satisfies the upper heat kernel estimate
(UHE), if for all $t\in (0,R_{0}^{\beta ^{\ast }}) $ and $\mu $-almost all $%
x,y\in M$,
\begin{equation}
p_{t}(x,y)\leq \frac{c_{3}}{V(x,t^{1/\beta ^{\ast }})}\exp \left(
-c_{4}\left( \frac{d(x,y)}{t^{1/\beta ^{\ast }}}\right) ^{\frac{\beta ^{\ast
}}{\beta ^{\ast }-1}}\right).  \label{UHE}
\end{equation}

Conditions (UHE) and (LHE) hold in many cases. The classical
Gauss-Weierstrass function
\begin{equation}
p_{t}(x,y)=\frac{1}{(4\pi t)^{n/2}}\exp \left( -\frac{|x-y|^{2}}{4t}\right) ,
\label{Gau-Wei}
\end{equation}%
is a heat kernel for the standard Brownian motion in $\mathbb{R}^{n}$. Li
and Yau have shown in \cite{LiYau.1986.AM153} that, for a complete
Riemannian manifold with non-negative Ricci curvature, two-sided estimates
(UHE) and (LHE) hold with $\beta ^{\ast }=2$:
\begin{equation}
p_{t}(x,y)\asymp \frac{C}{V(x,\sqrt{t})}\exp \left( -\frac{d(x,y)^{2}}{ct}%
\right) ,  \label{LiYau}
\end{equation}%
where $d$ is the geodesic metric and $V(x,r)$ is the Riemannian volume. The
following sub-Gaussian heat kernel estimates
\begin{equation}
p_{t}(x,y)\asymp \frac{C}{t^{\alpha /\beta ^{\ast }}}\exp \left( -c\left(
\frac{d(x,y)}{t^{1/\beta ^{\ast }}}\right) ^{\frac{\beta ^{\ast }}{\beta
^{\ast }-1}}\right)  \label{hk_F}
\end{equation}%
also hold on many fractal spaces with different $\beta ^{\ast }$, such as
the Sierpi\'{n}ski carpet and nested fractals \cite%
{BarlowBass.1992.PTRF307,BarlowBass.1999.CJM673,FitzsimmonsmHamblyKumagai.1994.CMP595,Kumagai.1993.PTaRF205}%
.

\subsubsection{Besov and Korevaar-Schoen norms}

From now on we fix $p\in (1,\infty)$. As in \cite{BaudoinLecture2022}, for $%
\sigma >0$, define the semi-norm $[u]_{B_{p,\infty }^{\sigma }}$ for $u\in
L^{p}(M,\mu )$ by
\begin{equation}
\lbrack u]_{B_{p,\infty }^{\sigma }}^{p}:=\sup_{r\in (0,R_{0})}r^{-p\sigma
}\int_{M}\frac{1}{V(x,r)}\int_{B(x,r)}|u(x)-u(y)|^{p}d\mu (y)d\mu (x),
\label{pi}
\end{equation}%
and define 
\begin{equation*}
B_{p,\infty }^{\sigma }:=B_{p,\infty }^{\sigma }(M)=\{u:||u||_{B_{p,\infty }^{\sigma }}<\infty \},
\end{equation*}
with the norm $||u||_{B_{p,\infty }^{\sigma }}:=||u||_p+[u]_{B_{p,\infty }^{\sigma }}$.
Further, define the semi-norm $[u]_{B_{p,p}^{\sigma }}$ by
\begin{equation}
\lbrack u]_{B_{p,p}^{\sigma }}^{p}:=\int_{0}^{R_{0}}r^{-p\sigma }\left(
\int_{M}\frac{1}{V(x,r)}\int_{B(x,r)}|u(x)-u(y)|^{p}d\mu (y)d\mu (x)\right)
\frac{dr}{r},  \label{pp}
\end{equation}%
and
\begin{equation*}
B_{p,p}^{\sigma }:=B_{p,p}^{\sigma }(M)=\{u:||u||_{B_{p,p }^{\sigma }}<\infty \},
\end{equation*}
with the norm $||u||_{B_{p,p }^{\sigma }}:=||u||_p+[u]_{B_{p,p }^{\sigma }}$.

The space $B_{p,\infty }^{\sigma }$ coincides with the Korevaar-Schoen space
$KS_{p,\infty }^{\sigma }$ in \cite[Section 4.2]{BaudoinLecture2022} where the Korevaar-Schoen semi-norm $[u] _{KS_{p,\infty }^{\sigma }}$ is given by
\begin{equation}
[u]^p _{KS_{p,\infty }^{\sigma }}=\limsup_{r\rightarrow 0}\int_{M}%
\frac{1}{V(x,r)}\int_{B(x,r )}\frac{|u(x)-u(y)|^{p}}{r ^{p\sigma }}d\mu
(y)d\mu (x)<\infty.  \label{ks_energy1}
\end{equation}

In our paper, we define the \emph{critical exponent of $(M,d,\mu )$} by%
\begin{equation*}
\sigma _{p}^{\#}=\sup \{\sigma >0:B_{p,\infty }^{\sigma }\text{ contains
non-constant functions}\}.
\end{equation*}%
In many related studies, the critical exponent is defined by
\begin{equation*}
\sigma _{p}^{\ast }=\sup \{\sigma >0:B_{p,\infty }^{\sigma }\text{ is dense
in }L^{p}(M,\mu )\}.
\end{equation*}

Clearly, $\sigma _{p}^{\ast }\leq \sigma _{p}^{\#}$. Whenever the critical
exponents appear in our paper, we assume that they are finite. It is known
by \cite[Theorem 4.2]{BaudoinLecture2022} that when the chain condition
holds,
\begin{equation*}
\sigma _{p}^{\#}\leq 1+\frac{\alpha_1}{p}<\infty .
\end{equation*}

\subsubsection{Heat kernel-based $p$-energy norm}

It is well-known that, a heat kernel $\{p_{t}\}_{t>0}$ can induce a
Dirichlet form on $L^{2}(M,\mu )$ by
\begin{align*}
\mathcal{E}(u,u)& =\lim_{t\rightarrow 0^{+}}\mathcal{E}_{t}(u,u),\ u\in
L^{2}(M,\mu) \\
\mathcal{D}(\mathcal{E})& =\{u\in L^{2}(M,\mu ):\mathcal{E}(u,u)<\infty \}
\end{align*}%
(see for example \cite[Section 1.3]{FukushimaOshimaTakeda.2011.489}), where
\begin{equation}
\mathcal{E}_{t}(u,u):=\frac{1}{2t}\int_{M}\int_{M}|u(x)-u(y)|^{2}p_{t}(x,y)d%
\mu (y)d\mu (x).  \label{E_t}
\end{equation}%
The limit exists, since for any given $u\in L^{2}(M,\mu )$, $\mathcal{E}%
_{t}(u,u)$ decreases in $t$ by using the spectral theorem.

Following \cite
{AlonsoBaudoinchen2020JFA,Pietruska-Paluba.2010.}, given a heat kernel $\{p_{t}\}_{t>0}$, we denote its $
p $-energy semi-norm on $L^{p}(M,\mu )$ by
$\left(E_{p,\infty}^\sigma(\cdot)\right)^{1/p}$, where
\begin{equation*}
E_{p,\infty }^{\sigma }(u):=\sup_{t\in (0,R_{0}^{\beta ^{\ast }})}\frac{1}{%
t^{{p\sigma }/{\beta ^{\ast }}}}\int_{M}\int_{M}|u(x)-u(y)|^{p}p_{t}(x,y)d%
\mu (y)d\mu (x),
\end{equation*}
with its domain
\begin{equation*}
\mathcal{D}(E_{p,\infty }^{\sigma }):=\{u\in L^{p}(M,\mu ):E_{p,\infty
}^{\sigma }(u)<\infty \}.
\end{equation*}
We call $\left(E_{p,\infty}^\sigma(\cdot)\right)^{1/p}+||\cdot||_p$ the heat kernel-based $
p $-energy norm. This is a natural way to extend the Dirichlet form case for $p=2$ to $p\in
(1,\infty)$. $E_{p,\infty }^{\sigma_{p}^{\#}}$ is
equivalent to the $p$-energy on certain fractals (see Section \ref{sec5}).

To extend the BBM type characterization from $p=2$ to $1<p<\infty$, we define
\begin{equation*}
E_{p,p}^{\sigma }(u):=\int_{0}^{R_{0}^{\beta ^{\ast }}}\frac{1}{t^{{p\sigma }%
/{\beta ^{\ast }}}}\left( \int_{M}\int_{M}|u(x)-u(y)|^{p}p_{t}(x,y)d\mu
(x)d\mu (y)\right) \frac{dt}{t},
\end{equation*}%
with domain
\begin{equation*}
\mathcal{D}(E_{p,p}^{\sigma }):=\{u\in L^{p}(M,\mu ):E_{p,p}^{\sigma
}(u)<\infty \}.
\end{equation*}
It is known that (see for example \cite{Pietruska-Paluba.2008}), when the
two-sided heat kernel estimates hold, the proper scaling of $E_{2,2}^{\sigma
} $ converges to $E_{2,\infty }^{\sigma _{2}^{\#}}$ thanks to the fact that $%
E_{2,2}^{\sigma }$ is equivalent to the Dirichlet form defined by the
subordinated heat kernel. In this paper, we establish an analogous convergence of $E_{p,p}^{\sigma }$ to $E_{p,\infty }^{\sigma
_{p}^{\#}}$. However, $\left(E_{p,p}^{\sigma }\right)^{1/p}$ does not seem to be the $p $-energy semi-norm of the subordinated heat kernel, so we need weak-monotonicity
properties to replace the subordination technique in the case $p=2$ (see Theorem \ref%
{thm4.1}).

\subsection{Weak-monotonicity properties and main results}

We introduce some notions that will be used for weak-monotonicity
properties and rewrite the above notions in the following way. For $u\in
L^{p}(M,\mu )$ and $\sigma ,t>0$, denote
\begin{equation}
\Psi _{u}^{\sigma }(t)=\frac{1}{t^{{p\sigma }/{\beta ^{\ast }}}}%
\int_{M}\int_{M}|u(x)-u(y)|^{p}p_{t}(x,y)d\mu (y)d\mu (x),  \label{psi_u}
\end{equation}%
then
\begin{equation}
E_{p,\infty }^{\sigma }(u)=\sup_{t\in (0,R_{0}^{\beta ^{\ast }})}\Psi
_{u}^{\sigma }(t),\ E_{p,p}^{\sigma }(u)=\int_{0}^{R_{0}^{\beta ^{\ast
}}}\Psi _{u}^{\sigma }(t)\frac{dt}{t}.  \label{EEE}
\end{equation}

Similarly, denote
\begin{equation}
\Phi _{u}^{\sigma }(r)=r^{-p\sigma }\int_{M}\frac{1}{V(x,r)}%
\int_{B(x,r)}|u(x)-u(y)|^{p}d\mu (y)d\mu (x),  \label{phi_u}
\end{equation}%
then
\begin{equation*}
\lbrack u]_{B_{p,\infty }^{\sigma }}^{p}=\sup_{r\in (0,R_{0})}\Phi
_{u}^{\sigma }(r),\ [u]_{B_{p,p}^{\sigma }}^{p}=\int_{0}^{R_{0}}\Phi
_{u}^{\sigma }(r)\frac{dr}{r},\ [u] _{KS_{p,\infty }^{\sigma
}}^{p}=\limsup_{r\rightarrow 0}\Phi _{u}^{\sigma }(r).
\end{equation*}

\begin{definition}
We say that a metric measure space $(M,d,\mu )$ with heat kernel $%
\{p_{t}\}_{t>0}$ satisfies property $(KE)_{\sigma}$ with $\sigma>0$, if
there exists a constant $C>0$ such that for all $u\in \mathcal{D}%
(E_{p,\infty }^{\sigma })$,
\begin{equation}  \label{KE}
\sup_{t\in (0,R_{0}^{\beta ^{\ast }})}\Psi _{u}^{\sigma }(t)\leq
C\liminf_{t\rightarrow 0}\Psi _{u}^{\sigma }(t),  \tag*{$(KE)_{\sigma}$}
\end{equation}
where $\Psi _{u}^{\sigma }(t)$ is defined as in \eqref{psi_u}. We further
say that property $(\widetilde{KE})_{\sigma}$ is satisfied with $\sigma >0$, if there
exists a constant $C>0$ such that for all $u\in \mathcal{D}(E_{p,\infty
}^{\sigma }) $,
\begin{equation}  \label{wKE}
\limsup_{t\rightarrow 0}\Psi _{u}^{\sigma }(t)\leq C\liminf_{t\rightarrow
0}\Psi _{u}^{\sigma }(t),  \tag*{$(\widetilde{KE})_{\sigma}$}
\end{equation}
\end{definition}

\begin{remark}
\label{rk1}When $p=2$, property $(KE)_{\beta
^{\ast }/2}$ automatically holds as $\mathcal{E}_{t}(u,u)$ decreases in $t$ for any $u\in \mathcal{D}(E_{2,\infty}^{\beta^{\ast }/2 })$.
\end{remark}

\begin{definition}
We say that a metric measure space $(M,d,\mu )$ satisfies property $(NE)_{\sigma}$
with $\sigma >0$ if there exists $C>0$ such that for all $%
u\in B_{p,\infty }^{\sigma }$,
\begin{equation}  \label{NE1}
\sup_{r\in (0,R_{0})}\Phi _{u}^{\sigma }(r)\leq C\liminf_{r\rightarrow
0}\Phi _{u}^{\sigma }(r).  \tag*{$(NE)_{\sigma}$}
\end{equation}%
We further say that property $(\widetilde{NE})_{\sigma}$ is satisfied with $%
\sigma >0$, if there exists $C>0$ such that for all $u\in B_{p,\infty
}^{\sigma }$,
\begin{equation}  \label{wNE}
\limsup_{r\rightarrow 0}\Phi _{u}^{\sigma }(r)\leq C\liminf_{r\rightarrow
0}\Phi _{u}^{\sigma }(r).  \tag*{$(\widetilde{NE})_{\sigma}$}
\end{equation}
\end{definition}

\begin{remark}
It is known from \cite[Lemma 4.7]{BaudoinLecture2022} that, property
(NE)$_\sigma$ can only hold when $\sigma \geq \sigma _{p}^{\#}$. So property $(NE)_{\sigma}$ is only interested when $\sigma=\sigma _{p}^{\#},$ and
whenever we say that property $(NE)_{\sigma}$ is satisfied, we automatically assume $%
\sigma=\sigma _{p}^{\#}$. Recently, property $(NE)_{\sigma}$ was verified on nested
fractals by Chang and the authors \cite[Theorem 2.3]{CGYZ24} with the help
of a lemma in this paper, on Sierpi\'nski carpets for $p>\dim_{ARC}(K, d)$
by Yang \cite[Theorem 2.8]{Yang23} and for all $p>1$ by Murugan and Shimizu
\cite[Theorem 1.4]{MuruganShimizu23}. Kajino and Shimizu \cite[Section 5]%
{KS24} studied property $(NE)_{\sigma}$ under the `$p$-contraction property' based on
Kigami's partition of metric spaces \cite{Kigami2022penergy}.
\end{remark}

Our first result is that, two kinds of BBM type characterization under $(KE)_{\sigma}$
and $(NE)_{\sigma}$ are established in Theorems \ref{thm4.1} and \ref{thm4.2},
without using heat kernel estimates.

Our second result is that, under the heat kernel estimates (UHE) and (LHE), we
show the equivalence of the heat kernel-based $p$-energy semi-norms and
Besov-Korevaar-Schoen semi-norms on bounded and unbounded metric measure spaces
in Lemma \ref{thm1}, and the following weak-monotonicity equivalence.

\begin{theorem}
\label{thm4}If a metric measure space $(M,d,\mu )$ admits a heat kernel $%
\{p_{t}\}_{t>0}$ satisfying (UHE) and (LHE), then we have the following
equivalences for $\sigma>0$:

\begin{itemize}
\item[(i)~] $(\widetilde{KE})_{\sigma}\Longleftrightarrow (\widetilde{NE})_{\sigma}$,

\item[(ii)] $(KE)_{\sigma}\Longleftrightarrow (NE)_{\sigma}$.
\end{itemize}
\end{theorem}

Now we introduce our results for fractals. In what follows, we assume that $K$ is a
connected homogeneous p.c.f. self-similar set. Let $\alpha =\func{dim}%
_{H}(K) $ and denote the $\alpha$-dimensional Hausdorff measure on $K$ by $\mu $.
Let $K^{F}:=\bigcup_{f\in F}f(K)$ be a fractal glue-up where $F$ is a proper
set of functions (see details in Subsection \ref{subsec5.1}) and equip it
with a natural glue-up measure $\mu ^{F}$ as in (\ref{muF}). This notion is
used to unify bounded and unbounded fractals. Following \cite{GaoYuZhang2022PA}, let
\begin{equation}
E_{n}^{(p)}(u):=\sum_{x,y\in V_{w},|w|=n}|u(x)-u(y)|^{p}\text{ \ \ and\ \ \ }%
\mathcal{E}_{n}^{\sigma }(u):=\rho ^{-n(p\sigma -\alpha )}E_{n}^{(p)}(u).
\label{EE}
\end{equation}%
Similarly, let
\begin{equation}
E_{n}^{(p),F}(u):=\sum_{f\in F}E_{n}^{(p)}(u\circ f)\text{ \ and\ }\mathcal{E}_{n}^{\sigma ,F}(u):=\rho ^{-n(p\sigma -\alpha
)}E_{n}^{(p),F}(u).  \label{p_energy}
\end{equation}%

\begin{definition}
\label{ve} We say that a fractal glue-up $K^{F}$ satisfies property $(
\normalfont{VE})_{\sigma}$ with $\sigma >0$, if there exists a constant $C>0$ such
that for all $u \in L^p(K^{F},\mu^F)$,
\begin{equation}  \label{VE}
\sup_{n\geq 0}\mathcal{E}_{n}^{\sigma ,F}(u)\leq C\liminf_{n\rightarrow
\infty }\mathcal{E}_{n}^{\sigma ,F}(u).  \tag*{$(VE)_{\sigma}$}
\end{equation}
We say that a fractal glue-up $K^{F}$ satisfies property $(\widetilde{VE})_{\sigma}$
with $\sigma >0$, if there exists $C>0$ such that for all $u \in
L^p(K^{F},\mu^F)$,
\begin{equation}  \label{wVE}
\limsup_{n\rightarrow \infty }\mathcal{E}_{n}^{\sigma ,F}(u)\leq
C\liminf_{n\rightarrow \infty }\mathcal{E}_{n}^{\sigma ,F}(u).
\tag*{$(\widetilde{VE})_{\sigma}$}
\end{equation}
\end{definition}

The proof of the following equivalence is more delicate. It involves
proper truncation, an improvement of the $p$-energies equivalence in \cite%
{GaoYuZhang2022PA} with annulus decomposition, and comparing the vertex-energies of
different levels. Also, analysis on unbounded fractals requires more
arguments than the bounded cases.

\begin{theorem}
\label{thm5}Let $K^{F}$ be a fractal glue-up, where $K$ is a connected
homogeneous p.c.f. self-similar set, then we have the following equivalences:

\begin{itemize}
\item[(i)~] $(\widetilde{VE})_{\sigma} \Longleftrightarrow (\widetilde{NE})_{\sigma}$ when $\sigma>\alpha /p$;

\item[(ii)] $(VE)_{\sigma
_{p}^{\#}} \Longleftrightarrow (NE)_{\sigma
_{p}^{\#}}$ when $\sigma
_{p}^{\#}>\alpha /p$ and $R_0=C_H$ (a positive number determined by $K^{F}$).
\end{itemize}
\end{theorem}


\begin{remark}
The equivalence of the integral-type properties $(KE)_{\sigma}$, $(NE)_{\sigma}$ and of the discrete-type property $(VE)_{\sigma}$
does not rely on the p.c.f. property, but requires the continuity of the
functions in $B_{p,\infty}^{\sigma}$. The reason why we consider p.c.f.
fractals is to avoid complexity in defining $(VE)_{\sigma}$ properly. Also
verifying $(VE)_{\sigma}$ is relatively easier under the p.c.f. property (see
Proposition \ref{prop:critical} and Lemma \ref{yyqx}).
\end{remark}

Theorem \ref{thm4} and Theorem \ref{thm5} are very important in our paper
with many corollaries. In Section \ref{subsec5.4}, we will apply them to
obtain Corollary \ref{corol1.8}, and further deduce Theorem \ref{corol1.9}
focusing on nested fractals and their blow-ups. Moreover, they also deduce
Corollary \ref{jjjj} and Corollary \ref{jjj} that are important to Dirichlet
forms when $p=2$.

\subsection{Structure of the paper}

The organization of the paper is as follows. In Section \ref{sec4}, we
establish a BBM type characterization under appropriate weak-monotonicity
properties (Theorems \ref{thm4.1} and \ref{thm4.2}). In Section \ref%
{sec2}, we demonstrate the equivalence between heat kernel-based and
Besov-Korevaar-Schoen $p$-energy semi-norms under two-sided estimates (UHE) and (LHE) (Lemma \ref%
{thm1}), and the equivalence of the corresponding weak-monotonicity
properties (Theorem \ref{thm4}). In Section \ref{sec5}, we study fractal
glue-ups $K^F$. In Subsection \ref{subsec5.1}, we introduce some
geometric properties of $K^F$ and present some useful lemmas. In Subsection %
\ref{subsec5.2}, we verify $(VE)_{\sigma}$ for nested fractal glue-ups. In Subsection %
\ref{subsec5.3}, we prove Theorem \ref{thm5}. In Subsection \ref{subsec5.4},
we show that the BBM type characterization and Gagliardo-Nirenberg
inequality hold true for nested fractal blow-ups (Theorem \ref{corol1.9}).

\textbf{Notation}: The letters $C$, $C^{\prime }$, $C_{i}$,$C_{i}^{\prime }$%
, $C_{i}^{\prime\prime }$, $c$ are universal positive constants depending
only on $M$ which may vary at each occurrence. The sign $\asymp $ means that
both $\leq $ and $\geq $ are true with uniform values of $C$ depending only
on $M$. 

\section{Consequence of $(KE)_{\protect\sigma}$ and $(NE)_{\protect\sigma}$: BBM type characterization}

\label{sec4} In this section, we present our BBM type characterization for
heat kernel-based and Besov-Korevaar-Schoen $p$-energy semi-norms. Under the
settings in Section \ref{sec2}, these two BBM type characterization results
will coincide.

\subsection{BBM type characterization of heat kernel-based $p$-energy semi-norms}

In our proof, we need the following embedding relation.

\begin{proposition}
\label{Prop32}For any $\delta \in (0,\sigma )$, we have
\begin{equation}
\mathcal{D}(E_{p,\infty }^{\sigma })\subseteq \mathcal{D}(E_{p,p}^{\sigma
-\delta }).  \label{306}
\end{equation}
\end{proposition}

\begin{proof}
By the elementary inequality $|a-b|^{p}\leq 2^{p-1}(|a|^{p}+|b|^{p})$, we have
\begin{align}
& \int_{M}\int_{M}|u(x)-u(y)|^{p}p_{t}(x,y)d\mu (x)d\mu (y)  \notag \\
\leq & 2^{p-1}\int_{M}\int_{M}(|u(x)|^{p}+|u(y)|^{p})p_{t}(x,y)d\mu (x)d\mu
(y)  \notag \\
=& 2^{p}\int_{M}|u(x)|^{p}\left( \int_{M}p_{t}(x,y)d\mu (y)\right) d\mu
(x)\leq 2^{p}||u||_{p}^p.  \label{304}
\end{align}

Fix $\epsilon \in (0,R_{0}^{\beta ^{\ast }})$. By \eqref{304},%
\begin{eqnarray*}
E_{p,p}^{\sigma -\delta }(u) &=&\int_{0}^{R_{0}^{\beta ^{\ast }}}\frac{1}{t^{%
{p(\sigma -\delta )}/{\beta ^{\ast }}}}\left(
\int_{M}\int_{M}|u(x)-u(y)|^{p}p_{t}(x,y)d\mu (x)d\mu (y)\right) \frac{dt}{t}
\\
&\leq &\int_{0}^{\epsilon }\frac{1}{t^{{p(\sigma -\delta )}/{\beta ^{\ast }}}%
}\left( \int_{M}\int_{M}|u(x)-u(y)|^{p}p_{t}(x,y)d\mu (x)d\mu (y)\right)
\frac{dt}{t} \\
&&+\int_{\epsilon }^{\infty }\frac{1}{t^{{p(\sigma -\delta )}/{\beta ^{\ast }%
}}}\left( \int_{M}\int_{M}|u(x)-u(y)|^{p}p_{t}(x,y)d\mu (x)d\mu (y)\right)
\frac{dt}{t} \\
&\leq &\int_{0}^{\epsilon }t^{-1+p\delta /{\beta ^{\ast }}}dt\sup_{t\in
(0,R_{0}^{\beta ^{\ast }})}\Psi _{u}^{\sigma }(t)+2^{p}||u||_{p}^{p}\int_{\epsilon }^{\infty }\frac{dt}{t^{1+{p(\sigma -\delta )}/{\beta
^{\ast }}}} \\
&=&\frac{{\beta ^{\ast }}\epsilon ^{p\delta /{\beta ^{\ast }}}}{p\delta }%
\sup_{t\in (0,R_{0}^{\beta ^{\ast }})}\Psi _{u}^{\sigma }(t)+2^{p}\Vert
u\Vert _{p}^{p}\frac{{\beta ^{\ast }}\epsilon ^{-p(\sigma -\delta )/{\beta
^{\ast }}}}{p(\sigma -\delta )} \\
&=&\frac{{\beta ^{\ast }}\epsilon ^{p\delta /{\beta ^{\ast }}}}{p\delta }%
E_{p,\infty }^{\sigma }(u)+\frac{2^{p}{\beta ^{\ast }}||u||_{p}^p}{%
p(\sigma -\delta )}\epsilon ^{-p(\sigma -\delta )/{\beta ^{\ast }}}<\infty ,
\end{eqnarray*}%
thus showing \eqref{306}.
\end{proof}

\begin{theorem}
\label{thm4.1} Suppose that $(M,d,\mu )$ admits heat kernel $%
\{p_{t}\}_{t>0}$ satisfies property $(KE)_{\tilde{\sigma}_{p}}$ with some $\tilde{\sigma}_{p}>0$. Then there exists
a constant $C\geq 1$ such that for all $u\in \mathcal{D}(E_{p,\infty }^{%
\tilde{\sigma}_{p}})$,
\begin{equation}
C^{-1}{E}_{p,\infty }^{\tilde{\sigma}_{p}}(u)\leq \liminf_{\sigma \uparrow
\tilde{\sigma}_{p}}(\tilde{\sigma}_{p}-\sigma ){E}_{p,p}^{\sigma }(u)\leq
\limsup_{\sigma \uparrow \tilde{\sigma}_{p}}(\tilde{\sigma}_{p}-\sigma ){E}%
_{p,p}^{\sigma }(u)\leq C{E}_{p,\infty }^{\tilde{\sigma}_{p}}(u).
\label{eq5.1}
\end{equation}
\end{theorem}

\begin{proof}
Let $u\in \mathcal{D}(E_{p,\infty }^{\tilde{\sigma _{p}}})$ and $\sigma \in
(0,\tilde{\sigma}_{p})$. By Proposition \ref{Prop32}, $u\in \mathcal{D}%
(E_{p,p}^{\sigma })$. Note that $$\Psi _{u}^{\sigma }(t)=t^{p(\tilde{\sigma}%
_{p}-\sigma )/{\beta ^{\ast }}}\Psi _{u}^{\tilde{\sigma}_{p}}(t).$$

For the left-hand side, by \eqref{EEE}, we have
\begin{align}
\liminf_{\sigma \uparrow \tilde{\sigma}_{p}}(\tilde{\sigma}_{p}-\sigma ){E}%
_{p,p}^{\sigma }(u)=& \liminf_{\sigma \uparrow \tilde{\sigma}_{p}}(\tilde{%
\sigma}_{p}-\sigma )\int_{0}^{R_{0}^{{\beta ^{\ast }}}}t^{p(\tilde{\sigma}%
_{p}-\sigma )/{\beta ^{\ast }}}\Psi _{u}^{\tilde{\sigma}_{p}}(t)\frac{dt}{t}
\notag \\
\geq & \liminf_{\sigma \uparrow \tilde{\sigma}_{p}}(\tilde{\sigma}%
_{p}-\sigma )\int_{0}^{\tilde{\sigma}_{p}-\sigma }t^{p(\tilde{\sigma}%
_{p}-\sigma )/{\beta ^{\ast }}}\Psi _{u}^{\tilde{\sigma}_{p}}(t)\frac{dt}{t}
\notag \\
\geq & \liminf_{\sigma \uparrow \tilde{\sigma}_{p}}(\tilde{\sigma}%
_{p}-\sigma )\int_{0}^{\tilde{\sigma}_{p}-\sigma }t^{p(\tilde{\sigma}%
_{p}-\sigma )/{\beta ^{\ast }}}\frac{dt}{t}\inf_{t\in (0,\tilde{\sigma}%
_{p}-\sigma )}\Psi _{u}^{\tilde{\sigma}_{p}}(t)  \notag \\
=& \frac{{\beta ^{\ast }}}{p}\liminf_{\sigma \uparrow \tilde{\sigma}_{p}}(%
\tilde{\sigma}_{p}-\sigma )^{p(\tilde{\sigma}_{p}-\sigma )/{\beta ^{\ast }}%
}\inf_{t\in (0,\tilde{\sigma}_{p}-\sigma )}\Psi _{u}^{\tilde{\sigma}_{p}}(t)
\notag \\
=& \frac{{\beta ^{\ast }}}{p}\liminf_{\sigma \uparrow \tilde{\sigma}%
_{p}}\inf_{t\in (0,\tilde{\sigma}_{p}-\sigma )}\Psi _{u}^{\tilde{\sigma}%
_{p}}(t)=\frac{{\beta ^{\ast }}}{p}\liminf_{t\rightarrow 0}\Psi _{u}^{\tilde{%
\sigma}_{p}}(t).  \notag
\end{align}%
It follows from property $(KE)_{\tilde{\sigma}_{p}}$ that
\begin{equation*}
\liminf_{\sigma \uparrow \tilde{\sigma}_{p}}(\tilde{\sigma}_{p}-\sigma ){E}%
_{p,p}^{\sigma }(u)\geq \frac{C{\beta ^{\ast }}}{p}E_{p,\infty }^{\tilde{%
\sigma}_{p}}(u).
\end{equation*}

For the right-hand side, let $A\in (0,R_{0})$ be a finite positive number.
By (\ref{304}), for any $\sigma \in (0,\tilde{\sigma _{p}})$,
\begin{eqnarray*}
{E}_{p,p}^{\sigma }(u) &=&\int_{0}^{R_{0}^{\beta ^{\ast }}}\frac{1}{t^{{%
p\sigma }/{\beta ^{\ast }}}}\left(
\int_{M}\int_{M}|u(x)-u(y)|^{p}p_{t}(x,y)d\mu (x)d\mu (y)\right) \frac{dt}{t}
\\
&\leq &\int_{0}^{A^{\beta ^{\ast }}}\frac{1}{t^{{p\sigma }/{\beta ^{\ast }}}}%
\left( \int_{M}\int_{M}|u(x)-u(y)|^{p}p_{t}(x,y)d\mu (x)d\mu (y)\right)
\frac{dt}{t}+\int_{A^{\beta ^{\ast }}}^{\infty }\frac{2^{p}||u||
_{p}^{p}}{t^{{p\sigma }/{\beta ^{\ast }}}}\frac{dt}{t} \\
&=&\int_{0}^{A^{\beta ^{\ast }}}t^{p(\tilde{\sigma}_{p}-\sigma )/{\beta
^{\ast }}}\Psi _{u}^{\tilde{\sigma}_{p}}(t)\frac{dt}{t}+\frac{2^{p}{\beta
^{\ast }}||u||_{p}^p}{p\sigma A^{{p\sigma }}}.
\end{eqnarray*}%
It follows that
\begin{align}
\limsup_{\sigma \uparrow \tilde{\sigma}_{p}}(\tilde{\sigma}_{p}-\sigma ){E}%
_{p,p}^{\sigma }(u)& \leq \limsup_{\sigma \uparrow \tilde{\sigma}_{p}}(%
\tilde{\sigma}_{p}-\sigma )\int_{0}^{A^{\beta ^{\ast }}}t^{p(\tilde{\sigma}%
_{p}-\sigma )/{\beta ^{\ast }}}\Psi _{u}^{\tilde{\sigma}_{p}}(t)\frac{dt}{t}%
+\limsup_{\sigma \uparrow \tilde{\sigma}_{p}}(\tilde{\sigma}_{p}-\sigma )%
\frac{2^{p}{\beta ^{\ast }}||u||_{p}^p}{p\sigma A^{{p\sigma }}}
\notag \\
& \leq \limsup_{\sigma \uparrow \tilde{\sigma}_{p}}(\tilde{\sigma}%
_{p}-\sigma )\int_{0}^{A^{\beta ^{\ast }}}t^{p(\tilde{\sigma}_{p}-\sigma )/{%
\beta ^{\ast }}}\frac{dt}{t}\sup_{t\in (0,A^{{\beta ^{\ast }}})}\Psi _{u}^{%
\tilde{\sigma}_{p}}(t)+0  \notag \\
& =\frac{{\beta ^{\ast }}}{p}\limsup_{\sigma \uparrow \tilde{\sigma}%
_{p}}A^{p(\tilde{\sigma}_{p}-\sigma )}\sup_{t\in (0,A^{{\beta ^{\ast }}%
})}\Psi _{u}^{\tilde{\sigma}_{p}}(t)  \notag \\
& =\frac{{\beta ^{\ast }}}{p}\sup_{t\in (0,A^{{\beta ^{\ast }}})}\Psi _{u}^{%
\tilde{\sigma}_{p}}(t)\leq \frac{{\beta ^{\ast }}}{p}E_{p,\infty }^{\tilde{%
\sigma}_{p}}(u).  \label{NE_right}
\end{align}%
The proof is complete.
\end{proof}

\subsection{BBM type characterization of Besov-Korevaar-Schoen $p$-energy
semi-norms}

The proof of this version is similar to the previous one. 

\begin{proposition}
\label{Prop31}For any $\delta \in (0,\sigma )$, we have
\begin{equation}
B_{p,\infty }^{\sigma }\subseteq B_{p,p}^{\sigma -\delta }.  \label{301}
\end{equation}
\end{proposition}

\begin{proof}
By the elementary inequality $|a-b|^{p}\leq 2^{p-1}(|a|^{p}+|b|^{p})$, we
have
\begin{align}
& \int_{M}\frac{1}{V(x,r)}\int_{B(x,r)}|u(x)-u(y)|^{p}d\mu (y)d\mu (x)
\notag \\
\leq & 2^{p-1}\int_{M}\frac{1}{V(x,r)}\int_{B(x,r)}(|u(x)|^{p}+|u(y)|^{p})d%
\mu (y)d\mu (x)  \notag \\
=& 2^{p-1}\left( ||u||_{p}^p+\int_{M}\int_{M}\frac{\mathbf{1}%
_{\{d(y,x)<r\}}}{V(x,r)}|u(y)|^{p}d\mu (y)d\mu (x)\right)  \notag \\
=& 2^{p-1}\left( ||u||_{p}^p+\int_{M}\left( \int_{B(y,r)}\frac{1}{%
V(x,r)}d\mu (x)\right) |u(y)|^{p}d\mu (y)\right)  \notag \\
\leq & 2^{p-1}\left( ||u||_{p}^p+\sup_{x\in M}\frac{V(x,2r)}{V(x,r)%
}\int_{M}|u(y)|^{p}d\mu (y)\right) \leq C||u||_{p}^p.  \label{300}
\end{align}

Fix $\epsilon \in (0,R_{0})$. By \eqref{300},%
\begin{eqnarray*}
\lbrack u]_{B_{p,p}^{\sigma -\delta }}^{p} &=&\int_{0}^{R_{0}}r^{-(p\sigma
-p\delta )}\int_{M}\frac{1}{V(x,r)}\int_{B(x,r)}|u(x)-u(y)|^{p}d\mu (y)d\mu
(x)\frac{dr}{r} \\
&\leq &\int_{0}^{\epsilon }r^{-(p\sigma -p\delta )}\int_{M}\frac{1}{V(x,r)}%
\int_{B(x,r)}|u(x)-u(y)|^{p}d\mu (y)d\mu (x)\frac{dr}{r} \\
&&+\int_{\epsilon }^{\infty }r^{-(p\sigma -p\delta )}\int_{M}\frac{1}{V(x,r)}%
\int_{B(x,r)}|u(x)-u(y)|^{p}d\mu (y)d\mu (x)\frac{dr}{r} \\
&\leq &\int_{0}^{\epsilon }r^{-1+p\delta }dr\sup_{r\in (0,R_{0})}\Phi
_{u}^{\sigma }(r)+C||u||_{p}^p\int_{\epsilon }^{\infty
}r^{-(p\sigma -p\delta )}\frac{dr}{r} \\
&=&\frac{\epsilon ^{p\delta }}{p\delta }\sup_{r\in (0,R_{0})}\Phi
_{u}^{\sigma }(r)+C||u||_{p}^p\frac{\epsilon ^{-p(\sigma -\delta )}%
}{p(\sigma -\delta )}=\frac{\epsilon ^{p\delta }}{p\delta }[u]_{B_{p,\infty
}^{\sigma }}^{p}+C^{\prime }||u||_{p}^p<\infty ,
\end{eqnarray*}%
thus showing \eqref{301}.
\end{proof}

\begin{theorem}
\label{thm4.2} Suppose that $(M,d,\mu )$ satisfies property $(NE)_{\sigma _{p}^{\#}}$, then there
exists a positive constant $C$ such that for all $u\in B_{p,\infty }^{\sigma
_{p}^{\#}}$,
\begin{equation}
C^{-1}[u]_{{B}_{p,\infty }^{\sigma _{p}^{\#}}}^{p}\leq \liminf_{\sigma
\uparrow \sigma _{p}^{\#}}(\sigma _{p}^{\#}-\sigma )[u]_{{B}_{p,p}^{\sigma
}}^{p}\leq \limsup_{\sigma \uparrow \sigma _{p}^{\#}}(\sigma
_{p}^{\#}-\sigma )[u]_{{B}_{p,p}^{\sigma }}^{p}\leq C[u]_{{B}_{p,\infty
}^{\sigma _{p}^{\#}}}^{p}.  \label{NE_conv}
\end{equation}%
Suppose that $(M,d,\mu )$ satisfies property $(\widetilde{NE})_{\sigma_p}$ with some $\sigma_{p}>0$, then there exists a positive constant $C$ such that for all $u\in
B_{p,\infty }^{\sigma _{p}}$,
\begin{equation}
C^{-1}[u] _{{KS}_{p,\infty }^{\sigma _{p}}}^{p}\leq \liminf_{\sigma
\uparrow \sigma _{p}}(\sigma _{p}-\sigma )[u]_{{B}_{p,p}^{\sigma }}^{p}\leq
\limsup_{\sigma \uparrow \sigma _{p}}(\sigma _{p}-\sigma )[u]_{{B}%
_{p,p}^{\sigma }}^{p}\leq C[u] _{{KS}_{p,\infty }^{\sigma
_{p}}}^{p}.  \label{KS_conv}
\end{equation}
\end{theorem}

\begin{proof}
We show \eqref{NE_conv} first. For the left-hand side, similarly, we have
that for any $\sigma _{p}>0,$
\begin{align}
\liminf_{\sigma \uparrow \sigma _{p}}(\sigma _{p}-\sigma )[u]_{{B}%
_{p,p}^{\sigma }}^{p}=& \liminf_{\sigma \uparrow \sigma _{p}}(\sigma
_{p}-\sigma )\int_{0}^{R_{0}}r^{p(\sigma _{p}-\sigma )}\Phi _{u}^{\sigma
_{p}}(r)\frac{dr}{r}  \notag \\
\geq & \liminf_{\sigma \uparrow \sigma _{p}}(\sigma _{p}-\sigma
)\int_{0}^{\sigma _{p}-\sigma }r^{p(\sigma _{p}-\sigma )}\Phi _{u}^{\sigma
_{p}}(r)\frac{dr}{r}  \notag \\
\geq & \liminf_{\sigma \uparrow \sigma _{p}}(\sigma _{p}-\sigma
)\int_{0}^{\sigma _{p}-\sigma }r^{p(\sigma _{p}-\sigma )}\frac{dr}{r}%
\inf_{r\in (0,\sigma _{p}-\sigma )}\Phi _{u}^{\sigma _{p}}(r)  \notag \\
=& \frac{1}{p}\liminf_{\sigma \uparrow \sigma _{p}}\inf_{r\in (0,\sigma
_{p}-\sigma )}\Phi _{u}^{\sigma _{p}}(r)=\frac{1}{p}\liminf_{r\rightarrow
0}\Phi _{u}^{\sigma _{p}}(r).  \label{408}
\end{align}%
Therefore, by property $(NE)_{\sigma _{p}^{\#}}$, we have for $\sigma _{p}=\sigma _{p}^{\#}$ in \eqref{408} that
\begin{equation}
\liminf_{\sigma \uparrow \sigma _{p}^{\#}}(\sigma _{p}^{\#}-\sigma )[u]_{{B}%
_{p,p}^{\sigma }}^{p}\geq \frac{C^{-1}}{p}[u]_{{B}_{p,\infty }^{\sigma
_{p}^{\#}}}^{p}.  \label{NE_left}
\end{equation}

For the other side, let $A\in (0,R_{0})$ be a finite positive number. We
have from \eqref{300} that
\begin{eqnarray}
\lbrack u]_{{B}_{p,p}^{\sigma }}^{p} &=&\int_{0}^{R_{0}}r^{-p\sigma }\left(
\int_{M}\frac{1}{V(x,r)}\int_{B(x,r)}|u(x)-u(y)|^{p}d\mu (y)d\mu (x)\right)
\frac{dr}{r}  \notag \\
&\leq &\int_{0}^{A}r^{-p\sigma }\left( \int_{M}\frac{1}{V(x,r)}%
\int_{B(x,r)}|u(x)-u(y)|^{p}d\mu (y)d\mu (x)\right) \frac{dr}{r}%
+C\int_{A}^{\infty }r^{-p\sigma }||u||_{p}^p\frac{dr}{r}  \notag \\
&=&\int_{0}^{A}r^{p(\sigma _{p}^{\#}-\sigma )}\Phi _{u}^{\sigma _{p}^{\#}}(r)%
\frac{dr}{r}+\frac{C||u||_{p}^p}{p\sigma A^{p\sigma }}.
\label{300-1}
\end{eqnarray}%
It follows that
\begin{align}
\limsup_{\sigma \uparrow \sigma _{p}^{\#}}(\sigma _{p}^{\#}-\sigma )[u]_{{B}%
_{p,p}^{\sigma }}^{p}& \leq \limsup_{\sigma \uparrow \sigma
_{p}^{\#}}(\sigma _{p}^{\#}-\sigma )\int_{0}^{A}r^{p(\sigma _{p}^{\#}-\sigma
)}\Phi _{u}^{\sigma _{p}^{\#}}(r)\frac{dr}{r}+\limsup_{\sigma \uparrow
\sigma _{p}^{\#}}(\sigma _{p}^{\#}-\sigma )\frac{C||u||_{p}^p}{%
p\sigma A^{p\sigma }}  \notag \\
& \leq \limsup_{\sigma \uparrow \sigma _{p}^{\#}}(\sigma _{p}^{\#}-\sigma
)\int_{0}^{A}r^{p(\sigma _{p}^{\#}-\sigma )}\frac{dr}{r}\sup_{r\in
(0,R_{0})}\Phi _{u}^{\sigma _{p}^{\#}}(r)+0  \notag \\
& =\frac{1}{p}\limsup_{\sigma \uparrow \sigma _{p}^{\#}}A^{p(\sigma
_{p}^{\#}-\sigma )}\sup_{r\in (0,R_{0})}\Phi _{u}^{\sigma _{p}^{\#}}(r)=%
\frac{1}{p}[u]_{B_{p,\infty }^{\sigma _{p}^{\#}}}^{p},  \label{42}
\end{align}%
thus showing \eqref{NE_conv}.

Next, we show \eqref{KS_conv}. By property $(\widetilde{NE})_{\sigma_p}$ and \eqref{408}, we have
\begin{equation}
\liminf_{\sigma \uparrow \sigma _{p}}(\sigma _{p}-\sigma )[u]_{{B}%
_{p,p}^{\sigma }}^{p}\geq C^{-1}[u] _{KS_{p,\infty}^{\sigma _{p}}}^{p}.
\label{KS_left}
\end{equation}%
Then for the other side of \eqref{KS_conv}, by the definition of $\limsup $,
there exists $\epsilon \in (0,R_{0})$ such that
\begin{equation*}
\sup_{r\in (0,\epsilon )}\Phi _{u}^{\sigma _{p}}(r)\leq
2\limsup_{r\rightarrow 0}\Phi _{u}^{\sigma _{p}}(r).
\end{equation*}%
It follows from replacing $\sigma _{p}^{\#}$ by $\sigma _{p}$ and taking $%
A=\epsilon $ in (\ref{300-1}) that,%
\begin{align}
\limsup_{\sigma \uparrow \sigma _{p}}(\sigma _{p}-\sigma )[u]_{{B}%
_{p,p}^{\sigma }}^{p}& \leq \limsup_{\sigma \uparrow \sigma _{p}}(\sigma
_{p}-\sigma )\int_{0}^{\epsilon }r^{p(\sigma _{p}-\sigma )}\Phi _{u}^{\sigma
_{p}}(r)\frac{dr}{r}+\limsup_{\sigma \uparrow \sigma _{p}}(\sigma
_{p}-\sigma )\frac{C|| u||_{p}^p}{p\sigma \epsilon ^{p\sigma }}
\notag \\
& \leq \limsup_{\sigma \uparrow \sigma _{p}}(\sigma _{p}-\sigma
)\int_{0}^{\epsilon }r^{p(\sigma _{p}-\sigma )}\frac{dr}{r}\sup_{r\in
(0,\epsilon )}\Phi _{u}^{\sigma _{p}}(r)+0  \notag \\
& =\frac{1}{p}\sup_{r\in (0,\epsilon )}\Phi _{u}^{\sigma _{p}}(r)\leq \frac{2%
}{p}\limsup_{r\rightarrow 0}\Phi _{u}^{\sigma _{p}}(r)=\frac{2}{p}[
u] _{{KS}_{p,\infty }^{\sigma _{p}}}^{p}.  \label{KS_right}
\end{align}%
The proof is complete.
\end{proof}

\section{Equivalences on metric measure spaces}

\label{sec2} In this section, we assume that the underlying metric measure space $(M,d,\mu )$ admits
a heat kernel satisfying (UHE) and (LHE).


\subsection{Equivalence of integral-type semi-norms}

In this subsection, fix $\sigma>0$. The main result of this subsection is
the following.

\begin{lemma}
\label{thm1}If $(M,d,\mu )$ admits a heat kernel $%
\{p_{t}\}_{t>0}$ satisfying (UHE) and (LHE), then $\mathcal{D}(E_{p,\infty
}^{\sigma })=B_{p,\infty }^{\sigma }$ and for all $u\in B_{p,\infty
}^{\sigma }$,
\begin{equation}
E_{p,\infty }^{\sigma }(u)\asymp \lbrack u]_{B_{p,\infty }^{\sigma }}^{p}%
\text{\ and\ }[u] _{KS_{p,\infty }^{\sigma }}^{p}\asymp
\limsup_{t\rightarrow 0}\Psi _{u}^{\sigma }(t).  \label{eq1.7}
\end{equation}%
In addition, $\mathcal{D}(E_{p,p}^{\sigma })=B_{p,p}^{\sigma }$ and for all $%
u\in B_{p,p}^{\sigma }$,
\begin{equation}
E_{p,p}^{\sigma }(u)\asymp \lbrack u]_{B_{p,p}^{\sigma }}^{p}.  \label{eq1.8}
\end{equation}
\end{lemma}

We prove Lemma \ref{thm1} by pure elementary analysis (without probability
arguments) under slightly milder assumptions, and split the proof into
Propositions \ref{lem3.1}-\ref{corol3.5}. Part of Lemma \ref{thm1} was also
studied by K.~Pietruska-Pa{\l }uba in \cite[Theorems 3.1 and 3.2]%
{Pietruska-Paluba.2010.} using probability arguments, and by Alonso Ruiz,
Baudoin et al. in \cite[Proposition 4.2]{AlonsoBaudoinchen2020CVPDE} under Gaussian
heat kernel estimate. Recently, Cao and Qiu in \cite{Caoqiu23} proved
the equivalence of the Korevaar-Schoen norm and another type of Besov norm.

\begin{proposition}
\label{lem3.1} If $(M,d,\mu )$ admits a heat kernel $\{p_{t}\}_{t>0}$
satisfying (LHE), then there exists a positive constant $C$ such that for
all $u\in \mathcal{D}(E_{p,p}^{\sigma })$,
\begin{equation}
\lbrack u]_{B_{p,p}^{\sigma }}^{p}\leq CE_{p,p}^{\sigma }(u).  \label{eq3.1}
\end{equation}
\end{proposition}

\begin{proof}
By (LHE), we have
\begin{align}
& \frac{1}{t^{{p\sigma }/{\beta ^{\ast }}}}\int_{M}\int_{B(x, t^{{1}/{\beta
^{\ast }}})}|u(x)-u(y)|^{p}p_{t}(x,y)d\mu (y)d\mu (x)  \notag \\
\geq& \frac{c_{1}}{t^{p\sigma /\beta ^{\ast }}}\int_{M}\frac{1}{%
V(x,t^{1/\beta ^{\ast }})}\int_{B(x, t^{{1}/{\beta ^{\ast }}%
})}|u(x)-u(y)|^{p} \exp \left(-c_2\left(\frac{d(x,y)}{t^{1/\beta^*}}\right)^{%
\frac{\beta^*}{\beta^*-1}}\right)d\mu (y)d\mu (x)  \notag \\
\geq & c_{1}e^{-c_2} \Phi_u^{\sigma}( t^{1/\beta^*}) .  \label{eq3-2}
\end{align}
Therefore,
\begin{align*}
E_{p,p}^{\sigma }(u)&=\int_{0}^{R_{0}^{\beta ^{\ast }}}\frac{1}{t^{{p\sigma }%
/{\beta ^{\ast }}}}\left( \int_{M}\int_{M}|u(x)-u(y)|^{p}p_{t}(x,y)d\mu
(x)d\mu (y)\right) \frac{dt}{t} \\
& \geq \beta ^{\ast }c_{1}e^{-c_2}\int_{0}^{R_{0}} \Phi_u^{\sigma}(r)\frac{dr%
}{r} =\beta ^{\ast}c_{1}e^{-c_2}\lbrack u]_{B_{p,p}^{\sigma }}^{p},
\end{align*}
which completes the proof.
\end{proof}


\begin{proposition}
\label{lem3.2} If $(M,d,\mu )$ admits a heat kernel $\{p_{t}\}_{t>0}$
satisfying (UHE), then there exists a positive constant $C$ such that for
all $u\in B_{p,p}^{\sigma }$,
\begin{equation*}
E_{p,p}^{\sigma }(u)\leq C[u]_{B_{p,p}^{\sigma }}^{p}.
\end{equation*}
\end{proposition}

\begin{proof}
We decompose $B(x,t^{1/\beta ^{\ast }})^{c}$ as the union of annuli $
B(x,2^{n}t^{1/\beta ^{\ast }})\setminus B(x,2^{n-1}t^{1/\beta ^{\ast }})$
for $x\in M$, where the integers $1\leq n\leq \log _{2}(R_{0}t^{-1/\beta
^{\ast }})$, then by (\ref{UHE}),
\begin{align}
& \frac{1}{t^{{p\sigma }/{\beta ^{\ast }}}}\int_{M}\int_{B(x,t^{1/\beta
^{\ast }})^{c}}|u(x)-u(y)|^{p}p_{t}(x,y)d\mu (y)d\mu (x)  \notag \\
& =\frac{1}{t^{{p\sigma }/{\beta ^{\ast }}}}\int_{M}\sum_{n=1}^{[\log
_{2}(R_{0}t^{-1/\beta ^{\ast }})]}\int_{B(x,2^{n}t^{1/\beta ^{\ast
}})\setminus B(x,2^{n-1}t^{1/\beta ^{\ast }})}|u(x)-u(y)|^{p}p_{t}(x,y)d\mu
(y)d\mu (x)  \notag \\
& \leq \frac{1}{t^{{p\sigma }/{\beta ^{\ast }}}}\int_{M}\frac{c_{3}}{%
V(x,t^{1/\beta ^{\ast }})}\sum_{n=1}^{[\log _{2}(R_{0}t^{-1/\beta ^{\ast
}})]}\int_{B(x,2^{n}t^{1/\beta ^{\ast }})\setminus B(x,2^{n-1}t^{1/\beta
^{\ast }})}|u(x)-u(y)|^{p}\exp \left( -c_{4}2^{\frac{(n-1)\beta ^{\ast }}{%
\beta ^{\ast }-1}}\right) d\mu (y)d\mu (x)  \notag \\
& \leq \sum_{n=1}^{[\log _{2}(R_{0}t^{-1/\beta ^{\ast }})]}\frac{1}{t^{{%
p\sigma }/{\beta ^{\ast }}}}\int_{M}\frac{c_{3}}{V(x,t^{1/\beta ^{\ast }})}%
\int_{B(x,2^{n}t^{1/\beta ^{\ast }})}|u(x)-u(y)|^{p}\exp \left( -c_{4}2^{%
\frac{(n-1)\beta ^{\ast }}{\beta ^{\ast }-1}}\right) d\mu (y)d\mu (x)  \notag
\\
& \leq \sum_{n=1}^{[\log _{2}(R_{0}t^{-1/\beta ^{\ast }})]}\exp \left(
-c_{4}2^{\frac{(n-1)\beta ^{\ast }}{\beta ^{\ast }-1}}\right) \frac{%
2^{np\sigma }}{(2^{n}t^{1/\beta ^{\ast }})^{p\sigma }}\int_{M}\frac{%
c_{3}C_{d}2^{n\alpha _{1}}}{V(x,2^{n}t^{1/\beta ^{\ast }})}%
\int_{B(x,2^{n}t^{1/\beta ^{\ast }})}|u(x)-u(y)|^{p}d\mu (y)d\mu (x)  \notag
\\
& =c_{3}C_{d}\sum_{n=1}^{\infty }1_{\{n\leq \log _{2}(R_{0}t^{-1/\beta
^{\ast }})\}}\exp \left( -c_{4}2^{\frac{(n-1)\beta ^{\ast }}{\beta ^{\ast }-1%
}}\right) 2^{(p\sigma +\alpha _{1})n}\Phi _{u}^{\sigma }(2^{n}t^{1/\beta
^{\ast }}),  \label{eq3.3-2}
\end{align}%
where we use the volume doubling property \eqref{VD2} in the fifth line
\begin{equation}
\frac{V(x,2^{n}t^{1/\beta ^{\ast }})}{V(x,t^{1/\beta ^{\ast }})}\leq
C_{d}2^{n\alpha _{1}}.  \label{VD}
\end{equation}

Using (\ref{UHE}) again, we have%
\begin{eqnarray}
&&\frac{1}{t^{{p\sigma }/{\beta ^{\ast }}}}\int_{M}\int_{B(x,t^{1/\beta
^{\ast }})}|u(x)-u(y)|^{p}p_{t}(x,y)d\mu (y)d\mu (x)  \notag \\
&\leq &\frac{1}{t^{{p\sigma }/{\beta ^{\ast }}}}\int_{M}\int_{B(x,t^{1/\beta
^{\ast }})}|u(x)-u(y)|^{p}\frac{c_{3}}{V(x,t^{1/\beta ^{\ast }})}\exp \left(
-c_{4}\left( \frac{d(x,y)}{t^{1/\beta ^{\ast }}}\right) ^{\frac{\beta ^{\ast
}}{\beta ^{\ast }-1}}\right) d\mu (y)d\mu (x)  \notag \\
&\leq &c_{3}\Phi _{u}^{\sigma }(t^{1/\beta ^{\ast }}).  \label{eq3.3-1}
\end{eqnarray}

Therefore, by (\ref{eq3.3-2}) and (\ref{eq3.3-1}),

\begin{eqnarray}
\Psi _{u}^{\sigma }(t) &\leq &c_{3}C_{d}\sum_{n=1}^{\infty }1_{\{n\leq \log
_{2}(R_{0}t^{-1/\beta ^{\ast }})\}}\exp \left( -c_{4}2^{\frac{(n-1)\beta
^{\ast }}{\beta ^{\ast }-1}}\right) 2^{(p\sigma +\alpha _{1})n}\Phi
_{u}^{\sigma }(2^{n}t^{1/\beta ^{\ast }})+c_{3}\Phi _{u}^{\sigma
}(t^{1/\beta ^{\ast }})  \notag \\
&\leq &c_{3}C_{d}\sum_{n=0}^{\infty }1_{\{n\leq \log _{2}(R_{0}t^{-1/\beta
^{\ast }})\}}\exp \left( -c_{4}2^{\frac{(n-1)\beta ^{\ast }}{\beta ^{\ast }-1%
}}\right) 2^{(p\sigma +\alpha _{1})n}\Phi _{u}^{\sigma }(2^{n}t^{1/\beta
^{\ast }}),  \label{eq3.3}
\end{eqnarray}

it follows that
\begin{align*}
& E_{p,p}^{\sigma }(u)=\int_{0}^{R_{0}^{\beta ^{\ast }}}\Psi _{u}^{\sigma
}(t)\frac{dt}{t}\leq c_{3}C_{d}\sum_{n=0}^{\infty }\exp \left( -c_{4}2^{%
\frac{(n-1)\beta ^{\ast }}{\beta ^{\ast }-1}}\right) \int_{0}^{2^{-n\beta
^{\ast }}R_{0}^{\beta ^{\ast }}}2^{(p\sigma +\alpha _{1})n}\Phi _{u}^{\sigma
}(2^{n}t^{1/\beta ^{\ast }})\frac{dt}{t} \\
& =c_{3}\beta ^{\ast }C_{d}\sum_{n=0}^{\infty }\exp \left( -c_{4}2^{\frac{%
(n-1)\beta ^{\ast }}{\beta ^{\ast }-1}}\right) 2^{(p\sigma +\alpha
_{1})n}\int_{0}^{R_{0}}\Phi _{u}^{\sigma }(r)\frac{dr}{r}=c_{3}\beta ^{\ast
}C_{d}C_{0}[u]_{B_{p,p}^{\sigma }}^{p},
\end{align*}%
where
\begin{equation}
C_{0}=\sum_{n=0}^{\infty }\exp \left( -c_{4}2^{\frac{(n-1)\beta ^{\ast }}{%
\beta ^{\ast }-1}}\right) 2^{(p\sigma +\alpha _{1})n}<\infty.  \label{c1}
\end{equation}%
The proof is complete.
\end{proof}

The following two propositions deal with the `local' semi-norms.

\begin{proposition}
\label{lem3.3} If $(M,d,\mu )$ admits a heat kernel $\{p_{t}\}_{t>0}$
satisfying (LHE), then there exists a positive constant $C$ such that for
all $u\in \mathcal{D}(E_{p,\infty }^{\sigma })$,
\begin{equation*}
\lbrack u]_{B_{p,\infty }^{\sigma }}^{p}\leq CE_{p,\infty }^{\sigma }(u).
\end{equation*}
\end{proposition}

\begin{proof}
{By \eqref{eq3-2}, we have
\begin{align*}
& E_{p,\infty }^{\sigma }(u)=\sup_{t\in (0,R_{0}^{\beta ^{\ast }})}\frac{1}{%
t^{{p\sigma }/{\beta ^{\ast }}}} \int_{M}\int_{M}|u(x)-u(y)|^{p}p_{t}(x,y)d%
\mu (x)d\mu (y) \\
&\geq c_{1}e^{-c_2}\sup_{t\in (0,R_{0}^{\beta ^{\ast }})}\Phi_u^{\sigma}(
t^{1/\beta^*}) =c_{1}e^{c_2}\sup_{r\in (0,R_{0})}\Phi_u^{\sigma}(r)
=c_{1}e^{c_2}[u]_{B_{p,\infty }^{\sigma }}^{p},
\end{align*}%
}which ends the proof.
\end{proof}

\begin{proposition}
\label{lem3.4} If $(M,d,\mu )$ admits a heat kernel $\{p_{t}\}_{t>0}$
satisfying (UHE), then there exists a positive constant $C$ such that for
all $u\in B_{p,\infty }^{\sigma }$,
\begin{equation*}
E_{p,\infty }^{\sigma }(u)\leq C[u]_{B_{p,\infty }^{\sigma }}^{p}.
\end{equation*}
\end{proposition}

\begin{proof}
By \eqref{eq3.3} and \eqref{VD}, we have
\begin{align}
& \Psi _{u}^{\sigma }(t)\leq c_{3}C_{d}\sum_{n=0}^{\infty }1_{\{n\leq \log
_{2}(R_{0}t^{-1/\beta ^{\ast }})\}}\exp \left( -c_{4}2^{\frac{(n-1)\beta
^{\ast }}{\beta ^{\ast }-1}}\right) 2^{(p\sigma +\alpha _{1})n}\Phi
_{u}^{\sigma }(2^{n}t^{1/\beta ^{\ast }})  \notag \\
=& c_{3}C_{d}\sum_{n=0}^{\infty }\exp \left( -c_{4}2^{\frac{(n-1)\beta
^{\ast }}{\beta ^{\ast }-1}}\right) 2^{(p\sigma +\alpha
_{1})n}[u]_{B_{p,\infty }^{\sigma }}^{p}=c_{3}C_{d}C_{0}[u]_{B_{p,\infty
}^{\sigma }}^{p},  \label{eq3.14}
\end{align}%
where $C_{0}$ is finite by \eqref{c1}.
The proof is complete by taking the supremum in $(0,R_{0}^{\beta ^{\ast }})$ on both sides
of \eqref{eq3.14}.
\end{proof}

The following proposition concerns the Korevaar-Schoen semi-norm.

\begin{proposition}
\label{corol3.5} If $(M,d,\mu )$ admits a heat kernel $\{p_{t}\}_{t>0}$
satisfying (UHE) and (LHE), then for all $u\in B_{p,\infty }^{\sigma }$,
\begin{equation}
\limsup_{t\rightarrow 0}\Psi _{u}^{\sigma }(t)\asymp \limsup_{r\rightarrow
0}\Phi _{u}^{\sigma }(r).  \label{limsup_psiphi}
\end{equation}
\end{proposition}

\begin{proof}
By Proposition \ref{lem3.3}, we have%
\begin{equation}
\sup_{r\in (0,R_{0})}\Phi _{u}^{\sigma }(r)\leq C\sup_{t\in (0,R_{0}^{\beta
^{\ast }})}\Psi _{u}^{\sigma }(t),  \label{31}
\end{equation}%
for any $R_{0}\in (0,$\textrm{diam}$(M)]$. Therefore, if we choose $
R_{0}=\epsilon $ in \eqref{31}, then
\begin{align*}
\limsup_{t\rightarrow 0}\Psi _{u}^{\sigma }(t)& =\lim_{\epsilon \rightarrow
0}\sup_{t\in (0,\epsilon )}\Psi _{u}^{\sigma }(t) \\
& \geq C_{1}\lim_{\epsilon \rightarrow 0}\sup_{r\in (0,\epsilon ^{1/\beta
^{\ast }})}\Phi _{u}^{\sigma }(r)=C_{1}\limsup_{r\rightarrow 0}\Phi
_{u}^{\sigma }(r).
\end{align*}

Next, we consider the left-hand side of \eqref{limsup_psiphi}. By %
\eqref{eq3.3} and the dominated convergence theorem, we have that
\begin{align*}
& \limsup_{t\rightarrow 0}\frac{1}{t^{p\sigma /\beta ^{\ast }}}%
\int_{M}\int_{M}|u(x)-u(y)|^{p}p_{t}(x,y)d\mu (y)d\mu (x) \\
\leq & C_{1}\sum_{n=0 }^{\infty }\exp \left( -c_{4}2^{\frac{(n-1)\beta
^{\ast }}{\beta ^{\ast }-1}}\right) \limsup_{t\rightarrow 0}\frac{1_{\{n\leq
\log _{2}(R_{0}t^{-1/\beta ^{\ast }})\}}}{(2^{n}t^{1/\beta ^{\ast
}})^{p\sigma }}\int_{M}\frac{2^{(p\sigma +\alpha _{1})n}}{%
V(x,2^{n}t^{1/\beta ^{\ast }})}\int_{B(x,2^{n}t^{1/\beta ^{\ast
}})}|u(x)-u(y)|^{p}d\mu (y)d\mu (x) \\
\leq & C_{1}\sum_{n=0 }^{\infty }\exp \left( -c_{4}2^{\frac{(n-1)\beta
^{\ast }}{\beta ^{\ast }-1}}\right) 2^{(p\sigma +\alpha
_{1})n}\limsup_{r\rightarrow 0}\Phi _{u}^{\sigma }(r)\leq
C_{1}C_{0}\limsup_{r\rightarrow 0}\Phi _{u}^{\sigma }(r),
\end{align*}%
where $C_{1}=c_{3}C_{d}$ and the constant $C_{0}$ is defined in \eqref{c1}.
The proof is complete.
\end{proof}

\begin{proof}[Proof of Lemma \protect\ref{thm1}]
By Propositions \ref{lem3.1} and \ref{lem3.2}, we obtain $\mathcal{%
D}(E_{p,p}^{\sigma })=B_{p,p}^{\sigma }$ and (\ref{eq1.8}). By Propositions \ref{lem3.3} and \ref{lem3.4}, we obtain $%
\mathcal{D}(E_{p,\infty }^{\sigma })=B_{p,\infty }^{\sigma }$ and $%
E_{p,\infty }^{\sigma }(u)\asymp \lbrack u]_{B_{p,\infty }^{\sigma }}^{p}$
. By
Proposition \ref{corol3.5}, we finally obtain \eqref{eq1.7}. The proof is complete.
\end{proof}

\subsection{Equivalence between $(KE)_{\protect\sigma}$ and  $(NE)_{\protect\sigma}$}

We start by an easy side $(NE)_{\sigma}\Rightarrow (KE)_{\sigma}$.

\begin{proposition}
\label{thm_2} Suppose that $(M,d,\mu )$ admits a heat kernel $%
\{p_{t}\}_{t>0} $ satisfying (LHE). Then there exists $C>0$ such that for
all $u\in \mathcal{D}(E_{p,\infty }^{\sigma })$,
\begin{equation*}
\Psi _{u}^{\sigma }(t)\geq C\Phi _{u}^{\sigma }(t^{1/\beta ^{\ast }}).
\end{equation*}%
Consequently,
\begin{equation}
\liminf_{t\rightarrow 0}\Psi _{u}^{\sigma }(t)\geq C\liminf_{r\rightarrow
0}\Phi _{u}^{\sigma }(r),  \label{limpsi}
\end{equation}
where the constant $C$ depends only on $c_{1}$ in (LHE).
\end{proposition}

\begin{proof}
By \eqref{eq3-2}, we have
\begin{align*}
\Psi _{u}^{\sigma }(t)& \geq \frac{1}{t^{p\sigma /\beta ^{\ast }}}%
\int_{M}\int_{B(x,\delta t^{1/\beta ^{\ast }})}|u(x)-u(y)|^{p}p_{t}(x,y)d\mu
(y)d\mu (x) \\
& \geq \frac{c_{1}C_{d}^{-1}\delta ^{\alpha _{1}}}{t^{p\sigma /\beta ^{\ast
}}}\int_{M}\frac{1}{V(x,\delta t^{1/\beta ^{\ast }})}\int_{B(x,\delta t^{{1}/%
{\beta ^{\ast }}})}|u(x)-u(y)|^{p}d\mu (y)d\mu (x) \\
& =c_{1}C_{d}^{-1}\delta ^{p\sigma +\alpha _{1}}\Phi _{u}(\delta t^{1/\beta
^{\ast }}).
\end{align*}%
Taking the lower limit on both sides of the above inequality, we have
\begin{equation*}
\liminf_{t\rightarrow 0}\Psi _{u}^{\sigma }(t)\geq c_{1}C_{d}^{-1}\delta
^{p\sigma +\alpha _{1}}\liminf_{t\rightarrow 0}\Phi _{u}^{\sigma
}(t^{1/\beta ^{\ast }})=c_{1}C_{d}^{-1}\delta ^{p\sigma +\alpha
_{1}}\liminf_{r\rightarrow 0}\Phi _{u}^{\sigma }(r),
\end{equation*}%
which implies \eqref{limpsi}.
\end{proof}

We give a proof for the following claim which appeared in \cite[
The proof of Lemma 4.13]{AlonsoBaudoinchen2021CVPDE}.

\begin{proposition}
\label{hk_ct} Suppose that $(M,d,\mu )$ admits a heat kernel $%
\{p_{t}\}_{t>0} $ satisfying (UHE) and (LHE). Then there exist $C,c>1$ and $c^{\prime }>0 $ such that for any $%
\delta >0$ and for $\mu$-almost all $x,y\in M$ with $d(x,y)>\delta
t^{1/\beta ^{\ast }}$,
\begin{equation*}
p_{t}(x,y)\leq C\exp \left( -c^{\prime }\delta ^{\frac{\beta ^{\ast }}{\beta
^{\ast }-1}}\right) p_{ct}(x,y).
\end{equation*}
\end{proposition}

\begin{proof}
By (UHE) and the volume doubling property \eqref{VD2}, we have
\begin{align*}
p_{t}(x,y)& \leq \frac{c_{3}}{V(x,t^{1/\beta ^{\ast }})}\exp \left(
-c_{4}\left( \frac{d(x,y)}{t^{1/\beta ^{\ast }}}\right) ^{\frac{\beta ^{\ast
}}{\beta ^{\ast }-1}}\right) \\
& \leq \frac{c_{3}C_{d}c^{\alpha _{1}/\beta ^{\ast }}}{V(x,(ct)^{1/\beta
^{\ast }})}\exp \left( -c_{4}\left( \frac{d(x,y)}{(ct)^{1/\beta ^{\ast }}}%
\right) ^{\frac{\beta ^{\ast }}{\beta ^{\ast }-1}}\cdot c^{\frac{1}{\beta
^{\ast }-1}}\right) \\
& =\frac{c_{3}C_{d}c^{\alpha _{1}/\beta ^{\ast }}}{c_{1}}\frac{c_{1}}{%
V(x,(ct)^{1/\beta ^{\ast }})}\exp \left( -c_{4}\left( \frac{d(x,y)}{%
(ct)^{1/\beta ^{\ast }}}\right) ^{\frac{\beta ^{\ast }}{\beta ^{\ast }-1}%
}\cdot c^{\frac{1}{\beta ^{\ast }-1}}\right) ,
\end{align*}%
for $\mu $-almost all $x,y\in M$. Taking $\tilde{c}=c_{4}c^{\frac{1}{\beta
^{\ast }-1}}-c_{2}$, we use (LHE) to obtain
\begin{equation*}
p_{t}(x,y)\leq \frac{c_{3}C_{d}c^{\alpha _{1}/\beta ^{\ast }}}{c_{1}}\exp
\left( -\tilde{c}\left( \frac{d(x,y)}{(ct)^{1/\beta ^{\ast }}}\right) ^{%
\frac{\beta ^{\ast }}{\beta ^{\ast }-1}}\right) p_{ct}(x,y).
\end{equation*}%
We choose $c>1$ to ensure $c^{\frac{1}{\beta ^{\ast }-1}}c_{4}>c_{2}$ so
that $\tilde{c}>0$. Let $c^{\prime }={\tilde{c}}/{c^{\frac{1}{\beta ^{\ast
}-1}}}$, then for $d(x,y)>\delta t^{1/\beta ^{\ast }}$,
\begin{equation*}
p_{t}(x,y)\leq \frac{c_{3}C_{d}c^{\alpha _{1}/\beta ^{\ast }}}{c_{1}}\exp
\left( -c^{\prime }\left( \frac{d(x,y)}{t^{1/\beta ^{\ast }}}\right) ^{\frac{%
\beta ^{\ast }}{\beta ^{\ast }-1}}\right) p_{ct}(x,y)\leq C\exp \left(
-c^{\prime }\delta ^{\frac{\beta ^{\ast }}{\beta ^{\ast }-1}}\right)
p_{ct}(x,y),
\end{equation*}%
where $C=\frac{c_{3}C_{d}c^{\alpha _{1}/\beta ^{\ast }}}{c_{1}}$. The proof
is complete.
\end{proof}

The proof of the implication $(KE)_{\sigma}\Rightarrow (NE)_{\sigma}$ is inspired by \cite[
Lemma 4.13]{AlonsoBaudoinchen2021CVPDE} where $p=1$ there.

\begin{lemma}
\label{thm_1} Suppose that $(M,d,\mu )$ admits a heat kernel $%
\{p_{t}\}_{t>0} $ satisfying (UHE) and (LHE). Assuming the property
$(\widetilde{KE})_{\sigma} $ with $\sigma >0$, we have
\begin{equation}
\liminf_{t\rightarrow 0}\Psi _{u}^{\sigma }(t)\leq C\liminf_{r\rightarrow
0}\Phi _{u}^{\sigma }(r),  \label{limphiu}
\end{equation}
for all $u\in \mathcal{D}(E_{p,\infty }^{\sigma })$.
\end{lemma}

\begin{proof}
Temporally fix $u\in \mathcal{D}(E_{p,\infty }^{\sigma })$. Without loss of
generality, assume that $\liminf_{t\rightarrow 0}\Psi _{u}^{\sigma }(t)>0$.
Fix $\delta >1$ which will be determined later in (\ref{dd}). For $t\in (0,R_{0}^{\beta
^{\ast }})$, decompose
\begin{equation*}
\Psi _{u}^{\sigma }(t)=\Psi _{1}(t)+\Psi _{2}(t),
\end{equation*}%
where
\begin{align*}
\Psi _{1}(t)& =\frac{1}{t^{{p\sigma }/{\beta ^{\ast }}}}\iint_{\{d(x,y)\leq
\delta t^{1/\beta ^{\ast }}\}}|u(x)-u(y)|^{p}p_{t}(x,y)d\mu (y)d\mu (x), \\
\Psi _{2}(t)& =\frac{1}{t^{{p\sigma }/{\beta ^{\ast }}}}\iint_{\{d(x,y)>%
\delta t^{1/\beta ^{\ast }}\}}|u(x)-u(y)|^{p}p_{t}(x,y)d\mu (y)d\mu (x).
\end{align*}

For $\Psi _{1}(t)$, when $d(x,y)\leq \delta t^{1/\beta ^{\ast }}$, we have
\begin{align*}
\frac{c_{1}}{V(x,t^{1/\beta ^{\ast }})}\exp \left( -c_{2}\delta ^{\frac{%
\beta ^{\ast }}{\beta ^{\ast }-1}}\right) \leq & \frac{c_{1}}{V(x,t^{1/\beta
^{\ast }})}\exp \left( -c_{2}\left( \frac{d(x,y)}{t^{1/\beta ^{\ast }}}%
\right) ^{\frac{\beta ^{\ast }}{\beta ^{\ast }-1}}\right) \\
\leq & p_{t}(x,y)\leq \frac{c_{3}}{V(x,t^{1/\beta ^{\ast }})}\exp \left(
-c_{4}\left( \frac{d(x,y)}{t^{1/\beta ^{\ast }}}\right) ^{\frac{\beta ^{\ast
}}{\beta ^{\ast }-1}}\right) \leq \frac{c_{5}}{V(x,t^{1/\beta ^{\ast }})}.
\end{align*}%
Therefore, noting that $\delta>1$, we have from \eqref{VD2} that
\begin{align}
\Psi _{1}(t)& \leq \frac{c_{3}}{t^{p\sigma /\beta ^{\ast }}}%
\iint_{\{d(x,y)\leq \delta t^{1/\beta ^{\ast }}\}}\frac{|u(x)-u(y)|^{p}}{%
V(x,t^{1/\beta ^{\ast }})}d\mu (y)d\mu (x)  \notag \\
& \leq \frac{c_{3}C_{d}\delta ^{\alpha _{1}}}{t^{p\sigma /\beta ^{\ast }}}%
\iint_{\{d(x,y)\leq \delta t^{1/\beta ^{\ast }}\}}\frac{|u(x)-u(y)|^{p}}{%
V(x,\delta t^{1/\beta ^{\ast }})}d\mu (y)d\mu (x)=C_{1}\Phi _{u}^{\sigma
}(\delta t^{1/\beta ^{\ast }}),  \label{psi1}
\end{align}%
where $C_{1}=c_{3}C_{d}\delta ^{p\sigma +\alpha _{1}}$.

For $\Psi _{2}(t)$, by Proposition \ref{hk_ct}, we have that
\begin{equation}
\Psi _{2}(t)\leq A\Psi _{u}^{\sigma }(ct),  \label{psi2}
\end{equation}%
where $A=Cc^{p\sigma /\beta ^{\ast }}\exp \left( -c^{\prime }\delta ^{\frac{%
\beta ^{\ast }}{\beta ^{\ast }-1}}\right) $ can be sufficiently small by choosing sufficiently large $\delta $ (since $C,c,c^{\prime }$ are
independent constants). So, from now on, we choose sufficiently large $\delta $ such
that
\begin{equation}
A=Cc^{p\sigma /\beta ^{\ast }}\exp \left( -c^{\prime }\delta ^{\frac{\beta
^{\ast }}{\beta ^{\ast }-1}}\right) <\frac{1}{2}.  \label{dd}
\end{equation}
Combining \eqref{psi1} and \eqref{psi2}, we have%
\begin{equation}
\Psi _{u}^{\sigma }(t)\leq C_{1}\Phi _{u}^{\sigma }(\delta t^{1/\beta ^{\ast
}})+A\Psi _{u}^{\sigma }(ct).  \label{41}
\end{equation}

%

By property $(\widetilde{KE})_{\sigma}$, there exists $t_{0}>0$ such that for all $%
t\in (0,t_{0})$,
\begin{equation*}
\Psi _{u}^{\sigma }(t)\leq \sup_{0<t<t_{0}}\Psi _{u}^{\sigma }(t)\leq
2\limsup_{t\rightarrow 0}\Psi _{u}^{\sigma }(t)\leq
C_{2}\liminf_{t\rightarrow 0}\Psi _{u}^{\sigma }(t),
\end{equation*}%
where $C_{2}$ comes from property $(\widetilde{KE})_{\sigma}$. Therefore, we have
from \eqref{41} that, for all $t\in (0,\frac{t_{0}}{c})$,
\begin{equation*}
\Psi _{u}^{\sigma }(t)\leq C_{1}\Phi _{u}^{\sigma }(\delta t^{1/\beta ^{\ast
}})+AC_{2}\liminf_{t\rightarrow 0}\Psi _{u}^{\sigma }(t).
\end{equation*}
Fix $\delta>0 $ such that $AC_{2}<1$ with the aforementioned requirement $A<
\frac{1}{2}$. It follows that
\begin{equation*}
\liminf_{t\rightarrow 0}\Psi _{u}^{\sigma }(t)\leq
C_{1}\liminf_{t\rightarrow 0}\Phi _{u}^{\sigma
}(t)+AC_{2}\liminf_{t\rightarrow 0}\Psi _{u}^{\sigma }(t),
\end{equation*}%
which implies
\begin{equation*}
\liminf_{t\rightarrow 0}\Psi _{u}^{\sigma }(t)\leq
C_{3}\liminf_{t\rightarrow 0}\Phi _{u}^{\sigma }(t),
\end{equation*}%
where $C_{3}=\frac{C_{1}}{1-AC_{2}}$. The proof is complete.
\end{proof}

\begin{proof}[Proof of Theorem \protect\ref{thm4}]
Since the two-sided estimates (UHE) and (LHE) hold, we
have that $\mathcal{D}(E_{p,\infty }^{\sigma })=B_{p,\infty }^{\sigma }$ for
any $\sigma >0$ by Lemma \ref{thm1}.

By \eqref{limsup_psiphi} and \eqref{limphiu}, $(\widetilde{KE})_{\sigma}$ implies $%
(\widetilde{NE})_{\sigma}$ since for all $u\in B_{p,\infty }^{\sigma }=\mathcal{D}%
(E_{p,\infty }^{\sigma })$,
\begin{equation}
\limsup_{r\rightarrow 0}\Phi _{u}^{\sigma }(r)\leq C\limsup_{t\rightarrow
0}\Psi _{u}^{\sigma }(t)\leq C^{\prime }\liminf_{t\rightarrow 0}\Psi
_{u}^{\sigma }(t)\leq C^{\prime \prime }\liminf_{r\rightarrow 0}\Phi
_{u}^{\sigma }(r).  \label{ke_ne1}
\end{equation}

By \eqref{limsup_psiphi} and \eqref{limpsi}, $(\widetilde{NE})_{\sigma}$ implies $(\widetilde{KE})_{\sigma}$ since
\begin{equation}
\limsup_{t\rightarrow 0}\Psi _{u}^{\sigma }(t)\leq
C_{1}\limsup_{r\rightarrow 0}\Phi _{u}^{\sigma }(r)\leq C_{1}^{\prime
}\liminf_{r\rightarrow 0}\Phi _{u}^{\sigma }(r)\leq C_{1}^{\prime \prime
}\liminf_{t\rightarrow 0}\Psi _{u}^{\sigma }(t).  \label{ne_ke1}
\end{equation}%
Therefore, by \eqref{ke_ne1} and \eqref{ne_ke1}, we show Theorem \ref{thm4}
(i).

By \eqref{limphiu} and Proposition \ref{lem3.3}, $(KE)_{\sigma}$ implies $(NE)_{\sigma}$ since
\begin{equation}
\sup_{r\in (0,R_{0})}\Phi _{u}^{\sigma }(r)\leq C_{2}\sup_{t\in
(0,R_{0}^{\beta ^{\ast }})}\Psi _{u}^{\sigma }(t)\leq C_{2}^{\prime
}\liminf_{t\rightarrow 0}\Psi _{u}^{\sigma }(t)\leq C_{2}^{\prime \prime
}\liminf_{r\rightarrow 0}\Phi _{u}^{\sigma }(r),  \label{ke_ne2}
\end{equation}
where we use the simple fact $(KE)_{\sigma}$$\Rightarrow (\widetilde{KE})_{\sigma}$. Similarly,
by \eqref{limpsi} and Proposition \ref{lem3.4} adjoint with the fact $(NE)_{\sigma}$$%
\Rightarrow (\widetilde{NE})_{\sigma}$, we immediately derive $(KE)_{\sigma}$ since
\begin{equation}
\sup_{t\in (0,R_{0}^{\beta ^{\ast }})}\Psi _{u}^{\sigma }(t)\leq
C_{3}\sup_{r\in (0,R_{0})}\Phi _{u}^{\sigma }(r)\leq C_{3}^{\prime
}\liminf_{r\rightarrow 0}\Phi _{u}^{\sigma }(r)\leq C_{3}^{\prime \prime
}\liminf_{t\rightarrow 0}\Psi _{u}^{\sigma }(t).  \label{ne_ke2}
\end{equation}
Thus, Theorem \ref{thm4} (ii) holds by \eqref{ke_ne2} and \eqref{ne_ke2}.
\end{proof}

\begin{remark}
We say that a heat kernel $\{p_{t}\}_{t>0}$ satisfies the near diagonal lower estimate (NLE), if there exists $\delta \in (0,\infty)$ such that, for all $t\in (0,R_{0}^{\beta
^{\ast}}) $ and $\mu $-almost all $x,y\in M$ :
\begin{equation}
p_{t}(x,y)\geq \frac{c_{1}}{V(x,t^{1/\beta ^{\ast }})}\ \ \ \text{whenever }
\ d(x,y)\leq \delta t^{1/\beta ^{\ast }}.  \label{NLE}
\end{equation}
We conjecture that Lemma \ref{thm1} and Lemma \ref{thm_1} cannot be derived
from the weaker heat kernel estimates (UHE) and (NLE) without further assumptions (like the chain condition).
\end{remark}

By Theorem \ref{thm4} and Remark \ref{rk1}, we have the following
interesting corollary for $p=2$, which is important for the construction of local
Dirichlet forms based on Besov norms.

\begin{corollary}
\label{jjjj} Suppose that $(M,d,\mu )$ admits a heat kernel $\{p_{t}\}_{t>0}$
satisfying (UHE) and (LHE), then $(NE)_{\beta
^{\ast }/2}$ holds when $p=2$.
\end{corollary}

\section{Connected homogeneous p.c.f self-similar sets and their glue-ups}

\label{sec5} In this section, we analyze bounded and unbounded `fractal glue-ups' by using connected homogeneous p.c.f self-similar sets as `tiles'. We first obtain equivalent discrete semi-norms
(vertex energies) in Subsection \ref{subsec5.1}. Then in Subsection \ref{subsec5.2}, we introduce weak-monotonicity properties for vertex
energies, including property (E) in \cite[Definition 3.1]{GaoYuZhang2022PA},
properties $(VE)_{\sigma}$ and $(\widetilde{VE})_{\sigma}$. We will verify property (E) on
nested fractals. In Subsection \ref{subsec5.3}, we show the equivalence between $(KE)_{\sigma}$ and
$(VE)_{\sigma}$. Lastly, we show that aforementioned consequences hold for certain
`fractal glue-ups', for example, nested fractals and their fractal blow-ups in Subsection \ref{subsec5.4}.

\subsection{Equivalent discrete semi-norms of fractal glue-ups}

\label{subsec5.1} We first introduce connected homogeneous p.c.f.
self-similar sets in $\mathbb{R}^{d}$ ($d\geq 2$). Let $\{\phi _{i}\}_{i=1}^{N}$ $(N\geq 2)$ be an
IFS where
each $\phi _{i}:\mathbb{R}^{d}\rightarrow \mathbb{R}%
^{d}$ is of the form
\begin{equation}
\phi _{i}(x)=\rho (x-b_{i})+b_{i},  \label{phi_i}
\end{equation}%
with $\rho \in (0,1)$, $b_{i}\in \mathbb{R}^{d}$ for $%
i=1,...,N$. Let $K$ be the
attractor of the IFS $\{\phi _{i}\}_{i=1}^{N}$, namely, $K$ is the unique
non-empty compact set in $\mathbb{R}^{d}$ satisfying
\begin{equation*}
K=\tbigcup_{i=1}^{N}\phi _{i}(K).
\end{equation*}%
To explain what `p.c.f' is, we introduce the natural symbolic space
associated with the IFS $\{\phi _{i}\}_{i=1}^{N}$. Let $W=\{1,2,...,N\}$, $%
W_{n}$ be the set of words with length $n$ over $W$, and $W^{\mathbb{N}}$ be
the set of all infinite words over $W$. For $w=\mathrm{w_{1}w_{2}}...\in W^{%
\mathbb{N}}$, the canonical projection $\pi :$ $W^{\mathbb{N}%
}\rightarrow K$ is defined by $\pi (w):=\tbigcap\limits_{n\in \mathbb{N}%
^{\ast }}F_{\mathrm{w_{1}...w_{n}}}(K),$ where $F_{\mathrm{w_{1}...w_{n}}%
}:=F_{\mathrm{w_{1}}}\circ ...\circ F_{\mathrm{w_{n}}}$.
Then the critical set $\Gamma $ and the post-critical set $%
\mathcal{P}$ is defined by
\begin{equation*}
\Gamma =\pi ^{-1}\left( \tbigcup_{1\leq i<j\leq N}\left( \phi _{i}(K)\cap
\phi _{j}(K)\right) \right) ,\quad \mathcal{P}=\tbigcup_{m\geq 1}\tau
^{m}(\Gamma ),
\end{equation*}%
where $\tau $ is the left shift by one index on $W^{\mathbb{N}}$ (see for example \cite[%
Definition 1.3.13]{Kigami.2001.}). We say that the IFS $\{\phi _{i}\}_{i=1}^{N}$ is \emph{post-critically
finite} (p.c.f.) if $\mathcal{P}$ is finite. Define
\begin{equation}
V_{0}=\pi (\mathcal{P}),\quad V_{w}:=F_{w}\left( V_{0}\right), \quad
V_{n}=\tbigcup_{w\in W_{n}}V_{w} ,\quad V_{\ast }=\tbigcup_{n\geq 1}V_{n}, \label{e.V}
\end{equation}
then $K$ is the closure of $V_{\ast }$ with respect to the Euclidean metric.
For $w\in W_n$, $$K_{w}:=F_{w}\left( K\right) $$ is called an $n$-cell of $K$.
Without loss of generality, we always assume that \begin{equation}\textrm{diam}(K)=1\label{dia}\end{equation} by an
affine transformation on $\{b_{i}\}_{i=1}^{N}$. From now on, denote by $K$ a
homogeneous p.c.f. self-similar set that is connected. It is known that a
homogeneous p.c.f. IFS in the form of (\ref{phi_i}) satisfies the \emph{open
set condition} (OSC) (see for example \cite[Theorem 1.1]{Denglau.2008}).
Hence the Hausdorff dimension of $K$ is $$\alpha =\func{dim}_{H}(K)=-\log
N/\log \rho,$$ and the $\alpha$-dimensional Hausdorff measure on $K$, denote by $\mu $, is \emph{$\alpha $-regular}.

It is known from \cite[Proposition 2.5]{GuLau.2020.AASFM} that \emph{%
Condition(H)} holds for connected homogeneous p.c.f. self-similar sets. That
is, there exists $c>0$ depending only on $K$ such that, for any two words $w$
and $w^{\prime }$ with the same length $m\geq 1$, $K_{w}\cap K_{w^{\prime
}}=\emptyset $ implies that $\func{dist}\left( K_{w},K_{w^{\prime }}\right)
\geq c\rho ^{m}$.

In order to unify the notation for bounded and unbounded fractals, we
introduce the concept of \emph{fractal glue-ups}. The definition is adapted
from the idea of fractafold in \cite{StrichartzTeplyaev2012} although it is not the
same. Given a compact set $K\subset \mathbb{R}^d$ and a set $F$ of
similitudes on $\mathbb{R}^d$, define
\begin{equation*}
K^F:=\bigcup_{f\in F}f(K)\subset \mathbb{R}^d.
\end{equation*}
In this paper, we need the following requirements.

\begin{enumerate}
\item Isometries: all similitudes in $F$ are isometries.

\item Just-touching property: for any $f\neq g\in F$,
\begin{equation*}
f(K)\tbigcap g(K)=f(V_0)\tbigcap g(V_0).
\end{equation*}

\item Condition (H): There exists a constant $C_{H}\in (0,1)$ such that, if $%
|x-y|<C_{H}\rho ^{m}$ and $x\in f_{1}(K_{w})$ and $y\in f_{2}(K_{\tilde{w}%
}) $ for two words $w$ and $\tilde{w}$ with the same length $m\geq0$ and $f_1,f_2\in F$, then $%
f_{1}(K_{w})$ intersects $f_{2}(K_{\tilde{w}})$.

\item Connectedness: $K^{F}$ is connected.
\end{enumerate}


From now on, whenever we mention a \emph{fractal glue-up} $K^{F}$, we
automatically assume that these four conditions are satisfied. We call each $%
f(K)$ a \emph{tile} ($f\in F$). The above conditions (2)
and (3) imply the following \emph{uniform finitely-joint} property.

\begin{proposition}
For any metric ball $B(x,R)$ of radius $R>0$ in $K^F$, the number of tiles
that intersect it is bounded by a constant $C_K(R)<\infty$, which depends
only on $C_H$ and the IFS. Thus for any tile $f(K)$, the number of tiles $%
\tilde{f}(K)$ that intersect it is bounded by $C_K(1)$.\label{4.1}
\end{proposition}

\begin{proof}
Denote all the tiles that intersect $B(x,R)$ by $f_{n}(K),\ 1\leq n\leq
N_{0} $ (here $N_{0}$ might be $\infty $). Denote by $t$ the smallest
integer satisfying
\begin{equation*}
N^t>|V_0|.
\end{equation*}
Choose a level-$t$ cell $\phi_{w}(K)$ ($%
|w|=t$) that does not intersect $V_{0}$. Such a cell exists by the pigeonhole
principle. Condition (H) implies that all $f_{n}(\phi _{w}(K))$ are
separated by distance $C_{H}\rho ^{t}$ for $1\leq n\leq N_{0}$, but they all
belong to the ball $B(x,R+1)$ by \eqref{dia}. Let $C_K(R)$ be the maximal cardinality of
disjoint balls of radii $C_{H}\rho ^{t}/2$ inside a ball with diameter $R+1$
in $\mathbb{R}^d$.
Since there are $N_{0}$
disjoint balls $B(f_{n}(\phi _{w}(v)),C_{H}\rho ^{t}/2)$ inside $B(x,R+1)$, we obtain $N_{0}\leq C_K(R)$. The second inclusion follows from choosing $R=1$
and $x\in f(K)$.
\end{proof}

\begin{remark}
The concept of the fractal glue-up unifies the notions
of bounded and unbounded fractals. For example, $K=K^{F}$ when $F=\{Id\}$.
We are mainly interested in the following unbounded fractal by blowing up $K$
 in \cite[Section 3]{Strichartz.1998.CJM}. However,
fractal glue-ups may lack self-similarity on large scales.
\end{remark}

\begin{example}[Fractal blow-up]\label{ex1} Let $\{\phi _{i}\}_{i=1}^{N}$ be defined as in (\ref{phi_i}). Define
\begin{equation*}
K_{l}:=\underbrace{\phi _{1}^{-1}\circ \phi _{1}^{-1}\circ \cdots \circ \phi
_{1}^{-1}}_{l}K=\rho^{-l}K=\tbigcup _{f\in F_{l}}f(K),
\end{equation*}%
where $F_{l}=\left\{x+\sum_{j=1}^{l}\rho ^{-j}c_{j}: c_{j}\in \{(1-\rho)b_{i}\}_{i=1}^{N}\right\}$.

Clearly, $F_{l}$ and $K_{l}$ are increasing sequences, so that
$K=K_{0}\subseteq K_{1}\subseteq K_{2}\subseteq \cdots $.

The final fractal blow-up of $K$ given by
\begin{equation*}
K_{\infty }=\tbigcup _{l=1}^{\infty }K_{l}= \lim_{l\rightarrow\infty} K_l,
\end{equation*}
is in the form $K^{F}$ with $F=\bigcup_{l=1}^{\infty }
F_{l}= \lim_{l\rightarrow\infty} F_l$.
\end{example}

From now on, we assume that $K$ is a connected homogeneous p.c.f.
self-similar set when we say that $K^{F}$ is a fractal glue-up. In what
follows, we use the superscript $F$ to distinguish the underlying space $K$
and $K^{F}$ for the energies and norms etc.

Let
\begin{equation}
\mu ^{F}:=\sum_{f\in F}\mu \circ f^{-1},  \label{muF}
\end{equation}%
be the glue-up measure on $K^{F}$. Then $\mu ^{F}$ is $\alpha $%
-regular by our assumptions on $F$. Let
\begin{equation}
\lbrack u]_{B_{p,\infty }^{\sigma ,F}}^{p}:=\sup_{n\geq 0}\Phi _{u}^{\sigma
,F}(C_H\rho ^{n}),  \label{ct}
\end{equation}%
where $\Phi _{u}^{\sigma ,F}(r):=r^{-(p\sigma +\alpha
)}\int_{K^{F}}\int_{B(x,r)}|u(x)-u(y)|^{p}d\mu ^{F}(y)d\mu ^{F}(x)$, and
\begin{equation*}
B_{p,\infty }^{\sigma ,F}:=\{u\in L^{p}(K^{F},\mu ^{F}):[u]_{B_{p,\infty
}^{\sigma ,F}}<\infty \}.
\end{equation*}%
It is easy to see that%
\begin{equation}
\sup_{n\geq 0}\Phi _{u}^{\sigma ,F}(C_H\rho ^{n})\asymp \sup_{r\in
(0,C_H)}\Phi _{u}^{\sigma ,F}(r),  \label{ns}
\end{equation}%
which implies that
\begin{equation}
B_{p,\infty }^{\sigma ,F}=B_{p,\infty }^{\sigma }(K^{F})\text{ \ \ and\ \ \ }%
[u]_{B_{p,\infty }^{\sigma ,F}}\asymp \lbrack u]_{B_{p,\infty }^{\sigma
}(K^{F})}\text{ (with }R_{0}=C_H\text{ in (\ref{pi})).}  \label{BC}
\end{equation}
Also, we have%
\begin{equation}
\liminf_{n\rightarrow \infty }\Phi _{u}^{\sigma ,F}(\rho ^{n})\asymp
\liminf_{r\rightarrow 0}\Phi _{u}^{\sigma ,F}(r)\text{ \ \ and \ \ }%
\limsup_{n\rightarrow \infty }\Phi _{u}^{\sigma ,F}(\rho ^{n})\asymp
\limsup_{r\rightarrow 0}\Phi _{u}^{\sigma ,F}(r).  \label{ne}
\end{equation}

We need the following Morrey-Sobolev inequality to show that $B_{p,\infty
}^{\sigma }(K^F)$ essentially embeds into $C(K^F)$ when $\sigma >\alpha /p$.

\begin{lemma}
(\cite[Theorem 3.2]{BaudoinLecture2022}) \label{lemmaMR} Let $(M,d,\mu )$ be
a metric measure space with $\mu $ satisfying (\ref{alpha_r}). When $\sigma
>\alpha /p$, for any $u\in B_{p,\infty }^{\sigma }(M)$, there exists a
continuous version $\tilde{u}\in C^{(p\sigma -\alpha )/p}(M)$ satisfying $%
\tilde{u}=u$ $\mu $-almost everywhere in $M$ and
\begin{equation}
|\tilde{u}(x)-\tilde{u}(y)|\leq C|x-y|^{(p\sigma -\alpha )/p}\sup_{r\in
(0,3d(x,y)]}\Phi _{u}^{\sigma }(r)\leq C|x-y|^{(p\sigma -\alpha
)/p}[u]_{B_{p,\infty }^{\sigma }(M)},  \label{M-S}
\end{equation}%
for all $x,$ $y\in M$ with $d(x,y)<R_{0}/3$, where $C$ is a positive
constant. Here $C^{\beta }(M)$ denotes the class of H\"{o}lder continuous
functions of order $\beta $ on $M$.
\end{lemma}

%
%
%
%
%
%
%
Next, we estimate%
\begin{equation*}
I_{\infty ,n}^{F}(u):=\int_{K^{F}}\int_{B(x,C_{H}\rho
^{n})}|u(x)-u(y)|^{p}d\mu ^{F}(y)d\mu ^{F}(x).
\end{equation*}
By (\ref{ct}), for all $u\in B_{p,\infty }^{\sigma ,F}$,
\begin{equation}
\rho ^{-n(p\sigma +\alpha )}I_{\infty ,n}^{F}(u)=C_{H}^{p\sigma +\alpha
}\Phi _{u}^{\sigma ,F}(C_{H}\rho ^{n}).  \label{ct-1}
\end{equation}
Define
\begin{equation*}
\mu _{m}:=\frac{1}{|V_{m}|}\sum_{a\in V_{m}}\delta _{a}
\end{equation*}%
and
\begin{equation}
\mu _{m}^{F}:=\sum_{f\in F}\mu_m\circ f^{-1}=\sum_{f\in F}\frac{1}{|V_{m}|}%
\sum_{a\in V_{m}}\delta _{f(a)},  \label{200}
\end{equation}%
where $\delta _{a}$ is the Dirac measure at point $a$ and $|A|$ denotes the
cardinality of a finite set $A$.
Let
\begin{equation}
I_{m,n}^{F}(u):=\int_{K^{F}}\int_{B(x,C_{H}\rho ^{n})}|u(x)-u(y)|^{p}d\mu
_{m}^{F}(y)d\mu _{m}^{F}(x).  \label{201}
\end{equation}%
 We will use these discrete sums
to estimate the integral $I_{\infty ,n}^{F}(u)$.
\begin{lemma}
\label{lem4.4}Let $K^{F}$ be a fractal glue-up. Then for all $u\in C(K^F)$,
\begin{equation}
\liminf_{m\rightarrow \infty }I_{m,n}^{F}(u)\geq I_{\infty ,n}^{F}(u).
\label{mu_m}
\end{equation}
\end{lemma}

The proof of this lemma is based on the proof of \cite[Lemma 3.2]
{GuLau.2020.AASFM}. The key is to establish the following Proposition \ref{prop:41}, which is not obvious since $K$ is uncountable. We may assume that $K$ is not contained in any hyperplane $H\subset \mathbb{R}^{d}$.
Otherwise, we could reduce the IFS to $H$ by restricting each similitude $\phi
_{i}$ to $H$ and applying a reversible affine transformation from $H$ to $\mathbb{R}^{d^{\prime }}$, where $%
d^{\prime }$ is the Hausdorff dimension of $H$. The attractor of this reduced IFS is a scaled copy of $K$. Thus
we can apply \cite[Lemma 3.2]{GuLau.2020.AASFM}.

\begin{proposition}\label{prop:41}
Assume that $K$ is not contained in any hyperplane $H\subset\mathbb{R}^{d}$.
For any $r>0$, let $E(r):=\bigcup_{x\in K}\{x\}\times \partial B(x,r).$ Then
$\mu \times \mu (E(r))=0$.
\end{proposition}

\begin{proof}
Let
\begin{equation*}
R_{x}(a,b)=\overline{B(x,a)}\setminus B(x,b).
\end{equation*}%
Fix $x_{w}\in K_{w}$ for each $w\in W_n$. Then
\begin{equation*}
\bigcup_{x\in K_{w}}\partial B(x,r)\subset R_{x_{w}}(r-\rho ^{n},r+\rho
^{n}),
\end{equation*}%
since for any $y\in \partial B(x,r)$, $d(x,y)=r$ and $d(x,x_{w})\leq \rho
^{n}$, thus $d(y,x_{w})\in \lbrack r-\rho ^{n},r+\rho ^{n}].$ Note that
\begin{equation*}
E(r):=\bigcup_{x\in K}\{x\}\times \partial B(x,r)\subset \bigcup_{w\in
W_{n}}K_{w}\times R_{x_{w}}(r-\rho ^{n},r+\rho ^{n}),
\end{equation*}
thus
\begin{align}
\mu \times \mu (E(r))& \leq \sum_{w\in W_{n}}\mu (K_{w})\mu
(R_{x_{w}}(r-\rho ^{n},r+\rho ^{n}))\notag \\
&\leq \mu (K)\sup_{w\in W_{n}}\mu (R_{x_{w}}(r-\rho ^{n},r+\rho ^{n}))\notag \\
&\leq \mu (K)\sup_{x\in K}\mu (R_{x}(r-\rho ^{n},r+\rho ^{n})).\label{hjx}
\end{align}

We claim that
\begin{equation*}
\sup_{x\in K}\mu (R_{x}(r-\rho ^{n},r+\rho ^{n}))\rightarrow 0\text{ as }%
n\rightarrow \infty .
\end{equation*}%
Otherwise, if $\mu (R_{x_{n}}(r-\rho ^{n},r+\rho ^{n}))\geq \delta >0$ for a
subsequence $\{x_{n}\}_n\subset K$, we may choose a subsequence $%
\{x_{n_{i}}\}_i$ that converge to $x_{\infty }\in K$ due to the compactness
of $K$, and further require that $d(x_{n_{i}},x_{\infty })$ decrease in $i$.
Then
\begin{equation*}
A_{i}:=R_{x_{\infty }}(r-\rho ^{n_{i}}-d(x_{n_{i}},x_{\infty }),r+\rho
^{n_{i}}+d(x_{n_{i}},x_{\infty }))\supset R_{x_{n_{i}}}(r-\rho
^{n_{i}},r+\rho ^{n_{i}})
\end{equation*}%
and $A_{i}\downarrow \partial B(x_{\infty },r)$, thus
\begin{equation*}
\mu (\partial B(x_{\infty },r))=\lim_{i\rightarrow \infty }\mu (A_{i})\geq
\delta ,
\end{equation*}%
which contradicts \cite[Proposition 2.4]{GuLau.2020.AASFM} that $\mu
(\partial B(x,r))=0$ for any $x\in K$.

The proof is complete by letting $n\rightarrow\infty$ in \eqref{hjx}.
\end{proof}

\begin{proof}[Proof of Lemma \protect\ref{lem4.4}]
Let $F=\{f_{i}\}_{i\geq 1}$, and let $F_{k}=\{f_{i}\}_{1\leq i\leq k}$%
. As $\mu _{m}$ weak $\ast $-converges to $\mu $, we know by definition that
the finite glue-up measure $\mu _{m}^{F_{k}}$ weak $\ast $-converges to $\mu
^{F_{k}}$ on $K^{F_{k}}$ (the pieces $f_{i}(K)$ are mutually
measure-disjoint), and that $\mu ^{F_{k}}$ is exactly the restriction of $%
\mu ^{F}$ to $K^{F_{k}}$. Using the same argument as in that of \cite[Lemma 3.2]%
{GuLau.2020.AASFM}, for $u\in C(K^F)$, we have
\begin{eqnarray*}
&&\lim_{m\rightarrow \infty }I_{m,n}^{F_{k}}(u)=\lim_{m\rightarrow \infty
}\int_{K^{F_{k}}}\int_{B(x,C_{H}\rho ^{n})}|u(a)-u(b)|^{p}d\mu
_{m}^{F_{k}}(b)d\mu _{m}^{F_{k}}(a) \\
&=&\int_{K^{F_{k}}}\int_{B(x,C_{H}\rho ^{n})}|u(a)-u(b)|^{p}d\mu
^{F_{k}}(b)d\mu ^{F_{k}}(a):=I_{\infty ,n}^{F_{k}}(u).
\end{eqnarray*}

Note that the nonnegative sequence $I_{m,n}^{F_{k}}(u)\uparrow
I_{m,n}^{F}(u)$ when $k\rightarrow \infty $. Denote $I_{m ,n}^{F_{0}}(u)=0$, then
\begin{equation}
I_{m,n}^{F}(u)=\sum_{k=0}^{\infty }(I_{m,n}^{F_{k+1}}(u)-I_{m,n}^{F_{k}}(u)),  \label{4.8-1}
\end{equation}%
and by noting that
\begin{equation*}
\liminf_{m\rightarrow \infty }(I_{m,n}^{F_{k+1}}(u)-I_{m,n}^{F_{k}}(u))=I_{\infty
,n}^{F_{k+1}}(u)-I_{\infty ,n}^{F_{k}}(u),
\end{equation*}%
our conclusion directly follows from Fatou's Lemma.
\end{proof}

We are now in the position to prove the first estimate lemma.

\begin{lemma}
\label{lem6.2} Let $K^{F}$ be a fractal glue-up. If $\sigma >\alpha /p$,
then for all $u\in B_{p,\infty }^{\sigma ,F}$,
\begin{equation}
I_{\infty ,n}^{F}(u)\leq C\rho ^{2n\alpha }\sum_{k=n}^{\infty
}E_{k}^{(p),F}(u)\leq C^{\prime }\rho ^{n(p\sigma +\alpha )}\sup_{k\geq n}%
\mathcal{E}_{k}^{\sigma ,F}(u).  \label{eq6.3}
\end{equation}
\end{lemma}

\begin{proof}
Let $I_{m,n}^{F}(u)$ be defined as in (\ref{201}) ($m>n)$. Since $K^{F}$ is
connected and satisfies the condition $(\mathrm{H})$, $|x-y|\leq C_{H}\rho
^{n}$ implies that\ $x,y$ lie in the same or the neighboring $n$-cells. For
every $w\in W_{n}$ and $x\in f(K_{w})$, if $y\in B(x,C_{H}\rho ^{n})$, then
there exists $(g,\tilde{w})\in F\times W_{n}$ such that $y\in g(K_{\tilde{w}%
})$ and
\begin{equation*}
g(K_{\tilde{w}})\cap f(K_{w})\neq \emptyset .
\end{equation*}%
Therefore,
\begin{align*}
I_{m,n}^{F}(u)\leq & \sum_{|w|=n}\sum_{f\in F}\int_{f(K_{w})}\left( \sum_{|%
\tilde{w}|=n}\sum_{g\in F}\int_{g(K_{\tilde{w}})}|u(x)-u(y)|^{p}1_{\{g(K_{%
\tilde{w}})\cap f(K_{w})\neq \emptyset \}}d\mu _{m}^{F}(y)\right) d\mu
_{m}^{F}(x) \\
=& \sum_{|w|=n}\sum_{f\in F}\sum_{|\tilde{w}|=n}\sum_{g\in F}\sum_{x\in
f(K_{w}\cap V_{m})}\sum_{y\in g(K_{\tilde{w}}\cap V_{m})}\frac{1_{\{g(K_{%
\tilde{w}})\cap f(K_{w})\neq \emptyset \}}}{|V_{m}|^{2}}|u(x)-u(y)|^{p}.
\end{align*}%
When $g(K_{\tilde{w}})\cap f(K_{w})\neq \emptyset $, we can find a common
vertex of these two cells and denote it by $z\in f(V_{w})\cap g(V_{\tilde{w}%
})$, due to the just-touching property when $f\neq g$ and the p.c.f.
structure when $f=g$. Furthermore, by the inequality $%
|u(x)-u(y)|^{p}\leq 2^{p-1}\left( |u(x)-u(z)|^{p}+|u(z)-u(y)|^{p}\right) $,
\begin{equation*}
1_{\{g(K_{\tilde{w}})\cap f(K_{w})\neq \emptyset \}}|u(x)-u(y)|^{p}\leq
2^{p-1}\sum_{z\in f(V_{w})\cap g(V_{\tilde{w}})}\left(
|u(x)-u(z)|^{p}+|u(z)-u(y)|^{p}\right) .
\end{equation*}%
It follows that
\begin{align}
I_{m,n}^{F}(u)\leq & 2^{p-1}\sum_{|w|=n}\sum_{f\in F}\sum_{|\tilde{w}%
|=n}\sum_{g\in F}\sum_{x\in f(K_{w}\cap V_{m})}\sum_{y\in g(K_{\tilde{w}%
}\cap V_{m})}\sum_{z\in f(V_{w})\cap g(V_{\tilde{w}})}\frac{1}{|V_{m}|^{2}}%
\left( |u(x)-u(z)|^{p}+|u(z)-u(y)|^{p}\right)  \notag \\
\leq & C_{1}N^{-2m}\sum_{|w|=n}\sum_{f\in F}\sum_{|\tilde{w}|=n}\sum_{g\in
F}\sum_{z\in f(V_{w})\cap g(V_{\tilde{w}})}\sum_{x\in f(K_{w}\cap
V_{m})}\sum_{y\in g(K_{\tilde{w}}\cap V_{m})}\left(
|u(x)-u(z)|^{p}+|u(z)-u(y)|^{p}\right)  \notag \\
\leq & C_{1}N^{-2m}\sum_{|w|=n}\sum_{f\in F}\sum_{z\in f(V_{w})}\sum_{x\in
f(K_{w}\cap V_{m})}|u(x)-u(z)|^{p}\left( \sum_{|\tilde{w}|=n}\sum_{g\in
F}1_{\{z\in f(V_{w})\cap g(V_{\tilde{w}})\}}|g(K_{\tilde{w}}\cap
V_{m})|\right)  \notag \\
& +C_{1}N^{-2m}\sum_{|\tilde{w}|=n}\sum_{g\in F}\sum_{z\in g(V_{\tilde{w}%
})}\sum_{y\in g(K_{\tilde{w}}\cap V_{m})}|u(z)-u(y)|^{p}\left(
\sum_{|w|=n}\sum_{f\in F}1_{\{z\in f(V_{w})\cap g(V_{\tilde{w}%
})\}}|f(K_{w}\cap V_{m})|\right)  \notag \\
\leq & C_{2}N^{-(m+n)}\left(\sum_{|w|=n}\sum_{f\in F}\sum_{z\in
f(V_{w})}\sum_{x\in f(K_{w}\cap V_{m})}|u(x)-u(z)|^{p} +\sum_{|\tilde{w}%
|=n}\sum_{g\in F}\sum_{z\in g(V_{\tilde{w}})}\sum_{y\in g(K_{\tilde{w}}\cap
V_{m})}|u(z)-u(y)|^{p} \right)  \notag \\
=& 2C_{2}N^{-(m+n)}\sum_{|w|=n}\sum_{f\in F}\sum_{x\in f(K_{w}\cap
V_{m})}\sum_{z\in f(V_{w})}|u(x)-u(z)|^{p},  \label{eq5.6}
\end{align}%
where we use Proposition \ref{4.1} for each common vertex $%
z\in f(V_{w})\cap g(V_{\tilde{w}})$ in the last inequality, so that
\begin{equation*}
\sum_{|\tilde{w}|=n}\sum_{g\in F}1_{\{z\in f(V_{w})\cap g(V_{\tilde{w}})\}}
\end{equation*}%
is uniformly bounded for any $(f,w)\in F\times W_{n}$, and the
estimates
\begin{eqnarray}
|V_{m}| &\asymp &N^{m}=\rho ^{-\alpha m},  \label{v1} \\
\text{\ }|K_{w}\cap V_{m}| &\asymp &N^{m-n}\text{ for every }w\in W_{n}.
\label{v2}
\end{eqnarray}%
For $x\in f(V_{m}),z\in f(V_{n})$, one
can naturally fix a decreasing sequence of cells $K_{w_{k}}$ with $|w_{k}|=k$
for $k=n,\cdots ,m$ and vertices $x_{k}(x,z)\in V_{w_{k}}$ such that, $%
z=x_{n}(x,z),\ x=x_{m}(x,z)$, so by H\"{o}lder's inequality we have
\begin{align}
|u(z)-u(x)|^{p}\leq & \left( \sum_{k=n}^{m-1}N^{(n-k)q/p}\right)
^{p/q}\left(
\sum_{k=n}^{m-1}N^{k-n}|u(x_{k}(x,z))-u(x_{k+1}(x,z))|^{p}\right)  \notag \\
\leq & C_{3}\sum_{k=n}^{m-1}N^{k-n}|u(x_{k}(x,z))-u(x_{k+1}(x,z))|^{p},
\label{u_xz}
\end{align}%
where $q=p/(p-1)$. Note that for every pair $(a,b)\in f(K_{w^{\prime }}\cap
V_{k})\times f(K_{w^{\prime }}\cap V_{k+1})$, the cardinality of $(x,z)\in
f(K_{w}\cap V_{m})\times f(V_{w})$ with $(x_{k}(x,z),x_{k+1}(x,z))=(a,b)$ is
no greater than $C^{\prime }N^{m-k}$ for some $C^{\prime }>0$, due to the
p.c.f. structure. Plugging \eqref{u_xz} into \eqref{eq5.6}, we obtain
\begin{align}
I_{m,n}^{F}(u)& \leq C_{4}N^{-(m+n)}\sum_{|w|=n}\sum_{f\in F}\sum_{(x,z)\in
f(K_{w}\cap V_{m})\times
f(V_{w})}\sum_{k=n}^{m-1}N^{k-n}|u(x_{k}(x,z))-u(x_{k+1}(x,z))|^{p}  \notag
\\
& \leq C_{4}N^{-(m+n)}\sum_{|w|=n}\sum_{k=n}^{m-1}\sum_{f\in F}\sum_{\QATOP{%
|w^{\prime }|=k}{f(K_{w^{\prime }})\subset f(K_{w})}}\sum_{\QATOP{(x,z)\in
f(K_{w}\cap V_{m})\times f(V_{w}),}{(a,b)\in f(K_{w^{\prime }}\cap
V_{k})\times f(K_{w^{\prime }}\cap V_{k+1})}}N^{k-n}|u(a)-u(b)|^{p}  \notag
\\
& \leq C_{5}N^{-(m+n)}\sum_{|w|=n}\sum_{k=n}^{m-1}\sum_{f\in F}\sum_{\QATOP{%
|w^{\prime }|=k}{f(K_{w^{\prime }})\subset f(K_{w})}}\sum_{a,b\in
f(K_{w^{\prime }}\cap V_{k+1})}N^{k-n}\cdot N^{m-k}|u(a)-u(b)|^{p}  \notag \\
& =C_{5}\rho ^{2n\alpha }\sum_{k=n}^{m-1}\sum_{|w|=n}\sum_{f\in F}\sum_{%
\QATOP{|w^{\prime }|=k}{f(K_{w^{\prime }})\subset f(K_{w})}}\sum_{a,b\in
f(K_{w^{\prime }}\cap V_{k+1})}|u(a)-u(b)|^{p}  \notag \\
& \leq C_{6}\rho ^{2n\alpha }\sum_{k=n}^{m-1}\sum_{f\in
F}E_{k+1}^{(p)}(u\circ f)\leq C_{6}\rho ^{2n\alpha
}\sum_{k=n}^{m}E_{k}^{(p),F}(u),  \label{I_mn}
\end{align}%
where we use \eqref{p_energy} in the last line. It follows that
\begin{equation*}
I_{m,n}^{F}(u)\leq C_{6}\rho ^{2n\alpha }\rho ^{n(p\sigma -\alpha
)}\sum_{k=n}^{m}\rho ^{-n(p\sigma -\alpha )}E_{k}^{(p),F}(u)\leq C_{7}\rho
^{n(p\sigma +\alpha )}\sup_{k\geq n}\mathcal{E}_{k}^{\sigma ,F}(u),
\end{equation*}%
which ends the proof by letting $m\rightarrow \infty $ and using Lemma \ref%
{lem4.4}.
\end{proof}

The following simple proposition is required.

\begin{proposition}
\label{prop:I}For all $u\in B_{p,\infty }^{\sigma ,F}$, we have
\begin{eqnarray}
\liminf_{n\rightarrow \infty }\rho ^{-n(p\sigma +\alpha )}I_{\infty
,n}^{F}(u) &\geq &\rho ^{p\sigma +\alpha }C_{H}^{p\sigma +\alpha
}\liminf_{n\rightarrow \infty }\Phi _{u}^{\sigma ,F}(\rho ^{n}),
\label{II-1} \\
\limsup_{n\rightarrow \infty }\rho ^{-n(p\sigma +\alpha )}I_{\infty
,n}^{F}(u) &\geq &\rho ^{p\sigma +\alpha }C_{H}^{p\sigma +\alpha
}\limsup_{n\rightarrow \infty }\Phi _{u}^{\sigma ,F}(\rho ^{n}).
\label{II-3}
\end{eqnarray}
\end{proposition}

\begin{proof}
Since $C_{H}\in (0,1)$, assume that $\rho ^{s+1}\leq C_{H}<\rho ^{s}$ for
some nonnegative integer $s$. Then, for any integer $n$,
\begin{eqnarray}
\rho ^{-n(p\sigma +\alpha )}I_{\infty ,n}^{F}(u) &=&\rho ^{-n(p\sigma
+\alpha )}\int_{K^{F}}\int_{B(x,C_{H}\rho ^{n})}|u(x)-u(y)|^{p}d\mu
^{F}(y)d\mu ^{F}(x)  \notag \\
&\geq &\rho ^{-n(p\sigma +\alpha )}\int_{K^{F}}\int_{B(x,\rho
^{n+s+1})}|u(x)-u(y)|^{p}d\mu ^{F}(y)d\mu ^{F}(x)  \label{II-2} \\
&=&\rho ^{(s+1)(p\sigma +\alpha )}\rho ^{-(n+s+1)(p\sigma +\alpha
)}\int_{K^{F}}\int_{B(x,\rho ^{n+s+1})}|u(x)-u(y)|^{p}d\mu ^{F}(y)d\mu
^{F}(x)  \notag \\
&>&\rho ^{p\sigma +\alpha }C_{H}^{p\sigma +\alpha }\rho ^{-(n+s+1)(p\sigma
+\alpha )}\int_{K^{F}}\int_{B(x,\rho ^{n+s+1})}|u(x)-u(y)|^{p}d\mu
^{F}(y)d\mu ^{F}(x).  \notag
\end{eqnarray}%
Taking $\liminf_{n\rightarrow \infty }$ and $\limsup_{n\rightarrow \infty }$
on both sides, we obtain (\ref{II-1}) and (\ref{II-3}).
\end{proof}

Denote
\begin{equation*}
I_{n}^{F}(u):=\int_{K^{F}}\int_{\{\rho ^{n+1}\leq d(x,y)<\rho
^{n}\}}|u(x)-u(y)|^{p}d\mu ^{F}(y)d\mu ^{F}(x).
\end{equation*}

\begin{lemma}
\label{lem6.3} Let $K^{F}$ be a fractal glue-up. If $\sigma >\alpha /p$,
then for any $\delta \in (0,p\sigma -\alpha )$ and $u\in B_{p,\infty
}^{\sigma ,F}$,
\begin{equation}
\mathcal{E}_{n}^{\sigma ,F}(u)\leq C\sum_{k=0}^{\infty }\rho ^{(-2\alpha
-\delta )k}\rho ^{-(p\sigma +\alpha )n}I_{k+n}^{F}\leq C^{\prime
}\sup_{k\geq 0}\Phi _{u}^{\sigma ,F}(\rho ^{n+k}).  \label{eq6.4}
\end{equation}
\end{lemma}

\begin{proof}
For $a,b\in f(K_{w})$, we have $|u(a)-u(b)|^{p}\leq
2^{p-1}(|u(a)-u(x)|^{p}+|u(x)-u(b)|^{p})$, where $x\in f(K_{w})$ and $f\in F$%
. Approximate $F$ with increasing finite sets $F_{j}$ so that $%
F=\lim_{j\rightarrow \infty }F_{j}$. Integrating with respect to $x$ and
dividing by $\mu ^{F}(f(K_{w}))$, we have
\begin{align}
E_{n}^{(p),F_{j}}(u)& =\sum_{f\in F_{j}}\sum_{a,b\in
f(V_{w}),|w|=n}|u(a)-u(b)|^{p}  \notag \\
& \leq 2^{p-1}\sum_{f\in F_{j}}\sum_{a,b\in f(V_{w}),a\neq b,|w|=n}\left(
\frac{1}{\mu ^{F}(f(K_{w}))}\int_{f(K_{w})}|u(a)-u(x)|^{p}+|u(x)-u(b)|^{p}d%
\mu ^{F}(x)\right)  \notag \\
& \leq 2^{p-1}|V_{0}|\sum_{f\in F_{j}}\sum_{a\in f(V_{w}),|w|=n}\frac{1}{\mu
^{F}(f(K_{w}))}\int_{f(K_{w})}|u(a)-u(x)|^{p}d\mu ^{F}(x).  \label{lem3.0}
\end{align}%
For every $a\in f(V_{w})$ with $|w|=n$, we can fix a decreasing sequence of
cells $\{f(K_{w_{k}})\}_{k=n}^{m}$ where $|w_{k}|=k$ for all large enough $%
m>n$, such that $a\in \cap _{k=n}^{m}f(K_{w_{k}})$ with $w_{n}=w$. We choose
$x_{k}\in f(K_{w_{k}})$ for $k=n,n+1,...m$. Let $\delta \in (0,p\sigma
-\alpha )$. By H\"{o}lder's inequality,
\begin{align}
& |u(a)-u(x_{n})|^{p}  \notag \\
\leq & 2^{p-1}|u(a)-u(x_{m})|^{p}+2^{p-1}\left( \sum_{k=n}^{m-1}\rho
^{\delta (k-n)q/p}\right) ^{p/q}\left( \sum_{k=n}^{m-1}\rho ^{\delta
(n-k)}|u(x_{k})-u(x_{k+1})|^{p}\right) ,  \label{203}
\end{align}%
where $q=p/(p-1)$. Integrating \eqref{203} with respect to $x_{k}\in
K_{w_{k}}$ and dividing by $\mu ^{F}(f(K_{w_{k}}))$,
\begin{align*}
E_{n}^{(p),F_{j}}(u)\leq & C_{1}\sum_{f\in F_{j}}\sum_{a\in f(V_{w}),|w|=n}%
\frac{2^{p-1}}{\mu ^{F}(f(K_{w_{m}}))}%
\int_{f(K_{w_{m}})}|u(a)-u(x_{m})|^{p}d\mu ^{F}(x_{m}) \\
& +C_{2}\sum_{f\in F_{j}}\sum_{k=n}^{m-1}\frac{\rho ^{\delta (n-k)}}{\mu
^{F}(f(K_{w_{k}}))\mu ^{F}(f(K_{w_{k+1}}))}\int_{f(K_{w_{k}})}%
\int_{f(K_{w_{k+1}})}|u(x_{k})-u(x_{k+1})|^{p}d\mu ^{F}(x_{k+1})d\mu
^{F}(x_{k}).
\end{align*}%
Using Lemma \ref{lemmaMR}, (\ref{BC}) and (\ref{v1}), the first term of the
right-hand side above vanishes as $m\rightarrow \infty $ since
\begin{align*}
& \sum_{a\in f(V_{w}),|w|=n}\frac{2^{p-1}}{\mu ^{F}(f(K_{w_{m}}))}%
\int_{f(K_{w_{m}})}|u(a)-u(x_{m})|^{p}d\mu ^{F}(x_{m}) \\
\leq & C|f(V_{n})|2^{p-1}\rho ^{(p\sigma -\alpha )m}[u]_{B_{p,\infty
}^{\sigma }(K^{F})}^{p}\leq C^{\prime }|f(V_{n})|2^{p-1}\rho ^{(p\sigma
-\alpha )m}[u]_{B_{p,\infty }^{\sigma ,F}}^{p} \\
\leq & C_{3}\rho ^{(p\sigma -\alpha )m-\alpha n}[u]_{B_{p,\infty }^{\sigma
,F}}^{p}\rightarrow 0.
\end{align*}%
Letting $m\rightarrow \infty $, we have
\begin{align}
E_{n}^{(p),F_{j}}(u)& \leq C_{4}\sum_{k=n}^{\infty }\rho ^{\delta (n-k)}\rho
^{-2\alpha k}\int_{K^{F}}\int_{B(x,\rho ^{k})}|u(x)-u(y)|^{p}d\mu
^{F}(y)d\mu ^{F}(x)  \notag \\
& =C_{4}\sum_{k=n}^{\infty }\rho ^{\delta (n-k)}\rho ^{-2\alpha
k}\sum_{l=k}^{\infty }I_{l}^{F}(u)=C_{4}\sum_{k=0}^{\infty }\rho ^{-\delta
k}\rho ^{-2\alpha (n+k)}\sum_{l=k}^{\infty }I_{l+n}^{F}(u)  \notag \\
& =C_{4}\sum_{l=0}^{\infty }\left( \sum_{k=0}^{l}\rho ^{-\delta k}\rho
^{-2\alpha (n+k)}\right) I_{l+n}^{F}(u)=C_{4}\sum_{l=0}^{\infty }\left(
\sum_{k=0}^{l}\rho ^{\delta (l-k)}\rho ^{2\alpha (l-k)}\right) \rho
^{-\delta l}\rho ^{-2\alpha (l+n)}I_{l+n}^{F}(u)  \notag \\
& \leq C_{5}\sum_{l=0}^{\infty }\rho ^{-\delta l}\rho ^{-2\alpha
(l+n)}I_{l+n}^{F}(u)=C_{5}\sum_{k=0}^{\infty }\rho ^{-\delta k}\rho
^{-2\alpha (k+n)}I_{k+n}^{F}(u).  \label{I_2}
\end{align}%
Letting $j\rightarrow \infty $, we have%
\begin{eqnarray*}
\mathcal{E}_{n}^{\sigma ,F}(u) &=&\rho ^{-(p\sigma -\alpha
)n}E_{n}^{(p),F}(u)\leq C_{5}\sum_{k=0}^{\infty }\rho ^{(-2\alpha -\delta
)k}\rho ^{-(p\sigma +\alpha )n}I_{k+n}^{F}(u) \\
&\leq &C_{5}\sum_{k=0}^{\infty }\rho ^{(-2\alpha -\delta )k}\rho ^{-(p\sigma
+\alpha )n}\int_{K^{F}}\int_{B(x,\rho ^{n+k})}|u(x)-u(y)|^{p}d\mu
^{F}(y)d\mu ^{F}(x) \\
&\leq &C_{5}\sum_{k=0}^{\infty }\rho ^{(p\sigma +\alpha -2\alpha -\delta
)k}\sup_{k\geq 0}\Phi _{u}^{\sigma ,F}(\rho ^{n+k}) \\
&=&C_{5}\left( \sum_{k=0}^{\infty }\rho ^{(p\sigma -\alpha -\delta
)k}\right) \sup_{k\geq 0}\Phi _{u}^{\sigma ,F}(\rho ^{n+k})=\frac{C_{5}}{%
1-\rho ^{p\sigma -\alpha -\delta }}\sup_{k\geq 0}\Phi _{u}^{\sigma ,F}(\rho
^{n+k}),
\end{eqnarray*}%
where the second inequality follows from enlarging the integral region.
\end{proof}

We need the following proposition to manage the transitions between
different levels of vertex energies.

\begin{proposition}
Let $K$ be a connected homogeneous p.c.f. self-similar set, then there
exists $C>0$ such that for all $u\in L^{p}(K,\mu )$ and nonnegative integer $%
n$,
\begin{equation*}
E_{n}^{(p)}(u)\leq CE_{n+1}^{(p)}(u).
\end{equation*}%
Moreover, there exists $C>0$ such that for all $u\in L^{p}(K^{F},\mu ^{F})$,
\begin{equation}
E_{n}^{(p),F}(u)\leq CE_{n+1}^{(p),F}(u).  \label{E-n}
\end{equation}
\end{proposition}

\begin{proof}
We start with the basic case $n=0$. For any $x,y\in V_{0}$, we can fix a
vertex-disjoint path $x=z_{1}(x,y),z_{2}(x,y),\cdots ,z_{k}(x,y)=y$ such
that $z_{j}(x,y)$ and $z_{j+1}(x,y)$ ($1\leq j\leq k$) belong to the same $%
V_{w}$ for some $w\in W_{1}$, and $2\leq k\leq |V_{1}|$. Using Jensen's
inequality,
\begin{align*}
|u(x)-u(y)|^{p}& \leq
(k-1)^{p-1}\sum_{j=1}^{k-1}|u(z_{j}(x,y))-u(z_{j+1}(x,y))|^{p} \\
& \leq |V_{1}|^{p-1}\sum_{j=1}^{k-1}|u(z_{j}(x,y))-u(z_{j+1}(x,y))|^{p}.
\end{align*}%
It follows that
\begin{align*}
E_{0}^{(p)}(u)& \leq \sum_{x,y\in
V_{0}}|V_{1}|^{p-1}\sum_{j=1}^{k-1}|u(z_{j}(x,y))-u(z_{j+1}(x,y))|^{p} \\
& \leq |V_{0}|^{2}|V_{1}|^{p-1}\sum_{a,b\in V_{w},|w|=1}|u(a)-u(b)|^{p} \\
& =|V_{0}|^{2}|V_{1}|^{p-1}E_{1}^{(p)}(u),
\end{align*}%
since for each pair $(a,b)\in V_{w}$ with $|w|=1$, there exists at most one $%
j\in \{1,2,...k\}$ (depending on $(x,y)$) with $%
(z_{j}(x,y),z_{j+1}(x,y))=(a,b)$ for any $x,y\in V_{0}$. Thus, we show the
desired for $n=0$ with $C:=|V_{0}|^{2}|V_{1}|^{p-1}$ for any $u$. The rest
simply follows from that
\begin{align*}
E_{n}^{(p)}(u)& =\sum_{x,y\in
V_{w},|w|=n}|u(x)-u(y)|^{p}=\sum_{|w|=n}E_{0}^{(p)}(u\circ \phi _{w}) \\
& \leq C\sum_{|w|=n}E_{1}^{(p)}(u\circ \phi _{w})=C\sum_{a,b\in
V_{wi},|w|=n,|i|=1}|u(a)-u(b)|^{p}=CE_{n+1}^{(p)}(u).
\end{align*}%
The second inclusion follows from the definition of $K^{F}$.
\end{proof}

In \cite[Theorem 1.4]{GaoYuZhang2022PA}, for a homogeneous p.c.f.
self-similar set $K$, we established the equivalence
\begin{equation}
\sup\limits_{n\geq 0}\rho ^{-n(p\sigma -\alpha )}E_{n}^{(p)}(u)\asymp
\lbrack u]_{B_{p,\infty }^{\sigma }(K)}^{p}\text{ \ (}p\sigma >\alpha \text{)%
},  \label{204}
\end{equation}%
where $E_{n}^{(p)}(u)$ is from (\ref{EE}).

Using Lemma \ref{lem6.2} and Lemma \ref{lem6.3}, we immediately derive such
equivalence for $K^{F}$. Define the discrete local $p$-energy for $K^{F}$ by
\begin{equation*}
\mathcal{E}_{p,\infty }^{\sigma ,F}(u):=\sup\limits_{n\geq 0}\rho
^{-n(p\sigma -\alpha )}E_{n}^{(p),F}(u)=\sup_{n\geq 0}\mathcal{E}%
_{n}^{\sigma ,F}(u).
\end{equation*}

\begin{corollary}
\label{lem6.1} Let $K^{F}$ be a fractal glue-up. If $\sigma>\alpha /p$, then
for all $u\in B_{p,\infty }^{\sigma ,F}$,
\begin{equation}
\limsup_{n\rightarrow \infty }\mathcal{E}_{n}^{\sigma ,F}(u)\asymp
\limsup_{n\rightarrow \infty }\Phi _{u}^{\sigma ,F}(\rho ^{n}),\ \mathcal{E}%
_{p,\infty }^{\sigma ,F}(u)\asymp \lbrack u]_{B_{p,\infty}^{\sigma ,F}}^{p}.
\label{eq6.2}
\end{equation}
\end{corollary}

\begin{proof}
Taking limsup and sup (of $n$) on both sides of \eqref{eq6.3} and %
\eqref{eq6.4}, we have
\begin{eqnarray}
\limsup_{n\rightarrow \infty }\rho ^{-n(p\sigma +\alpha )}I_{\infty
,n}^{F}(u) &\leq &C\limsup_{n\rightarrow \infty }\mathcal{E}_{n}^{\sigma
,F}(u)\leq C^{\prime }\limsup_{n\rightarrow \infty }\Phi _{u}^{\sigma
,F}(\rho ^{n}),  \label{C6.1-1} \\
\sup\limits_{n\geq 0}\rho ^{-n(p\sigma +\alpha )}I_{\infty ,n}^{F}(u) &\leq
&C\sup\limits_{n\geq 0}\mathcal{E}_{n}^{\sigma ,F}(u).  \label{C6.1-2}
\end{eqnarray}

By (\ref{II-3}), we have%
\begin{equation*}
\limsup_{n\rightarrow \infty }\rho ^{-n(p\sigma +\alpha )}I_{\infty
,n}^{F}(u)\geq \rho ^{p\sigma +\alpha }C_{H}^{p\sigma +\alpha
}\limsup_{n\rightarrow \infty }\Phi _{u}^{\sigma ,F}(\rho ^{n}),
\end{equation*}%
thus showing
\begin{equation*}
\limsup_{n\rightarrow \infty }\mathcal{E}_{n}^{\sigma ,F}(u)\asymp
\limsup_{n\rightarrow \infty }\Phi _{u}^{\sigma ,F}(\rho ^{n})
\end{equation*}%
by (\ref{C6.1-1}). On the other hand, we have from (\ref{ct-1}) and (\ref{ct}%
) that
\begin{equation}
\sup\limits_{n\geq 0}\rho ^{-n(p\sigma +\alpha )}I_{\infty
,n}^{F}(u)=C_{H}^{p\sigma +\alpha }\sup\limits_{n\geq 0}\Phi _{u}^{\sigma
,F}(C_{H}\rho ^{n})=C_{H}^{p\sigma +\alpha }[u]_{B_{p,\infty }^{\sigma
,F}}^{p}  \label{C6.1-0}
\end{equation}

Since $C_{H}\in (0,1)$, assume that $\rho ^{s+1}\leq C_{H}<\rho ^{s}$ for
some nonnegative integer $s$. Replacing $n$ by $n+s+1$ in (\ref{eq6.4}), we
obtain%
\begin{equation}
\mathcal{E}_{n+s+1}^{\sigma ,F}(u)\leq C\sup_{k\geq 0}\Phi _{u}^{\sigma
,F}(\rho ^{n+s+1+k})\leq C^{\prime }\sup_{k\geq 0}\Phi _{u}^{\sigma
,F}(C_{H}\rho ^{n+k}).  \label{C6.1-3}
\end{equation}
By (\ref{E-n}),%
\begin{equation*}
\mathcal{E}_{n+s+1}^{\sigma ,F}(u)=\rho ^{-(n+s+1)(p\sigma -\alpha
)}E_{n+s+1}^{(p),F}(u)\geq C^{-s-1}\rho ^{-(n+s+1)(p\sigma -\alpha
)}E_{n}^{(p),F}(u)=C^{-s-1}\rho ^{-(s+1)(p\sigma -\alpha )}\mathcal{E}%
_{n}^{\sigma ,F}(u),
\end{equation*}
thus by (\ref{C6.1-3}),
\begin{equation}
\mathcal{E}_{n}^{\sigma ,F}(u)\leq C^{\prime \prime }C_{H}^{p\sigma -\alpha
}\sup_{k\geq 0}\Phi _{u}^{\sigma ,F}(C_{H}\rho ^{n+k}).  \label{C6.1-4}
\end{equation}

We conclude from (\ref{C6.1-2}), (\ref{C6.1-0}) and (\ref{C6.1-4}) that%
\begin{equation*}
\mathcal{E}_{p,\infty }^{\sigma ,F}(u)\asymp \lbrack u]_{B_{p,\infty
}^{\sigma ,F}}^{p}.
\end{equation*}

The proof is complete.
\end{proof}

\subsection{Property (E) holds for nested fractals}

\label{subsec5.2}For a set $V$, let $l (V)=\{u:$ $u$ maps $V$ into $%
\mathbb{R}\}$.

\begin{definition}
(\cite[Definition 3.1]{GaoYuZhang2022PA}) \label{dfE} We say that a
connected homogeneous p.c.f. self-similar set $K$ satisfies \textrm{property
(E)}, if there exist $\sigma>\alpha/p$ and a positive constant $C$ such that

\begin{itemize}
\item[(i)~] for any $u\in B_{p,\infty }^{\sigma}$ and for all $n\geq 0$, $%
\mathcal{E}_{0}^{\sigma}(u)\leq C\mathcal{E}_{n}^{\sigma}(u)$,

\item[(ii)] for any $u\in \ell (V_{0})$, there exists an extension $\tilde{u}%
\in B_{p,\infty }^{\sigma}$.
\end{itemize}
\end{definition}

We know from \cite[Remark 3.2]{GaoYuZhang2022PA} that, if property (E) holds with $\sigma$, then $\sigma=\sigma _{p}^{*}(K)$, and for all $n\geq m$, $\mathcal{E}_{m}^{\sigma
_{p}^{*}}(u)\leq C\mathcal{E}_{n}^{\sigma _{p}^{*}}(u)$. Thus we omit the index $\sigma$. The following proposition shows the importance of property (E).

\begin{proposition}
\label{prop:critical} If a connected homogeneous p.c.f. self-similar set $K$
satisfies property (E), then
\begin{equation}
\sigma _{p}^{\ast }(K)=\sigma _{p}^{\#}(K)=\sigma _{p}^{\#}(K^{F}).
\label{qnmd}
\end{equation}%
Moreover, $K^{F}$ satisfies property $(VE)_{\sigma _{p}^{\#}(K)}$, and $%
B_{p,\infty }^{\sigma _{p}^{\#}(K),F}$ contains non-constant functions.
\end{proposition}

\begin{proof}
The fact $\sigma _{p}^{\ast }(K)=\sigma _{p}^{\#}(K)$ follows from \cite[the
proof of Proposition 3.4]{GaoYuZhang2022PA}. It is also clear that $\sigma
_{p}^{\#}(K)\geq \sigma _{p}^{\#}(K^{F})$, since any non-constant function $%
u\in l(K^{F})$ in $B_{p,\infty }^{\sigma ,F}$ naturally gives a non-constant
function in $B_{p,\infty }^{\sigma }(K)$ for any $\sigma >0$: just restrict $%
u$ to a tile $f(K)$ where it is non-constant (for some $f\in F$).

For the reverse inclusion $\sigma _{p}^{\#}(K)\leq \sigma _{p}^{\#}(K^{F})$
to hold, we use property (E)(ii) to find a non-constant
function in $B_{p,\infty }^{\sigma _{p}^{\#}(K),F}$. To see this, we pick a
contraction $f\in F$, and determine the values of $u$ on $f(V_{1})$ by $%
u(f(V_{0}))=0$ while $u(f(V_{1}\setminus V_{0}))=1$. We can then extend $u$
from $f(V_{1})$ to $f(K)$ and make sure that $u\circ f\in B_{p,\infty
}^{\sigma _{p}^{\#}(K)}$ by property (E)(ii), and simply define $u=0$
outside of $f(K)$ on $K^{F}$. Then $u$ is an extended non-trivial function
in $B_{p,\infty }^{\sigma _{p}^{\#}(K),F}$ since by (\ref{eq6.2}) and (\ref%
{204}),
\begin{equation*}
\lbrack u]_{B_{p,\infty }^{\sigma _{p}^{\#}(K),F}}^{p}\asymp \mathcal{E}%
_{p,\infty }^{\sigma _{p}^{\#}(K),F}(u)=\mathcal{E}_{p,\infty }^{\sigma
_{p}^{\#}(K)}(u\circ f)\asymp \lbrack u\circ f]_{B_{p,\infty }^{\sigma
_{p}^{\#}(K)}(K)}^{p}<\infty.
\end{equation*}
Thus we show \eqref{qnmd} and that $B_{p,\infty }^{\sigma _{p}^{\#}(K),F}$
contains non-constant functions.

Finally, since $K$ satisfies property (E), that is for all $n$,
\begin{equation*}
\mathcal{E}_{n}^{\sigma _{p}^{\#}}(u)\leq C\liminf_{k\rightarrow \infty }%
\mathcal{E}_{k}^{\sigma _{p}^{\#}}(u),
\end{equation*}%
summing over $F$ gives that for all $n$,
\begin{equation*}
\mathcal{E}_{n}^{\sigma _{p}^{\#},F}(u)\leq C\sum_{f\in
F}\liminf_{k\rightarrow \infty }\mathcal{E}_{k}^{\sigma _{p}^{\#}}(u\circ
f)\leq C\liminf_{k\rightarrow \infty }\sum_{f\in F}\mathcal{E}_{k}^{\sigma
_{p}^{\#}}(u\circ f)=C\liminf_{k\rightarrow \infty }\mathcal{E}_{k}^{\sigma
_{p}^{\#},F}(u),
\end{equation*}%
where we use Fatou's lemma to obtain the desired.
\end{proof}

When property (E) fails, $\sigma _{p}^{\#}=\sigma _{p}^{\ast }$ does not
always hold for connected homogeneous p.c.f. self-similar sets (see \cite%
{GuLau.2020.TAMS}). Moreover, we mention in passing that, the critical
domain $B_{p,\infty }^{\sigma _{p}^{\ast }}$ can be trivial for some metric
measure spaces (bounded or unbounded).

We present the following key lemma for \emph{nested fractals}, a class of
connected homogeneous p.c.f. self-similar sets defined in \cite%
{Kumagai.1993.PTaRF205} satisfying conditions $(A$-0$)\sim (A $-3$)$ and $%
|V_{0}|\geq 2$ therein.

%
%
%
%
%
%
%
%
%
%

\begin{lemma}
\label{lemma:E}For a nested fractal, property (E) holds and $\sigma_{p}^{\#}>\alpha/p$ for all $1<p<\infty$. \label{yyqx}
\end{lemma}

\begin{proof}
In this proof, we use the same terminology and notions as in \cite%
{Caoqiugu2022adv}. Let $K$ be a nested fractal. We fix the components $r_{i}$
of $\mathbf{r}$ in \cite[Theorem 6.3]{Caoqiugu2022adv} to be the same number
$r$, so it is `$\mathscr{G}$-symmetric' (see \cite[Section 3]%
{Caoqiugu2022adv} for definition). Then \cite[Theorem 6.3]{Caoqiugu2022adv}
states that, `condition $\left( \mathbf{A}\right) $' (see \cite[the
beginning of Section 4]{Caoqiugu2022adv} for definition) holds for affine
nested fractals, including $K$. By multiplying
$\mathbf{r}$ with a constant (still denoted by $\mathbf{r}$), \cite[Theorem
4.2]{Caoqiugu2022adv} guarantees that
\begin{equation}
\mathcal{T}E=E  \label{TE}
\end{equation}%
(see \cite[Definition 2.8, Definition 3.1]{Caoqiugu2022adv} for related
definitions), thus showing `condition $\left( \mathbf{A^{\prime }}\right) $'
(see \cite[the beginning of Section 5]{Caoqiugu2022adv}). So by \cite[Lemma
5.4]{Caoqiugu2022adv}, we can fix
\begin{equation}
r<1  \label{rp1}
\end{equation}%
to further satisfy condition $\left( \mathbf{A^{\prime }}\right) $ on $K$.

The equivalence of $E$ and $E_{0}^{(p)}$ (for functions $u\in l(V_{0})$, see
the statement of \cite[Theorem 5.1]{Caoqiugu2022adv}) is stated in the proof
of \cite[Proposition 5.3 (b)]{Caoqiugu2022adv}, which implies that
\begin{equation}
\Lambda ^{n}E(u)\asymp \Lambda ^{n}E_{0}^{(p)}(u)=r^{-n}E_{n}^{(p)}(u)
\label{205}
\end{equation}%
for all $n\geq 0$ and all functions $u\in l(V_{n})$ by the definition of $%
\Lambda $ in \cite[Definition 3.1]{Caoqiugu2022adv} and $E_{n}^{(p)}$ given in \eqref{EE}.

Now fix
\begin{equation}
\sigma =\frac{\log _{\rho }r+\alpha }{p},  \label{s1}
\end{equation}%
so that $r=\rho ^{p\sigma -\alpha }$.

Property (E)(i) holds with this $\sigma $, by using the monotonicity property in \cite[Proposition 5.3
(a)]{Caoqiugu2022adv} and (\ref{205}) that
\begin{equation}
E_{0}^{(p)}(u)\asymp E(u)\leq \Lambda E(u)\leq \cdots \leq \Lambda
^{n}E(u)\asymp r^{-n}E_{n}^{(p)}(u)=\rho ^{-n(p\sigma -\alpha
)}E_{n}^{(p)}(u)=\mathcal{E}_{n}^{\sigma}(u).  \label{EE-1}
\end{equation}%

To see Property (E)(ii), we show the existence of piecewise harmonic functions (see definition in
\cite[Section 5.1]{Caoqiugu2022adv}) using a standard iterative
construction as in \cite{HermanPeironeStrichartz.2004.PA125}. The iteration part is that, for any boundary value of a cell $%
u_{w}\in l(f_{w}(V_{0}))$, by (\ref{TE}), we can take its next-level
harmonic extension $\tilde{u}_{w}\in l(f_{w}(V_{1}))$ such that $\Lambda E(%
\tilde{u}_{w})=E(u_{w})$ and $\tilde{u}_{w}|_{f_{w}(V_{0})}=u_{w}$. The
iteration starts from $V_{0}$ to $V_{1}$, and from each level-1 cell to
level-2 cells in the above way, then finally to $V_{\ast }$. Such an
extension satisfies the definition of piecewise harmonic functions in \cite[%
Section 5.1]{Caoqiugu2022adv}, so we know by \cite[Section 5.2]%
{Caoqiugu2022adv} that it embeds into $C(K)$ and this gives an
extension to $K$ instead of $V_{\ast }$.

Therefore, we conclude from above that $K$ satisfies property (E) with $
\sigma $ defined as in (\ref{s1}). By (\ref{qnmd}), we have $\sigma =\sigma
_{p}^{\#}$, thus
\begin{equation}
r=\rho ^{p\sigma -\alpha }=\rho ^{p\sigma _{p}^{\#}-\alpha }.  \label{rr}
\end{equation}
The inclusion $\sigma _{p}^{\#}>\alpha /p$ follows from \eqref{rp1}.
\end{proof}

\begin{remark}
This lemma also indicates that,
$E_{p,\infty }^{\sigma _{p}^{\#}}(u)$,  $[u]_{B_{p,\infty}^{\sigma _{p}^{\#}}}^p$ and the homogeneous discrete $p$-energy in \cite[Theorem 5.1]{Caoqiugu2022adv}
are equivalent for a nested fractal $K$.
It is known from \cite[Theorem 8.18]{Barlow.1998.1} that $K$ admits a heat
kernel satisfying \eqref{hk_F}, so \eqref{eq1.7} holds by Lemma
\ref{thm1}. The critical exponents $\sigma
_{p}^{\#}(K)=\sigma _{p}^{\ast }(K)$ due to property (E) by Proposition %
\ref{prop:critical}. Using $\Lambda ^{n}E_{0}^{(p)}(u)\asymp \Lambda
^{n}E(u) $ in (\ref{205}), the $p$-energy in \cite[Theorem 5.1]%
{Caoqiugu2022adv} with $r_{i}=r$ is equivalent to $\lim_{n\rightarrow \infty
}\Lambda ^{n}E(u)$ (which exists by \cite[Proposition 5.3]{Caoqiugu2022adv}%
), where $r=\rho ^{p\sigma _{p}^{\#}(K)-\alpha }$ is given by (\ref{rr}).
By (\ref{204}) and (\ref{EE-1}),%
\begin{equation*}
[u]_{B_{p,\infty }^{\sigma _{p}^{\#}}}^{p}\asymp \lim_{n\rightarrow \infty
}\Lambda ^{n}E(u).
\end{equation*}
\end{remark}


\subsection{Equivalence between \protect$(VE)_{\protect\sigma}$ and \protect$(NE)_{\protect\sigma}$}

\label{subsec5.3} Our goal is to show the equivalence between $(VE)_{\sigma}$ and
$(NE)_{\sigma}$ for $K^{F}$. It is quite different and more delicate to obtain the lower limit
equivalence, compared with the equivalence in Corollary \ref{lem6.1} using the upper limit.

\begin{lemma}
\label{thm_3} If a fractal glue-up $K^{F}$ satisfies $(\widetilde{VE})_\sigma$ with
$\sigma >\alpha /p$, then there exists $C>0$ such that for all $u\in
B_{p,\infty }^{\sigma ,F}$,
\begin{equation}
\liminf_{n\rightarrow \infty }\mathcal{E}_{n}^{\sigma ,F}(u)\leq
C\liminf_{n\rightarrow \infty }\Phi _{u}^{\sigma ,F}(\rho ^{n}).
\label{limphi2}
\end{equation}
\end{lemma}

\begin{proof}
Since $C_{H}\in (0,1)$, we can assume that $\rho ^{s+1}\leq C_{H}<\rho ^{s}$ for
some $s\in \mathbb{N}$. For all $n>0$, we have by (\ref{II-2}) that%
\begin{equation}
I_{\infty ,n}^{F}(u)\geq \int_{K^{F}}\int_{B(x,\rho
^{n+s+1})}|u(x)-u(y)|^{p}d\mu ^{F}(y)d\mu ^{F}(x)\geq I_{n+s+1}^{F}(u).
\label{I-1}
\end{equation}

Note that there exists a positive integer $n_{0}$ such that
\begin{equation*}
\sup_{k\geq n_{0}}\mathcal{E}_{k}^{\sigma ,F}(u)\leq 2\limsup_{n\rightarrow
\infty }\mathcal{E}_{n}^{\sigma ,F}(u).
\end{equation*}
By (\ref{I-1}), we have for all $n\geq n_{0}+s+1$ that
\begin{eqnarray}
I_{n}^{F}(u) &\leq &I_{\infty ,n-s-1}^{F}(u)  \notag \\
&\leq &C\rho ^{(n-s-1)(p\sigma +\alpha )}\sup_{k\geq n-s-1}\mathcal{E}%
_{k}^{\sigma ,F}(u)\text{ \ \ (using \eqref{eq6.3})}  \notag \\
&\leq &CC_{H}^{-(p\sigma +\alpha )}\rho ^{n(p\sigma +\alpha )}\sup_{k\geq
n_{0}}\mathcal{E}_{k}^{\sigma ,F}(u)\leq 2CC_{H}^{-(p\sigma +\alpha )}\rho
^{n(p\sigma +\alpha )}\limsup_{n\rightarrow \infty }\mathcal{E}_{n}^{\sigma
,F}(u)  \notag \\
&\leq &C_{1}\rho ^{n(p\sigma +\alpha )}\liminf_{n\rightarrow \infty }%
\mathcal{E}_{n}^{\sigma ,F}(u)\text{ \ \ (by property }(\widetilde{VE})_\sigma\text{
)}.  \label{I_n}
\end{eqnarray}%
By Lemma \ref{lem6.3}, for any positive integer $L\geq n_{0}+s+1 $ and $%
0<\delta <p\sigma -\alpha $,
\begin{align*}
\mathcal{E}_{n}^{\sigma ,F}(u)& \leq C_{2}\sum_{k=0}^{\infty }\rho
^{(-2\alpha -\delta )k}\rho ^{-(p\sigma +\alpha )n}I_{k+n}^{F}(u) \\
& = C_{2}\sum_{k=0}^{L}\rho ^{(-2\alpha -\delta )k}\rho ^{-(p\sigma +\alpha
)n}I_{k+n}^{F}(u)+C_{2}\sum_{k=L}^{\infty }\rho ^{(-2\alpha -\delta )k}\rho
^{-(p\sigma +\alpha )n}I_{k+n}^{F}(u) \\
& \leq C_{2}\sum_{k=0}^{L}\rho ^{(-2\alpha -\delta )k}\rho ^{-(p\sigma
+\alpha )n}I_{k+n}^{F}(u)+C_{1}C_{2}\sum_{k=L}^{\infty }\rho ^{(p\sigma
-\alpha -\delta )k}\liminf_{n\rightarrow \infty }\mathcal{E}_{n}^{\sigma
,F}(u), \\
& =C_{2}\left( \sum_{k=0}^{L}\rho ^{(-2\alpha -\delta )k}\rho ^{-(p\sigma
+\alpha )n}I_{k+n}^{F}(u)+\frac{C_{1}\rho ^{(p\sigma -\alpha -\delta )L}}{%
1-\rho ^{(p\sigma -\alpha -\delta )}}\liminf_{n\rightarrow \infty }\mathcal{E%
}_{n}^{\sigma ,F}(u)\right),
\end{align*}%
where we use \eqref{I_n} in the third line. Taking $\liminf_{n\rightarrow
\infty }$ in the right-hand side above, we have
\begin{equation*}
C_{3}\liminf_{n\rightarrow \infty }\mathcal{E}_{n}^{\sigma ,F}(u)\leq
\liminf_{n\rightarrow \infty }\sum_{k=0}^{L}\rho ^{-(2\alpha +\delta )k}\rho
^{-(p\sigma +\alpha )n}I_{k+n}^{F}(u),
\end{equation*}%
where $C_{3}:=\frac{1}{C_{2}}-\frac{C_{1}\rho ^{(p\sigma -\alpha -\delta )L}%
}{1-\rho ^{(p\sigma -\alpha -\delta )}}$.

Fix a large integer $L$ such that $C_{3}>0$, then
\begin{align*}
C_{3}\rho ^{(2\alpha +\delta )L}\liminf_{n\rightarrow \infty }\mathcal{E}%
_{n}^{\sigma ,F}(u)& \leq \rho ^{(2\alpha +\delta )L}\liminf_{n\rightarrow
\infty }\left( \sum_{k=0}^{L}\rho ^{-n(p\sigma +\alpha )}\rho ^{-(2\alpha
+\delta )k}I_{n+k}^{F}(u)\right) \\
& \leq \liminf_{n\rightarrow \infty }\rho ^{-n(p\sigma +\alpha
)}\sum_{k=0}^{L}I_{n+k}^{F}(u)\leq \liminf_{n\rightarrow \infty }\Phi
_{u}^{\sigma ,F}(\rho ^{n}),
\end{align*}%
thus showing \eqref{limphi2}.
\end{proof}

Next, we show the reverse inequality of (\ref{limphi2}) by using (\ref{E-n}) under property $(\widetilde{NE})_{\sigma}$.

\begin{lemma}
\label{thm_4} If property $(\widetilde{NE})_{\sigma}$ holds for a fractal glue-up $%
K^{F}$ with $\sigma >\alpha /p$, then there exists $C>0$ such that for all $%
u\in B_{p,\infty }^{\sigma ,F}$,
\begin{equation}
\liminf_{n\rightarrow \infty }\Phi _{u}^{\sigma ,F}(\rho ^{n})\leq
C\liminf_{n\rightarrow \infty }\mathcal{E}_{n}^{\sigma ,F}(u).
\label{limphi}
\end{equation}
\end{lemma}

\begin{proof}
By \eqref{I_mn}, we have
\begin{align}
\rho ^{-n(p\sigma +\alpha )}I_{m,n}^{F}(u)\leq & C\rho ^{-n(p\sigma -\alpha
)}\sum_{k=n}^{m}\rho ^{k(p\sigma -\alpha )}\rho ^{-k(p\sigma -\alpha
)}E_{k}^{(p),F}(u)  \notag \\
\leq & C\sum_{k=n}^{m}\rho ^{(k-n)(p\sigma -\alpha )}\mathcal{E}_{k}^{\sigma
,F}(u)\leq C\sum_{k=0}^{\infty }\rho ^{k(p\sigma -\alpha )}\mathcal{E}%
_{k+n}^{\sigma ,F}(u).  \label{202}
\end{align}

Note that there exists $n_{0}>0$ such that for all positive integer $L>n_{0}$%
,
\begin{eqnarray*}
\sup_{n\geq L}\mathcal{E}_{n}^{\sigma ,F}(u) &\leq &2\limsup_{n\rightarrow
\infty }\mathcal{E}_{n}^{\sigma ,F}(u) \\
&\asymp &\limsup_{n\rightarrow \infty }\Phi _{u}^{\sigma ,F}(\rho ^{n})\text{
\ (by Corollary \ref{lem6.1})} \\
&\leq &C\liminf_{n\rightarrow \infty }\Phi _{u}^{\sigma ,F}(\rho ^{n})\text{
\ \ (by property }(\widetilde{NE})_{\sigma}\text{)}.
\end{eqnarray*}
It follows from (\ref{202}) that
\begin{eqnarray*}
\rho ^{-n(p\sigma +\alpha )}I_{m,n}^{F}(u) &\leq &C\sum_{k=0}^{\infty }\rho
^{k(p\sigma -\alpha )}\mathcal{E}_{k+n}^{\sigma ,F}(u) \\
&=&C\sum_{k=0}^{L}\rho ^{k(p\sigma -\alpha )}\mathcal{E}_{k+n}^{\sigma
,F}(u)+C\sum_{k=L+1}^{\infty }\rho ^{k(p\sigma -\alpha )}\mathcal{E}%
_{k+n}^{\sigma ,F}(u) \\
&\leq &C\sum_{k=0}^{L}\rho ^{k(p\sigma -\alpha )}\mathcal{E}_{k+n}^{\sigma
,F}(u)+C\sum_{k=L+1}^{\infty }\rho ^{k(p\sigma -\alpha )}\sup_{n\geq L}%
\mathcal{E}_{n}^{\sigma ,F}(u) \\
&\leq &C\sum_{k=0}^{L}\rho ^{k(p\sigma -\alpha )}\mathcal{E}_{k+n}^{\sigma
,F}(u)+C^{\prime }\sum_{k=L+1}^{\infty }\rho ^{k(p\sigma -\alpha
)}\liminf_{n\rightarrow \infty }\Phi _{u}^{\sigma ,F}(\rho ^{n}) \\
&=&C\sum_{k=0}^{L}\rho ^{k(p\sigma -\alpha )}\mathcal{E}_{k+n}^{\sigma
,F}(u)+C^{\prime }\frac{\rho ^{(p\sigma -\alpha )(L+1)}}{1-\rho ^{p\sigma
-\alpha }}\liminf_{n\rightarrow \infty }\Phi _{u}^{\sigma ,F}(\rho ^{n}).
\end{eqnarray*}

Using Lemma \ref{lem4.4}, we have
\begin{align*}
\rho ^{-n(p\sigma +\alpha )}I_{\infty ,n}^{F}(u)\leq & \rho ^{-n(p\sigma
+\alpha )}\liminf_{m\rightarrow \infty }I_{m,n}^{F}(u) \\
\leq & C\sum_{k=0}^{L}\rho ^{k(p\sigma -\alpha )}\mathcal{E}_{k+n}^{\sigma
,F}(u)+\frac{C^{\prime }\rho ^{(p\sigma -\alpha )(L+1)}}{1-\rho ^{p\sigma
-\alpha }}\liminf_{n\rightarrow \infty }\Phi _{u}^{\sigma ,F}(\rho ^{n}).
\end{align*}
Combining this with (\ref{II-1}), we have%
\begin{equation*}
\rho ^{p\sigma +\alpha }C_{H}^{p\sigma +\alpha }\liminf_{n\rightarrow \infty
}\Phi _{u}^{\sigma ,F}(\rho ^{n})\leq C\liminf_{n\rightarrow \infty
}\sum_{k=0}^{L}\rho ^{k(p\sigma -\alpha )}\mathcal{E}_{k+n}^{\sigma ,F}(u)+%
\frac{C^{\prime }\rho ^{(p\sigma -\alpha )(L+1)}}{1-\rho ^{p\sigma -\alpha }}%
\liminf_{n\rightarrow \infty }\Phi _{u}^{\sigma ,F}(\rho ^{n}).
\end{equation*}
For sufficiently large $L$ ($>n_{0}$), we have
\begin{equation*}
C^{\prime \prime }:=\frac{1}{C}\left( \rho ^{p\sigma +\alpha }C_{H}^{p\sigma
+\alpha }-\frac{C^{\prime }\rho ^{(p\sigma -\alpha )(L+1)}}{1-\rho ^{p\sigma
-\alpha }}\right) >0.
\end{equation*}
Then for this $L$, we obtain
\begin{align}
C^{\prime \prime }\liminf_{n\rightarrow \infty }\Phi _{u}^{\sigma ,F}(\rho
^{n})& \leq \liminf_{n\rightarrow \infty }\sum_{k=0}^{L}\rho ^{k(p\sigma
-\alpha )}\mathcal{E}_{k+n}^{\sigma ,F}(u)  \notag \\
& =\liminf_{n\rightarrow \infty }\sum_{k=0}^{L}\rho ^{k(p\sigma -\alpha
)}\rho ^{-(k+n)(p\sigma -\alpha )}E_{k+n}^{(p),F}(u)  \notag \\
& =\liminf_{n\rightarrow \infty }\sum_{k=0}^{L}\rho ^{-n(p\sigma -\alpha
)}E_{k+n}^{(p),F}(u).  \label{E_kn}
\end{align}

Using \eqref{E-n}, there exists $c>0$ such that for all positive integer $n$
and $u\in B_{p,\infty }^{\sigma ,F}$,%
\begin{equation*}
E_{n}^{(p),F}(u)\leq cE_{n+1}^{(p),F}(u).
\end{equation*}
Applying the above inequality to \eqref{E_kn}, we have
\begin{align*}
C^{\prime \prime }\liminf_{n\rightarrow \infty }\Phi _{u}^{\sigma ,F}(\rho
^{n})& \leq \liminf_{n\rightarrow \infty }C_{L}\ \rho ^{-(n+L)(p\sigma
-\alpha )}E_{n+L}^{(p),F}(u) \\
& =C_{L}\liminf_{n\rightarrow \infty }\mathcal{E}_{n+L}^{(p),F}(u)=C_{L}%
\liminf_{n\rightarrow \infty }\mathcal{E}_{n}^{(p),F}(u),
\end{align*}%
where $C_{L}=\frac{1-c^{L}}{1-c}\rho ^{L(p\sigma -\alpha )}$ does not depend
on $n$, which gives \eqref{limphi}.
\end{proof}

\begin{proof}[Proof of Theorem \protect\ref{thm5}]
Let $\sigma >\alpha /p$. By \eqref{eq6.2}, \eqref{limphi2} and (\ref{ne}), $(\widetilde{VE})_{\sigma}$ implies $(\widetilde{NE})_{\sigma}$ since
\begin{equation}
\limsup_{r\rightarrow 0}\Phi _{u}^{\sigma ,F}(r) \leq C\limsup_{n\rightarrow
\infty }\mathcal{E}_{n}^{\sigma ,F}(u)\leq C^{\prime }\liminf_{n\rightarrow
\infty }\mathcal{E}_{n}^{\sigma ,F}(u)\leq C^{\prime \prime
}\liminf_{r\rightarrow 0}\Phi _{u}^{\sigma ,F}(r).  \label{ve_ne1}
\end{equation}%
By \eqref{eq6.2}, (\ref{ne}) and \eqref{limphi}, $(\widetilde{NE})_{\sigma}$ implies $(\widetilde{VE})_{\sigma}$ since
\begin{equation}
\limsup_{n\rightarrow \infty }\mathcal{E}_{n}^{\sigma ,F}(u)\leq
C_{1}\limsup_{r\rightarrow 0}\Phi _{u}^{\sigma ,F}(r)\leq C_{1}^{\prime
}\liminf_{r\rightarrow 0}\Phi _{u}^{\sigma ,F}(r)\leq C_{1}^{\prime \prime
}\liminf_{n\rightarrow \infty }\mathcal{E}_{n}^{\sigma ,F}(u).
\label{ne_ve1}
\end{equation}%
Thus, by \eqref{ve_ne1} and \eqref{ne_ve1}, we show Theorem \ref{thm5} (i).

By \eqref{eq6.2} and \eqref{limphi2} adjoint with $(VE)_{\sigma}\Rightarrow (%
\widetilde{VE})_{\sigma}$, we immediately have from (\ref{ns}), (\ref{ne}) that
\begin{equation}
\sup_{r\in (0,C_H)}\Phi _{u}^{\sigma ,F}(r)\leq C_{2}\sup_{n\geq 0 }\mathcal{%
E}_{n}^{\sigma ,F}(u)\leq C_{2}^{\prime }\liminf_{n\rightarrow \infty }%
\mathcal{E}_{n}^{\sigma ,F}(u)\leq C_{2}^{\prime \prime
}\liminf_{r\rightarrow 0}\Phi _{u}^{\sigma ,F}(r).  \label{ve_ne2}
\end{equation}%
Similarly, by \eqref{eq6.2} and \eqref{limphi} adjoint with the fact $%
(NE)_{\sigma}\Rightarrow (\widetilde{NE})_{\sigma}$, we immediately have
\begin{equation}
\sup_{n\geq 0}\mathcal{E}_{n}^{\sigma ,F}(u)\leq C_{3}\sup_{n\geq 0}\Phi
_{u}^{\sigma ,F}(C_H\rho ^{n})\leq C_{3}^{\prime }\liminf_{n\rightarrow
\infty }\Phi _{u}^{\sigma ,F}(\rho ^{n})\leq C_{3}^{\prime \prime
}\liminf_{n\rightarrow \infty }\mathcal{E}_{n}^{\sigma ,F}(u).
\label{ne_ve2}
\end{equation}%
Thus, by \eqref{ve_ne2} and \eqref{ne_ve2}, we show Theorem \ref{thm5} (ii).
\end{proof}

\begin{remark}
The definition of discrete energies in \eqref{p_energy} uses point-wise
values of functions as opposed to cell-averaged values in \cite{Kigami2022penergy,KS92}. When $p\sigma >\alpha$, the domain of p-energy $\mathcal{E}_{p,\infty}^{\sigma,F}$ is a subset of $C(K^F)$, and both definitions are equivalent for a connected homogeneous p.c.f. self-similar
set by modifying the arguments in \cite[Lemma 3.1]{HK06} and \cite[Section 5]{Shimuzu.2022}. 
\end{remark}

\subsection{Consequences of \protect$(VE)_{\protect\sigma}$: BBM type characterization and
Gagliardo-Nirenberg inequality}

\label{subsec5.4} Let us state the following embedding inequality given by
Baudoin, using our notions.

\begin{lemma}
\label{GNTM}(\cite[Theorem 4.3]{BaudoinLecture2022}) Suppose that property
$(NE)_{\sigma_{p}^{\#}}$ holds with $R_0=\infty$ for the metric measure space $(M,d,\mu )$, and
there exists $R>0$ such that $\inf_{x\in M}\mu (B(x,R))>0$. When $p\sigma
_{p}^{\#}\neq \alpha_1 $, where $\alpha_1$ is from (\ref{VD2}), let $q=\frac{%
p\alpha_1 }{\alpha_1 -p\sigma_{p}^{\#}}$. For $r,s\in (0,\infty ]$, $\theta
\in (0,1]$ satisfying
\begin{equation*}
\frac{1}{r}=\frac{\theta }{q}+\frac{1-\theta }{s},
\end{equation*}%
there exists a constant $C>0$ such that for any $f\in B_{p,\infty }^{\sigma
_{p}^{\#}}$,
\begin{equation}
\Vert f\Vert _{r}\leq C(\Vert f\Vert _{p}+[f]_{B_{p,\infty }^{\sigma
_{p}^{\#}}})^{\theta }\Vert f\Vert _{s}^{1-\theta }.  \label{GNI}
\end{equation}
\end{lemma}

By Theorem \ref{thm4} and Theorem \ref{thm5}, we have the following
corollary.

\begin{corollary}
\label{corol1.8}Suppose that $K^{F}$ is a fractal glue-up that admits a heat
kernel satisfying \eqref{hk_F}, where $K$ is a connected homogeneous p.c.f.
self-similar set. Let $R_0=C_H$. Then properties $(\widetilde{KE})_{\sigma},
(\widetilde{NE})_{\sigma},(\widetilde{VE})_{\sigma}$ are equivalent with the same $%
\sigma>\alpha /p$, and properties $(KE)_{\sigma_{p}^{\#}}$,$(NE)_{\sigma_{p}^{\#}}$,$(VE)_{\sigma_{p}^{\#}}$ are equivalent when $\sigma_{p}^{\#}>\alpha /p$ on $K^{F}$.
\end{corollary}

The following theorem can be regarded as a summary of our paper when the underlying space is a fractal.

\begin{theorem}
\label{corol1.9}Let $K$ be a connected
homogeneous p.c.f. self-similar set with property (E), and let $K^{F}$ be a fractal glue-up. Then $K^{F}$
satisfies $(NE)_{\sigma_{p}^{\#}(K)}$ with $R_{0}=C_H$ and $B_{p,\infty }^{\sigma
_{p}^{\#}(K^{F})}(K^{F}):=B_{p,\infty }^{\sigma _{p}^{\#}}(K^{F})$ contains
non-constant functions. Also, BBM type characterization (\ref{NE_conv})
holds for $K^{F}$ with $R_{0}=C_H$, that is, for $R_{0}=C_H$, there exists a
positive constant $C$ such that for all $u\in B_{p,\infty }^{\sigma
_{p}^{\#}}(K^{F})$,
\begin{equation}
C^{-1}[u]_{{B}_{p,\infty }^{\sigma _{p}^{\#}}(K^{F})}^{p}\leq
\liminf_{\sigma \uparrow \sigma _{p}^{\#}(K^{F})}\left( \sigma
_{p}^{\#}(K^{F})-\sigma \right) [u]_{{B}_{p,p}^{\sigma }(K^{F})}^{p}\leq
\limsup_{\sigma \uparrow \sigma _{p}^{\#}(K^{F})}\left( \sigma
_{p}^{\#}(K^{F})-\sigma \right) [u]_{{B}_{p,p}^{\sigma }(K^{F})}^{p}\leq
C[u]_{{B}_{p,\infty }^{\sigma _{p}^{\#}}(K^{F})}^{p}.  \label{SS-2}
\end{equation}%
Furthermore, if $K$ is a nested fractal and $K^{F}=K_{\infty }$, where $%
K_{\infty }$ is defined in Example \ref{ex1}, then $(NE)_{\sigma_{p}^{\#}(K)}$ and $(KE)_{\sigma_{p}^{\#}(K)}$ hold for $K_{\infty }$ with $R_{0}=\infty $, and the Gagliardo-Nirenberg inequality (\ref{GNI}) holds.
\end{theorem}
\begin{proof}
By Proposition \ref{prop:critical}, $K^{F}$ satisfies $(VE)_{\sigma _{p}^{\#}(K)}$ and the critical
domain $B_{p,\infty }^{\sigma _{p}^{\#}}(K^{F})$ contains non-constant
functions, where $\sigma _{p}^{\#}=\sigma _{p}^{\#}(K)=\sigma
_{p}^{\#}(K^{F})$. Thus $(NE)_{\sigma _{p}^{\#}}$ holds true for $K^{F}$ by Theorem \ref{thm5}(ii)
with $R_{0}=C_H$, and \eqref{SS-2} holds by Theorem \ref{thm4.2}.

Since a nested fractal $K$ satisfies property (E) by Lemma \ref{lemma:E},
these inclusions hold for nested fractals. Next, we state how to upgrade $%
R_0=C_H$ to $R_0=\infty $ for $(NE)_{\sigma_{p}^{\#}}$ on $K_{\infty }$.

For $n\in \mathbb{N}$, define the scaling function $g_{n}:K^{F}%
\rightarrow \rho ^{-n}K^{F}$ by $g_{n}(x)=\rho ^{-n}x.$ Due to the global
self-similarity of $K_{\infty }$, we have
\begin{equation}
g_{n}(K_{\infty })=K_{\infty }.  \label{ss}
\end{equation}
For any $u\in B_{p,\infty }^{\sigma _{p}^{\#} ,F}$, clearly $u\circ g_{n}\in
B_{p,\infty }^{\sigma _{p}^{\#} ,F}$. By definition, for all $r>0$,
\begin{eqnarray}
\Phi _{u\circ g_{n}}^{\sigma _{p}^{\#} ,F}(\rho ^{n}r) &=&(\rho ^{n}r)^{-p\sigma _{p}^{\#}
-\alpha }\int_{K_{\infty }}\int_{B(x,\rho
^{n}r)}|u(g_{n}(x))-u(g_{n}(y))|d\mu ^{F}(y)d\mu ^{F}(x)  \notag \\
&=&(\rho ^{n}r)^{-p\sigma _{p}^{\#} -\alpha }\int_{K_{\infty
}}\int_{B(g_{n}(x),r)}|u(g_{n}(x))-u(y)|d(\mu ^{F}\circ g_{n}^{-1})(y)d\mu
^{F}(x)  \notag \\
&=&(\rho ^{n}r)^{-p\sigma _{p}^{\#} -\alpha }\int_{g_{n}(K_{\infty
})}\int_{B(x,r)}|u(x)-u(y)|d\mu ^{F}(\rho ^{n}y)d\mu ^{F}(\rho ^{n}x)  \notag
\\
&\asymp &(\rho ^{n}r)^{-p\sigma _{p}^{\#} -\alpha }\rho ^{2n\alpha }\int_{K_{\infty
}}\int_{B(x,r)}|u(x)-u(y)|d\mu ^{F}(y)d\mu ^{F}(x)\   \notag \\
&=&\rho ^{-n(p\sigma _{p}^{\#} -\alpha )}\Phi _{u}^{\sigma _{p}^{\#} ,F}(r),  \label{sp}
\end{eqnarray}%
where we use (\ref{ss}) and that $\mu ^{F}$ is $\alpha $-regular in the
forth line. For any $u\in B_{p,\infty }^{\sigma _{p}^{\#} ,F}$, there exists $n<\infty
$ such that
\begin{equation*}
\sup_{r\in (0,\infty )}\Phi _{u}^{\sigma _{p}^{\#} ,F}(r)\leq 2\sup_{r\in (0,C_{H}\rho
^{-n})}\Phi _{u}^{\sigma _{p}^{\#} ,F}(r).
\end{equation*}%
Combined with (\ref{sp}),
\begin{equation}
\sup_{r\in (0,\infty )}\Phi _{u}^{\sigma _{p}^{\#} ,F}(r)\leq C\sup_{r^{\prime }\in
(0,C_{H})}\Phi _{u\circ g_{n}}^{\sigma _{p}^{\#} ,F}(r^{\prime })\rho ^{n(p\sigma _{p}^{\#}
-\alpha )}.  \label{sp-1}
\end{equation}%
As $(NE)_{\sigma _{p}^{\#}}$ holds for $R_{0}=C_{H}$,
\begin{eqnarray*}
\sup_{r^{\prime }\in (0,C_{H})}\Phi _{u\circ g_{n}}^{\sigma
_{p}^{\#},F}(r^{\prime })\rho ^{n(p\sigma _{p}^{\#}-\alpha )} &\leq
&C\liminf_{r^{\prime }\rightarrow 0}\Phi _{u\circ g_{n}}^{\sigma
_{p}^{\#},F}(r^{\prime })\rho ^{n(p\sigma _{p}^{\#}-\alpha )} \\
&=&C\liminf_{r^{\prime }\rightarrow 0}\Phi _{u}^{\sigma _{p}^{\#},F}(\rho
^{-n}r^{\prime })\text{ \ (by (\ref{sp}))} \\
&=&C\liminf_{r\rightarrow 0}\Phi _{u}^{\sigma _{p}^{\#},F}(r),\text{ }
\end{eqnarray*}
where the constant $C$ does not depend on $u$ and $n$. Thus $(NE)_{\sigma _{p}^{\#}}$ also holds with $R_{0}=\infty $ by (\ref{sp-1}). Since $\mu ^{F}$ is $\alpha $-regular, we know by Lemmas \ref{lemma:E} and \ref{GNTM} that (\ref{GNI}) holds true. Finally, for a nested fractal $K$, it is
known by \cite{Kumagai.1993.PTaRF205} (see also \cite[Theorem 1.1]
{FitzsimmonsmHamblyKumagai.1994.CMP595}) that $K_{\infty }$ admits a heat
kernel satisfying \eqref{hk_F}, thus $(KE)_{\sigma _{p}^{\#}}$ is verified by Theorem \ref{thm4}.
The proof is complete.
\end{proof}

When $p=2$, we use heat kernel estimates to verify weak-monotonicity
properties in the following corollary, which is analytical (compared with
the algebraic resistance-estimate arguments for $(VE)_{\sigma}$).

\begin{corollary}
\label{jjj} When $p=2$, if a fractal glue-up $K^{F}$ admits a heat kernel
with sub-Gaussian estimates \eqref{hk_F}, then $(VE)_{\beta
^{\ast }/2}$ and $(NE)_{\beta
^{\ast }/2}$ automatically hold
for $K^F $.
\end{corollary}

\begin{proof}
By Remark \ref{rk1}, property $(KE)_{\beta
^{\ast }/2}$ automatically holds for $K^F$ when $p=2$.
The conclusion then follows from Corollary \ref{corol1.8}.
\end{proof}

\begin{figure}[tbph]
	\centering
		\begin{tikzpicture}[scale=0.5]
			\draw[fill=black] (0,0)--(1,0)--(0.5,0.866)--cycle;
			\draw[fill=black] (0.5,0.866)--(1+0.5,0.866)--(0.5+0.5,0.866*2)--cycle;
			\draw[fill=black] (0.5+10,0.866)--(1+0.5+10,0.866)--(0.5+0.5+10,0.866*2)--cycle;
			\draw[fill=black] (0.5*2,0.866*2)--(1+0.5*2,0.866*2)--(0.5+0.5*2,0.866*3)--cycle;
			\draw[fill=black] (0.5*2+9,0.866*2)--(1+0.5*2+9,0.866*2)--(0.5+0.5*2+9,0.866*3)--cycle;
			\draw[fill=black] (0.5*3,0.866*3)--(1+0.5*3,0.866*3)--(0.5+0.5*3,0.866*4)--cycle;
			\draw[fill=black] (0.5*3+8,0.866*3)--(1+0.5*3+8,0.866*3)--(0.5+0.5*3+8,0.866*4)--cycle;
			\draw[fill=black] (0.5*4,0.866*4)--(1+0.5*4,0.866*4)--(0.5+0.5*4,0.866*5)--cycle;
			\draw[fill=black] (0.5*4+7,0.866*4)--(1+0.5*4+7,0.866*4)--(0.5+0.5*4+7,0.866*5)--cycle;
			\draw[fill=black] (0.5*5,0.866*5)--(1+0.5*5,0.866*5)--(0.5+0.5*5,0.866*6)--cycle;
			\draw[fill=black] (0.5*5+6,0.866*5)--(1+0.5*5+6,0.866*5)--(0.5+0.5*5+6,0.866*6)--cycle;
			\draw[fill=black] (0.5*6,0.866*6)--(1+0.5*6,0.866*6)--(0.5+0.5*6,0.866*7)--cycle;
			\draw[fill=black] (0.5*6+5,0.866*6)--(1+0.5*6+5,0.866*6)--(0.5+0.5*6+5,0.866*7)--cycle;
			\draw[fill=black] (0.5*7,0.866*7)--(1+0.5*7,0.866*7)--(0.5+0.5*7,0.866*8)--cycle;
			\draw[fill=black] (0.5*7+4,0.866*7)--(1+0.5*7+4,0.866*7)--(0.5+0.5*7+4,0.866*8)--cycle;
			\draw[fill=black] (0.5*8,0.866*8)--(1+0.5*8,0.866*8)--(0.5+0.5*8,0.866*9)--cycle;
			\draw[fill=black] (0.5*8+3,0.866*8)--(1+0.5*8+3,0.866*8)--(0.5+0.5*8+3,0.866*9)--cycle;
			\draw[fill=black] (0.5*9,0.866*9)--(1+0.5*9,0.866*9)--(0.5+0.5*9,0.866*10)--cycle;
			\draw[fill=black] (0.5*9+2,0.866*9)--(1+0.5*9+2,0.866*9)--(0.5+0.5*9+2,0.866*10)--cycle;
			\draw[fill=black] (0.5*10,0.866*10)--(1+0.5*10,0.866*10)--(0.5+0.5*10,0.866*11)--cycle;
			\draw[fill=black] (0.5*10+1,0.866*10)--(1+0.5*10+1,0.866*10)--(0.5+0.5*10+1,0.866*11)--cycle;
			\draw[fill=black] (0.5*11,0.866*11)--(1+0.5*11,0.866*11)--(0.5+0.5*11,0.866*12)--cycle;
			\draw[fill=black] (0+1,0)--(1+1,0)--(0.5+1,0.866)--cycle;
			\draw[fill=black] (0+2,0)--(1+2,0)--(0.5+2,0.866)--cycle;
			\draw[fill=black] (0+3,0)--(1+3,0)--(0.5+3,0.866)--cycle;
			\draw[fill=black] (0+8,0)--(1+8,0)--(0.5+8,0.866)--cycle;
			\draw[fill=black] (0+9,0)--(1+9,0)--(0.5+9,0.866)--cycle;
			\draw[fill=black] (0+10,0)--(1+10,0)--(0.5+10,0.866)--cycle;
			\draw[fill=black] (0+11,0)--(1+11,0)--(0.5+11,0.866)--cycle;
			\draw[fill=black] (0+3+0.5,0+0.866)--(1+3+0.5,0+0.866)--(0.5+3+0.5,0.866+0.866)--cycle;
			\draw[fill=black] (0+3+0.5+1,0+0.866)--(1+3+0.5+1,0+0.866)--(0.5+3+0.5+1,0.866+0.866)--cycle;
			\draw[fill=black] (0+3+0.5+2,0+0.866)--(1+3+0.5+2,0+0.866)--(0.5+3+0.5+2,0.866+0.866)--cycle;
			\draw[fill=black] (0+3+0.5+3,0+0.866)--(1+3+0.5+3,0+0.866)--(0.5+3+0.5+3,0.866+0.866)--cycle;
			\draw[fill=black] (0+3+0.5+4,0+0.866)--(1+3+0.5+4,0+0.866)--(0.5+3+0.5+4,0.866+0.866)--cycle;
		\end{tikzpicture}
		\caption{ A u.f.r. fractal.}
		\label{fig1}
	\end{figure}

We use an example based on \cite[Section 4]{HamblyKumagai.2004.PSPM} to show that Corollaries \ref{corol1.8} and \ref{jjj} are also meaningful for some non-nested fractal blow-ups. That is, there exists a connected homogeneous p.c.f.
self-similar set which is not a nested fractal, such that, its fractal blow-up admits a heat kernel satisfying \eqref{hk_F}.
\begin{example}
\label{exp2}
Let $S$ be a unit equilateral triangle in $\mathbb{R}^2$. Its vertices are located at $q_0=(0, 0)$, $q_1=(1, 0)$, and $q_2=(1/2, \sqrt{3}/2)$. Divide $S$ into a mesh of sub-triangles of side length $1/12$, and pick $34$ sub-triangles as shown  in Figure \ref{fig1}.
Let $\{\phi_i\}_{i=1}^{34}$ be the IFS on $\mathbb{R}^2$ that maps $S$ to these $34$ sub-triangles with attractor $K$. Then $K=\bigcup_{i=1}^{34}\phi_i(K)$ is a connected homogeneous p.c.f. self-similar set, which is also a {\em uniformly finitely ramified fractal} (u.f.r. fractal for short). However, $K$ is not a nested fractal. Its Hausdorff dimension is \( \alpha = \dim_H(K) = \frac{\log 34}{\log 12} \), and the \( \alpha \)-dimensional Hausdorff measure exists since the open set condition is satisfied.
 Define $V_0=\{q_{k}\}_{k=0}^2$ and $V_{n+1}=\bigcup_{i=1}^{34}\phi_i(V_n)$ as in \eqref{e.V}.
Let $K_\infty$ be the fractal blow-up of $K$ (defined in Example \ref{ex1} as \( \bigcup_{l=1}^\infty K_l \), where \( K_l = 12^l K \)). Let $d_\infty$ be the shortest path length between any two points in \( K_\infty \). By the global self-similarity of blow-ups, \( d_\infty \) is bi-Lipschitz equivalent to the Euclidean metric \( d_{\mathbb{R}^2}|_{K_\infty} \).

Define $V=\bigcup_{i=0}^\infty 12^n V_n$ and $E := \{(x, y) \in V \times V : d(x, y) = 1\}$. Consider the graph $(V, E)$ with weights $\mu$ satisfying the $p_0$-condition in \cite[formula (1.5)]{BB04}. By substituting each edge in $E$ with an isometric copy of the line segment $[0, 1]$ (referred to as a \textit{cable}), and connecting them in a natural way at the vertices, we obtain an unbounded connected closed set $X_0 \subset \mathbb{R}^2$, known as the \textit{cable system} associated with $(V, E)$. The metric $d_0$ is defined on $X_0\times X_0$ by employing the Euclidean distance along each cable, and this is extended to a metric over $X_0\times X_0$ that is consistent with the graph distance on $V$.
We define $\Phi_0(r):=r \vee r^{\alpha}$ and $\Psi_0(r):=r^2 \vee r^{d_w}$. Denote by $m_0$ the Hausdorff measure on $(X_0,d_0)$, then we have $m_0(B(x,r))\asymp\Phi_0(r)$ by \cite[Lemma 2.1]{BB04}, where $B(x,r):=\{z\in X_0: d_0(x,z)<r\}$. Define a regular Dirichlet form $(\mathcal{E}_0,\mathcal{F}_0)$ corresponding to the cable process on $(X_0,d_0,m_0)$ as in \cite[Section 2]{BB04}, which is conservative and strongly local. By \cite[Corollary 4.14]{HamblyKumagai.2004.PSPM}, there exist constants $c_1$, $c_2$, $c_3$ and $c_4>0$ such that for all $x,y\in V$, $k\geq d_0(x,y)$
\begin{equation}\label{HK}
 \begin{aligned}
  p_k(x, y)  & \leq c_{1} k^{-\alpha/d_w} \exp\left(-c_{2} \left( \frac{d_0(x, y)^{d_w}}{k} \right)^{1/(d_w-1)}\right), \\
  p_k(x, y) + p_{k+1}(x, y)  & \geq c_{3} k^{-\alpha/d_w} \exp\left(-c_{4} \left( \frac{d_0(x, y)^{d_w}}{k} \right)^{1/(d_w-1)}\right).
\end{aligned}
\end{equation}
Therefore, \eqref{HK} implies \cite[formulas (1.12) and (1.13)]{BB04}. By \cite[Theorem 1.2, Corollary 2.5 and Lemma 2.6]{BB04}, the elliptic Harnack inequality $(\mathrm{EHI})$ and the mean exit time estimate $(\mathrm{E}_{\Psi_0})$ (see \cite[Definition 3.10]{GrigoryanTelcs.2012.AoP1212}) hold for the cable process. Thus the metric measure Dirichlet space $ (X_0, d_0, m_0, \mathcal{E}_0,\mathcal{F}_0)$ admits a heat kernel $\{p_t\}_{t>0}$ satisfying HKE($\Psi_0$) by \cite[Theorem 5.15]{GrigoryanTelcs.2012.AoP1212}, where the condition $HKE(\Psi)$ (for the metric measure space $(X,d,m)$)
is defined in \cite[Definition 1.1]{Chen24}:
there exist $c_{1},c_{2},c_{3}, \delta>0$ and a heat kernel $\{p_t\}_{t>0}$ such that for any $t>0$,
	\begin{align}
		p_{t}(x,y) &\leq \frac{c_{1}}{V(x,\Psi^{-1}(t))} \exp\left(-c_{2}t\Phi\left(c_{3}\frac{d(x, y)}{t}\right)\right)
		\qquad \mbox{for $m$-a.e.\ $x,y \in X$}\\
		\text{and }\ p_{t}(x,y) &\geq \frac{c_{1}^{-1}}{V(x,\Psi^{-1}(t))}
		\qquad \mbox{for $m$-a.e.\ $x,y\in X$ with $d(x,y) \leq \delta \Psi^{-1}(t)$},\label{d:HKE}
	\end{align}
	where $V(x,r):=m(\{z\in X: d(z,x)<r\})$ and \begin{equation*}
		\Phi(s):=\Phi_{\Psi}(s):=\sup_{r>0}\left({\frac{s}{r}-\frac{1}{\Psi(r)}}\right).
	\end{equation*}

 Next, we construct a sequence of metric measure spaces $ (X_n, d_n, m_n, \mathcal{E}_n,\mathcal{F}_n)$ ($n\in \mathbb{N}$), by
\begin{align*}
&X_n = X_0,~ d_n = 12^{-n} d_0~ \text{and}~~ m_n:=12^{-\alpha n} m_0,\\
&\mathcal{E}_n=12^{(d_w-\alpha)n} \mathcal{E}_0,~~ u,v\in \mathcal{F}_n:=\mathcal{F}_0.
\end{align*}

To apply the pointed Gromov-Hausdorff convergence results in \cite{Chen24}, we need to verify that the spaces $ (X_n, d_n, m_n, q_0) $ satisfy Conditions (A1)-(A6) in \cite[Theorem 1.2]{Chen24}. 
For the scaled sequence $ (X_n, d_n, m_n) $, the heat kernel $ p^{(n)}_t( x, y) $ is related to the original heat kernel $\{p_t\}_{t>0}$ by
\[ p^{(n)}_t(x, y) = 12^{\alpha n} \cdot p_{12^{d_w n} t}( x, y). \]
This scaling preserves the sub-Gaussian form, so $ (X_n, d_n, m_n,\mathcal{E}_n,\mathcal{F}_n)$ satisfies $ HKE(\Psi_n) $ with constants independent of \( n \) and $ \Psi_n(r) := 12^{-d_w n} \Psi_0(12^n r)$. Therefore, condition (A1) in \cite[Theorem 1.2]{Chen24} holds. It is easy to verify that conditions (A2-A5) in \cite[Theorem 1.2]{Chen24} also hold.

It remains to verify condition (A6) in \cite[Theorem 1.2]{Chen24}. Define
\[\widetilde{X}_{n}:=12^{-n}X_{0}=\{12^{-n}x:x\in X_{0}\},\ n\geq 1,\]
and
\begin{equation*}
\widetilde{d}_{n}(x,y):=12^{-n}d_{0}(12^{n}x,12^{n}y),\ \forall x,y\in \widetilde{X}_{n}.
\end{equation*}
This scaling satisfies the consistency relation with \( d_\infty \):
\begin{equation*}
  d_\infty(x, y) = \widetilde{d}_{n}(x, y) \quad \forall x, y \in \widetilde{X}_{n}.
\end{equation*}
Note that $\bigcup_{n=1}^\infty X_n=K_\infty$ and the map $x\mapsto 12^{-n}x$ is a bijective isometry from $(X_n, d_n) $ to $(\widetilde{X}_{n},\widetilde{d}_{n})$. Clearly, $(\widetilde{X}_{n},\widetilde{d}_{n})$ is a proper length space, since the intrinsic metric \( \widetilde{d}_{n}\) coincides with the shortest path length between points, so \( (\widetilde{X}_{n},\widetilde{d}_{n}, q_0) \) is an {\em eventually proper sequence of pointed length spaces}. Similar to \cite[Example 8.1(6) on Page 41]{Chen24}, we know that $(\widetilde{X}_{n},\widetilde{d}_{n},q_0)$ pointed Gromov-Hausdorff converge to $(K_{\infty},d_{\infty},q_0)$.

Finally, applying \cite[Theorem 1.2 and Remark 1.3 (i)]{Chen24}, we conclude that there is a measure $m_{\infty}$ on $(K_{\infty},d_{\infty})$ satisfies $m_{\infty}(B_{\infty}(x,r))\asymp r^{\alpha}$, and there is a regular strongly local Dirichlet form on $L^{2}(K_{\infty},d_{\infty},m_{\infty})$ that satisfies \eqref{hk_F}.
\end{example}

\begin{acknowledgement}
The authors thank Jiaxin Hu, and Hua Qiu for helpful suggestions. The authors thank Aobo Chen and Qingsong Gu for helpful discussions on Example \ref{exp2}. The authors appreciate the suggestions and revisions made
by anonymous referees. Jin Gao was supported by National Natural Science
Foundation of China (Grant. No. 12271282), and Zhejiang Provincial Natural Science
Foundation of China (Grant. No. LQN25A010019). Zhenyu Yu was supported by the Natural
Science Foundation of Hunan Province, China (Grant. No. 2025JJ60039) and National University of Defense Technology (Grant. No. ZK25-05).
\end{acknowledgement}

\end{document}